\documentclass[12pt]{article}
\usepackage{amscd}
\usepackage{latexsym,amsmath,amssymb,amsbsy, amsthm}
\usepackage[english]{babel}
\usepackage[square,numbers]{natbib}
\usepackage{subcaption}

\usepackage{hyperref}
\usepackage{lscape,fancyhdr,fancybox}
\usepackage{graphicx,epsfig, xy}
\usepackage{color}
\usepackage[hmarginratio=1:2, vmarginratio =5:5,
textheight=22cm,bindingoffset=1.5cm, textwidth=14.6cm]{geometry}
\usepackage{tikz}

\newcommand{\comment}[1]{}

\numberwithin{equation}{section}

\newtheorem{remark}{Remark}[section]
\newtheorem{theorem}{Theorem}[section]

\newtheorem{lemma}{Lemma}[section]

\theoremstyle{definition}
\newtheorem{definition}{Definition}[section]

\newtheorem{result}{Result}[section]

\usepackage{amsfonts}

\usepackage{txfonts}

\DeclareMathOperator{\sgn}{sgn} 
 
\DeclareMathOperator{\Tr}{Tr}

\newcommand{\beq}{\begin{eqnarray}}
\newcommand{\eeq}{\end{eqnarray}}
\newcommand{\ben}{\begin{eqnarray*}}
\newcommand{\een}{\end{eqnarray*}}

\allowdisplaybreaks

\begin{document}
\def\shorttitle{Covariance matrix}
 \def\shortauthors{A. Bose, P. Sen}
\title{\textbf{\Large \sc
\Large{$XX^T$ matrices with independent entries} 
}}\small

\author{
 \parbox[t]{0.20
\textwidth}{{\sc Arup Bose}
 \thanks{Statistics and Mathematics Unit, Indian Statistical Institute, 203 B.T. Road, Kolkata 700108, India. email: bosearu@gmail.com. }}
\parbox[t]{0.25\textwidth}{{\sc Priyanka Sen}
 \thanks{Statistics and Mathematics  Unit, Indian Statistical Institute, 203 B.T. Road, Kolkata 700108, INDIA. email: priyankasen0702@gmail.com}}
}

\date{\today}  
\maketitle

\begin{abstract}
Let $S=XX^T$ be the (unscaled) sample covariance matrix where $X$ is a real $p \times n$ matrix with independent entries. It is well known that if the entries of $X$ are independent and identically distributed (i.i.d.) with enough moments and $p/n \to y\neq 0$, then the limiting spectral distribution (LSD) of $\frac{1}{n}S$ converges to a Mar$\check{\text{c}}$enko-Pastur law. Several extensions of this result are also known. We prove a general result on the existence of the LSD of $S$ in probability or almost surely, and in particular, many of the above results follow as special cases. At the same time several new LSD results also follow from our general result.

The moments of the LSD are quite involved but can be described via a set of partitions. Unlike in the i.i.d. entries case, these partitions are not necessarily non-crossing, but are related to the special symmetric partitions which are known to appear in the LSD of (generalised) Wigner matrices with independent entries.

We also investigate the existence of the LSD of $S_{A}=AA^T$ when $A$ is the $p\times n$ symmetric or the asymmetric version of any of the following four random matrices: reverse circulant, circulant, Toeplitz and Hankel. The LSD of $\frac{1}{n}S_{A}$ for the above four cases  have been studied in \citep{bose2010limiting} when the entries are i.i.d. We show that under some general assumptions on the entries of $A$, the LSD of $S_{A}$ exists and this result generalises the existing results of \citep{bose2010limiting} significantly.  
\end{abstract}


\medskip

\section{Introduction}\label{introduction}
 Suppose $M_n$ is an $n\times n$ real symmetric random matrix (i.e., a matrix whose entries are random variables) with (real) eigenvalues $\lambda_1,\lambda_2,\ldots,\lambda_n$. Its \textit{empirical spectral measure} is the random probability measure:
\begin{align*}
\mu_{M_n} = \frac{1}{n} \sum_{i=1}^{n}  \delta_{\lambda_i},
\end{align*}
where $\delta_x$ is the Dirac measure at $x$. The random probability distribution function, $F^{M_n}$, known as the \textit{empirical spectral distribution} (ESD) of $M_n$ is given by 
$$F^{M_n}(x)= \frac{1}{n}\displaystyle \sum_{i=1}^n \boldsymbol{1} (\lambda_i\leq x).$$
The expectation of the above distribution function, denoted by $\mathbb{E}[F^{M_n}]$, is a non-random distribution function and is known as the 
\textit{expected empirical spectral distribution} (EESD). The corresponding probability measure will be denoted by $\mathbb{E}\mu_{M_n}$.
 The notions of convergence that are used in this article are: (i) the \textit{(weak) convergence of the EESD} $\mathbb{E}\mu_{M_n}$, and (ii) the \textit{(weak) convergence of the 
ESD} $\mu_{M_n}$  (either in \textit{probability} or \textit{almost surely (a.s.)}). The limits in (i) and (ii) are identical when the latter limits are non-random. In any case, any of these limits will be referred to as the \textit{limiting spectral distribution} (LSD) of $\{M_n\}$. 

Now let $X_p$ be a $p \times n$ matrix  with real independent entries $\{x_{ij,n}: 1\leq i \leq p, 1 \leq j \leq n\}$, where $p=p(n), p/n \rightarrow y\in (0,\infty)$. The matrix $S=X_pX_p^{T}$ will be called the \textit{Sample covariance matrix} (without scaling).
Note that the entries of $X_p$ are not necessarily identically distributed, and they do not necessarily have identical variances. We will also be interested in the matrix $S_{A}= A_pA_p^T$ where $A_p$ is any one of the $p \times n$ patterned matrices, namely  reverse circulant, circulant, Toeplitz and Hankel, with entries that are real and independent. 

We explore the existence of the LSD of $S$ and $S_A$ under suitable conditions on the entries of $X_p$ and $A_p$. 
The motivation to work on these problems, along with a brief discussion to relate our results with the models and results that already exist in  the literature are given below. Our two main theorems, namely Theorems \ref{res:XXt} and \ref{res:main}, are given in Section \ref{main results}. 
\vskip5pt

\noindent (a) The $S$ matrix is arguably one of the most important matrices in random matrix theory with varied applications in physics, statistics and other areas. There have been several works regarding its LSD.  When the entries of $X_p$ are i.i.d. with mean zero and  finite fourth moment, \citep{marchenko1967distribution} first established the LSD of $\frac{1}{n}S$ and this LSD has been named the Mar$\check{\text{c}}$enko-Pastur (MP) law. Subsequent works by \citep{grenander1977spectral}, \citep{wachter1978strong}, \citep{yin1986limiting}, \citep{belinschi2009spectral}, investigated the existence and properties of the LSD under varied assumptions on the entries. In these works the distribution of the entries of $X_p$ remain unaltered for every $n$. 

It is natural to ask what happens when the distribution of the entries depend on $n$ and/or the entries are not identically distributed. 
The convergence of the ESD of $\frac{1}{a_n^2}S$, when the entries of $X_p$ are i.i.d. with heavy tails and $a_n$ is a sequence of constants related to the tail probability of the entry distribution, was proved in \citep{belinschi2009spectral}. There, an appropriate truncation of the variables at levels that depend on $n$ was crucial in the arguments. Thus it becomes relevant to probe the case where the distribution of the entries of $X_p$ is allowed to depend on $n$, not just due to a scaling constant that depends on $n$ but where a genuine triangular sequence of entries is used. 

Such a model was already considered by Zakharevich \citep{zakharevich2006generalization} for the (symmetric) Wigner matrix. For any distribution $F$, let $\beta_k(F)$ be the $k$th moment of $F$.  
Consider a \textit{generalized Wigner matrix} $W_n$ whose entries are $\{x_{ij,n}; 1 \leq i \leq j \leq n\}$ (with $x_{ij,n}=x_{ji,n}$) with distribution $G_n$ for every fixed $n$. Assume that, 
\begin{equation}\label{zak}
\displaystyle{\lim_{n\to\infty} \frac{\beta_k(G_n)}{n^{k/2-1}\beta_2(G_n)^{k/2}}}= g_k,  \ \text{say,  exists for all} \ \ k\geq 1.
\end{equation}
 Then she proved that the ESD of $\frac{W_n}{\sqrt{n\mu_n(2)}}$ converges in probability to a distribution $\mu_{zak}$ that depends only on the sequence $\{g_{2k}\}$. LSD of Wigner matrices with general independent triangular array of entries were explored in \citep{bose2022random}. They found that a class of partitions, the \textit{special symmetric partitions}, play a crucial role in the moments of the LSD. 

Matrices whose entries satisfy conditions like \eqref{zak}, are referred to as \textit{matrices with exploding moments} and have been considered by several authors.  In particular, the $S$ matrix with exploding moments have been studied in Theorem 3.2 of \citep{Benaych-Georges2012},  and Proposition 3.1 of\citep{noiry2018spectral}. Moreover, formulae for the moments of the LSD have been provided using free probability theory and graph theory respectively.


We establish LSD results for the $S$ matrix (see Theorem \ref{res:XXt}) where the distribution of any entry is allowed  to be dependent not only on $n$ but also on its position in the matrix. We describe a formula for the moments of the LSD using certain partitions. We relate these moments not only to the ones that have appeared in \citep{Benaych-Georges2012} and \citep{noiry2018spectral} but also to the limiting moments in the (generalised) Wigner case (\citep{zakharevich2006generalization},\citep{bose2022random})--under our assumptions, only the class of special symmetric partitions  contribute to the moments.

In Section \ref{simulations}, we provide some simulations to show a glimpse of the various distributions that can appear as the LSD. In Section \ref{realtion to existing results}, we discuss how Theorem \ref{res:XXt} brings the various results such as \citep{marchenko1967distribution}, \citep{belinschi2009spectral},  \citep{Benaych-Georges2012}, \citep{noiry2018spectral}, \citep{Dykema2002DToperatorsAD} etc., under one umbrella, as well as generates some new results.
As a special case of Theorem \ref{res:XXt}, the ESD of $S$ with sparse entries converges a.s. (see Section \ref{subsubparse}), and we relate this LSD to the free Poisson and Poisson distributions. \textit{Matrices with variance profile} also come under our purview (see Section \ref{var-prof-smatrix}) and again, under suitable assumptions, the a.s. convergence of the ESD of $S$ holds. 
\\

\noindent (b) Let us now consider the matrix $S_{A}=A_pA_p^T$ where $A_p$  has one of the following patterns:
\begin{align*}
T^{(s)}= \begin{bmatrix}
x_0 & x_1 & x_2 & \cdots &  x_{n-1} \\
x_1 & x_0 & x_1 & \cdots &  x_{n-2}\\
x_2 & x_1 & x_0 & \cdots &  x_{n-3} \\
\vdots & \vdots & {\vdots} & \ddots  & \vdots \\
x_{p-1} & x_{p-2} & x_{p-3} & \cdots  & x_{|p-n|}
\end{bmatrix}, \ \ \ T=\begin{bmatrix}
x_0 & x_1 & x_2 & \cdots &  x_{n-1} \\
x_1 & x_0 & x_1 & \cdots &  x_{n-2}\\
x_2 & x_1 & x_0 & \cdots &  x_{n-3} \\
\vdots & \vdots & {\vdots} & \ddots  & \vdots \\
x_{p-1} & x_{p-2} & x_{n-3} & \cdots  & x_{p-n}
\end{bmatrix},
\end{align*}

\begin{align*}
H^{(s)}= \begin{bmatrix}
x_2 & x_3 & x_4 & \cdots &  x_{n+1} \\
x_3 & x_4 & x_5 & \cdots &  x_{n+2}\\
x_4 & x_5 & x_6 & \cdots &  x_{n+3} \\
\vdots & \vdots & {\vdots} & \ddots  & \vdots \\
x_{p+1} & x_{p+2} & x_{n+3} & \cdots  & x_{p+n}
\end{bmatrix}, \ \ \ H=\begin{bmatrix}
x_2 & x_{-3} & x_{-4} & \cdots &  x_{-(n+1)} \\
x_3 & x_4 & x_{-5} & \cdots &  x_{-(n+2)}\\
x_4 & x_5 & x_6 & \cdots &  x_{-(n+3)} \\
\vdots & \vdots & {\vdots} & \ddots  & \vdots \\
x_{p+1} & x_{p+2} & x_{p+3} & \cdots  & x_{-(p+n)}
\end{bmatrix},
\end{align*}

\begin{align*}
R^{(s)}= \begin{bmatrix}
x_0 & x_1 & x_2 & \cdots   & x_{n-1} \\
x_1 & x_2 & x_3 & \cdots   & x_{0}\\
x_2 & x_3 & x_4 & \cdots   & x_{1} \\
\vdots & \vdots & {\vdots} & \ddots & \vdots \\
x_{(p-1)\text{mod }n} &  \cdots  &  & \cdots  & x_{(p-2)\text{mod }n}
\end{bmatrix}, \ \ \ R=\begin{bmatrix}
x_0 & x_1 & x_2 & \cdots   & x_{n-1} \\
x_{-1} & x_2 & x_3 & \cdots   & x_{0}\\
x_{-2} & x_{-3} & x_4 & \cdots   & x_{1} \\
\vdots & \vdots & {\vdots} & \ddots & \vdots \\
x_{-(p-1)\text{mod }n} &  \cdots  &  & \cdots  & x_{(p-2)\text{mod }n}
\end{bmatrix},
\end{align*}
%

\begin{align*}
C^{(s)}= \begin{bmatrix}
x_0 & x_1 & x_2 & \cdots  & x_1 \\
x_1 & x_0 & x_1 & \cdots  & x_{2}\\
x_{2} & x_1 & x_0 & \cdots & x_{3}\\
\vdots & \vdots & {\vdots} & \ddots  & \vdots \\
x_{n/2-|n/2-|p-1||} & \cdots & \cdots & \cdots  & x_{n/2-|n/2-|p-n||}
\end{bmatrix}, 
\end{align*}

\begin{align*}
C=\begin{bmatrix}
x_0 & x_1 & x_2 & \cdots  & x_{n-1} \\
x_1 & x_0 & x_1 & \cdots  & x_{n-2}\\
x_{2} & x_1 & x_0 & \cdots & x_{n-3}\\
\vdots & \vdots & {\vdots} & \ddots  & \vdots \\
x_{(p-1)(\text{mod }n)} & \cdots & \cdots & \cdots  & x_{(p-n)(\text{mod }n)}
\end{bmatrix}.
\end{align*}
We have dropped the suffix $p$ here for ease of reading. The matrices $T^{(s)},H^{(s)},R^{(s)}$ and $C^{(s)}$ are the rectangular versions of the symmetric Toeplitz, Hankel, reverse circulant and circulant matrices where the $(i,j)$th entry is equal to the $(j,i)$th entry whenever $1 \leq i,j \leq \min(p,n)$. The matrices $T,H,R$ and $C$ are the \textit{asymmetric} versions of these matrices. These matrices can also be described via \textit{link functions} (see Section \ref{preliminaries-SA}).
\citep{bose2010limiting} showed that when the entries of $A_p$ come from a single i.i.d. sequence with mean zero and variance 1, 
the ESD of $\frac{1}{n}S_{A}$ converges a.s. to a non-random probability distribution. 
We generalise this result by allowing the distribution of the entries to vary with $n$ as well as with their positions in the matrix.

Such a model for \textit{symmetric patterned matrices} such as reverse circulant, circulant, Toeplitz and Hankel was considered in \citep{bose2021some}.

Under assumptions on $A_p$ similar to those used for $X$, the EESD of $S_{A}$ converges. 
In particular, this convergence holds for the special cases when $A_p$ is a triangular, a band or a block matrix, has a variance profile or is a matrix with exploding moments.
Illustration of the variety of distributions that appear as the LSD of $S_{A}$ is given in Section \ref{simulations}.
 
Theorem \ref{res:main} claims the convergence only of the EESD. The a.s. or in probability convergence of the ESD  to a non-random probability measure is not true in general. 
Simulation given in Figures \ref{fig:SA-1} and \ref{fig:SA-2} confirm this. This is very different from the case of the $S$ matrix. In particular for the sparse case, the ESD does not converge to a non-random limit. This is similar to the phenomenon observed by \cite{banerjee2017patterned} for certain symmetric patterned sparse random matrices. 
Of course, as mentioned above in special cases the a.s. convergence can hold, see\citep{bose2010limiting}.

   We find some relationships between the LSDs of $S_{R^{(s)}}, S_{C}, S_T$ and $S_{H^{(s)}}$. For instance, when the entries are i.i.d. for every $n$ and have exploding moments, the LSDs of $S_T$ and $S_{H^{(s)}}$ are identical; so are the LSDs of $S_C$ and $S_{R^{(s)}}$. In Section \ref{application-AAt}, we discuss the connection of our theorem to some existing results.\\

 
\section{Main results}\label{main results}
The notion of \textit{multiplicative extension} is required to describe our results. Let $[k]:=\{1,2,\ldots, k\}$ and let $\mathcal{P}(k)$ denote the set of all partitions of $[k]$. Let $\mathcal{P}_{2}(2k)$ be the set of \textit{pair-partitions} of $[2k]$.   
Suppose $\{c_k, \ k \geq 1\}$ is any sequence of numbers. Its multiplicative extension is defined on $\mathcal{P}(k)$, $k \geq 1$ as follows. For any $\sigma\in \mathcal{P}_k$,  define 
\begin{equation*}c_{\sigma}=\prod_{V \ \text{is a block of}\ \sigma} c_{|V|}.
\end{equation*}

Consider the matrix $S=X_pX_p^T$, where the entries of $X_p$ are given by the bi-sequence $\{x_{ij,n}\}$. We drop the suffix $n$ and $p$ for convenience wherever there is no scope for confusion. For any real-valued function $g$ on $[0, \ 1]$,  $\|g\|:=\sup_{0 \leq x\leq 1} |g(x)|$ will denote its sup norm. We introduce the following assumptions on the entries $\{x_{ij}\}$. \\

\noindent \textbf{Assumption A}. 
There exists a sequence $\{t_n\}$ with $t_n\in [0,\infty]$ such that 
\begin{enumerate}
\item [(i)] For each $k \in \mathbb{N}$,
\begin{align}
 & n \ \mathbb{E}\left[x_{ij}^{2k}\boldsymbol {1}_{\{|x_{ij}|\leq t_n\}}\right]= g_{2k,n}\Big(\frac{i}{p},\frac{j}{n}\Big) \ \ \ \text{for } \ 1\leq i\leq p, 1 \leq j \leq n, \label{gkeven}\\
& \displaystyle \lim_{n \rightarrow \infty} \ n^{\alpha} \underset{1\leq i \leq p,1 \leq j \leq n}{\sup} \ \mathbb{E}\left[x_{ij}^{2k-1}\boldsymbol {1}_{\{|x_{ij}|\leq t_n\}}\right] = 0  \ \ \text{for all } \alpha<1 . \label{gkodd}
\end{align}
where $\{g_{2k,n}\}$ is a sequence of bounded Riemann integrable functions on $[0,1] ^2$. 

\item [(ii)] The functions $g_{2k,n}(\cdot), n \geq 1$ converge uniformly to $g_{2k}(\cdot)$ for each $k \geq 1$. 
\item[(iii)] With  $M_{2k}=\|g_{2k}\|$, $M_{2k-1}=0$ for all $k \geq 1$, the sequnce $\alpha_{k}=  \sum_{\sigma \in \mathcal{P}(2k)} M_{\sigma}$  satisfies \textit{Carleman's condition}, 
$$ \displaystyle \sum_{k=1}^{\infty} \alpha_{2k}^{-\frac{1}{2k}}= \infty.$$
\end{enumerate}
All of these conditions are naturally satisfied by well-known models such as, the \textit{standard i.i.d.} case where the entries of $X$ are $\frac{x_{ij}}{\sqrt{n}}$ with $\{x_{ij}\}$ being i.i.d. with zero mean and finite variance, and the \textit{sparse} case  where entries of $X$ are i.i.d, $Ber(p_n)$ with $np_n \rightarrow \lambda>0$, for every $n$, etc. We will discuss this in more details in Section \ref{res:XXt}. Now we state our first result.
\begin{theorem}\label{res:XXt}
Let $X$ be a $p \times n$ real matrix with independent entries $\{x_{ij,n}; 1\leq i\leq p, 1\leq j \leq n\}$ that satisfy Assumption A and $p/n \rightarrow y\in(0,\infty)$ as $n \rightarrow \infty$. Suppose $Z$ is a $p \times n$ real matrix whose entries are $y_{ij}=x_{ij}\boldsymbol {1}_{[|x_{ij}|\leq t_n]}$. Then
\begin{enumerate}
\item[(a)] The ESD of $S_Z=ZZ^T$ converges a.s. to a probability measure $\mu$ say, whose moments are determined by the functions $\{g_{2k}\}_{k\geq 1}$ as described in \eqref{moment-XXt}.
\item[(b)] Moreover, if 
\begin{equation}\label{truncation-ineq}
\frac{1}{n} \displaystyle{\sum_{i=1}^{p}\sum_{j=1}^n \ x_{ij}^2\boldsymbol {1}_{\{|x_{ij}| > t_n\}}} \rightarrow 0, \ 
\mbox{a.s. (or in probability)},
\end{equation}
 then the ESD of $S=XX^T$ converges a.s. (or in probability) to the probability measure $\mu$ given in (a).
\end{enumerate} 
\end{theorem}
\begin{remark}\label{unbounded support}
While the $MP_y$ law has bounded support, that is not necessarily the case for $\mu$ in Theorem \ref{res:XXt}. Suppose the entries of $X$ satisfy Assumption A. Let for every $m\geq 1$, $f_{2m}(x)=\int_{[0,1]}g_{2m}(x,y)\ dy$. Now suppose that there exist an $m>1$ such that $ \displaystyle \inf_{t\geq 1}\int_{[0,1]}\bigg(\frac{f_{2m}(x)}{m!}\bigg)^t \ dx = c >0$. Then the LSD $\mu$ in Theorem \ref{res:XXt} has unbounded support. 

This has implications on the partition description of moments. As is known, the moments of the $MP_y$ can be described via the set of non-crossing pair partitions. In the present case, these partitions are not enough to describe $\mu$ and we need a much bigger set of partitions. This will be discussed in details later in Section \ref{proofs-smatrix}.
\end{remark}

\begin{remark}\label{p=n} It is known that if $Y$ follows the $MP_1$ law and $Y^{\prime}$ follows the semi-circle law, then $Y \overset{\mathcal{D}}{=} Y^{\prime 2}$. A similar result holds for $\mu$. 
Suppose $p/n \to 1$, and the entries of $X$ satisfy Assumption A and \eqref{truncation-ineq}. Then the ESD of $S$ converges a.s. to a probability distribution $\mu$ as given in Theorem \ref{res:XXt}. At the same time, consider the (generalised) Wigner matrix (i.e., a symmetric matrix) with independent entries $\{x_{ij,n}; 1\leq i \leq j \leq n\}$ that satisfy the conditions of Theorem 2.1 in \citep{bose2022random}. Then its ESD converges a.s. surely to a symmetric probability measure $\mu^{\prime}$. 
The two measures $\mu$ and $\mu^{\prime}$ are connected. 
Suppose $Y$ and $Y^{\prime}$ are two random variables such that $Y \sim \mu$ and $Y^{\prime} \sim \mu^{\prime}$. If $\{g_{2k}\}_{k\geq 1}$ are symmetric functions, then, $Y \overset{\mathcal{D}}{=} Y^{\prime 2}$. This is proved in Section    \ref{proof of theorem and remark}.
\end{remark}


Now, we shall consider the matrices $S_{A}$, where the entries of $A_p$ are constructed from the sequence of random variables $\{x_{i,n}; -(n+p) \leq i \leq (n+p)\}$. We will denote $A_p$ by $A$ and write $x_i$ for  $x_{i,n}$. Recall that as $p,n \rightarrow \infty$, $p/n \rightarrow y \in (0,\infty)$.

\vspace{0.2cm}
\noindent \textbf{Assumption B.} 
Suppose there exists a sequence $\{t_n\}$ with $t_n\in [0,\infty]$ such that 
\begin{enumerate}
\item [(i)] for each $k \in \mathbb{N}$,
\begin{align}
 & n \ \mathbb{E}\left[x_{i}^{2k}\boldsymbol {1}_{\{|x_{i}|\leq t_n\}}\right]= f_{2k,n}\Big(\frac{i}{n}\Big) 
 \ \ \ \text{for } \ -(n+p)\leq i \leq n+p, \label{fkeven}\\
& \displaystyle \lim_{n \rightarrow \infty} \ n^{\alpha} \underset{0\leq i  \leq n-1}{\sup} \ \mathbb{E}\left[x_{i}^{2k-1}\boldsymbol {1}_{\{|x_{i}|\leq t_n\}}\right] = 0  \ \ \text{for all } \alpha<1 . \label{fkodd}
\end{align}
where $\{f_{k,n};0\leq k\leq n\}$ is a sequence of bounded and  integrable functions on $[-(1+y),1+y]$. 
\item [(ii)] For each $k \geq 1$, $f_{2k,n}, n \geq 1$ converge uniformly to a function $f_{2k}$ . 
\item[(iii)] Let  $M_{2k}=\|f_{2k}\|$ (where $\|\cdot\|$ denotes the sup norm) and $M_{2k-1}=0$ for all $k \geq 1$. Suppose $\alpha_{k}=  \sum_{\sigma \in \mathcal{P}(2k)} M_{\sigma}$  satisfy \textit{Carleman's condition}, 
$$ \displaystyle \sum_{k=1}^{\infty} \alpha_{2k}^{-\frac{1}{2k}}= \infty.$$
\end{enumerate}
As we will see in Section \ref{application-AAt}, these assumptions are naturally satisfied by various well-known models.
Now we state our second result.

 \begin{theorem}\label{res:main}
 Suppose $A$ is one of the eight rectangular matrices $T^{(s)},T, H^{(s)},H, R^{(s)},R,C^{(s)},C$ described in Section \ref{introduction}, with entries $\{x_{i}\}$ which are independent and satisfy Assumption B. Let $Z_A$ be the $p\times n$ (truncated) matrix with entries $y_l=x_l\boldsymbol {1}_{\{|x_{l}|\leq t_n\}}$. Then the EESD of $S_{Z_A}=Z_AZ_A^T$ converges weakly to a probability measure $\mu_{A}$ say,  whose moment sequence is determined by the functions $f_{2k}, \ k \geq 1$, in each of the eight cases. Further if 
\begin{align}\label{S-truncationtoe-XXt}
 \sum_{l} \mathbb{E}[ x_{l}^2\boldsymbol {1}_{\{|x_{l}| > t_n\}}] \rightarrow 0, 
\end{align}
then the EESD of $S_{A}=AA^{T}$ converges weakly to $\mu_{A}$.
 \end{theorem}

\begin{remark}\label{notas} As mentioned earlier, the a.s. or in probability convergence of the ESD to the limit $\mu_A$ does not hold in general. This is clear from the simulations given in Figures \ref{fig:SA-1} and \ref{fig:SA-2}. In particular there is no a.s. convergence in the sparse case. A similar phenomenon occurs  for the sparse symmetric patterned matrices (see \citep{banerjee2017patterned}). Of course, a.s. convergence can hold in special cases, for example in the fully i.i.d. case.
\end{remark}

\section{Simulations}\label{simulations}
The LSDs cannot be universal and a variety limit distributions are possible. 
In Figure \ref{fig:S1}, we see the diversity of the LSDs for the $S$ matrix. 
Moreover, even though $\mu_S$ converges a.s. to $\mu$, as noted in Remark \ref{notas}, $\mu_{S_{A}}$ does not  converge a.s. to $\mu_A$ in general. The reason is that, unlike $X$, where each entry appears exactly once,  there is a significant structural dependence among the entries of $A$ and each entry appears $\mathcal{O}(n)$ times. 
However, the a.s. convergence can hold in special cases. 
For instance when the entries of $A$ are $\frac{x_i}{\sqrt{n}}$ with $x_i$s i.i.d. with mean zero finite variance, then it is well-known that the  ESDs of $S_{A}$ do converge a.s. to non-random probability measures (see \citep{bose2010limiting}). Figure \ref{fig:SA-1} and \ref{fig:SA-2} illustrates that a.s. convergence of $\mu_{S_{A}}$ holds  when the entries of $A$ are $N(0,1)/\sqrt{n}$ and fails when the entries of $A$ are $Ber(3/n)$ with for every fixed n.
 Since the entries of $A$ need to be only independent, matrices with variance profile serve as natural examples in demonstrating the diversity of the limit distributions. In Figure \ref{fig:S3} 
we give some simulated $\mu_{S_{A}}$ when $A$ obey a variance profile. 

\section{Details for the $S$ matrix}\label{s-matrix} 
We begin with some preliminaries that are required in the proofs. Then we discuss how Theorem \ref{res:XXt} is applicable when specific features such as i.i.d, heavy-tails,  triangular (size dependent entries), sparsity, variance profile are there in the model. This is followed by a detailed proof of Theorem \ref{res:XXt}. Finally we connect our moment formula for the LSD with the moment formulae known in the literature that are based on hypergraphs and words. 

\subsection{Preliminaries}\label{preliminaries-smatrix}
 \noindent We first briefly introduce the language of link functions, circuits, words etc. in the context of the $S$ matrix that we shall use heavily. For more details of these concepts, please refer to Section 3 of \citep{bose2022random} and Section 4 in \citep{bose2021some}. 
\vskip3pt

\noindent \textbf{Link function}:  The link function for $S$ is given by a pair of functions as follows. 
\begin{align*}
L1(i,j)=(i,j) \ \ \text{ and } \ \ L2(i,j)=(j,i).
\end{align*}

\vskip3pt

\noindent \textbf{Circuits and Words}: In case of the $S$ matrix, a \textbf{circuit} $\pi$ is a function $\pi : \{0,1,2,\ldots, 2m\}\rightarrow \{1,2,3,\ldots,\max(p,n)\}$ with $\pi(0)=\pi(2m)$ and $1 \leq \pi(2i)\leq p, 1 \leq \pi(2i-1)\leq n$ for $1 \leq i \leq m$. We say that the \textit{length} of  $\pi$ is $2m$ and denote it by $\ell(\pi)$. Next, let
\begin{align*}
\xi_{\pi}(2i-1) &= L1(\pi(2i-2),\pi(2i-1)), 1 \leq i \leq k\\
\xi_{\pi}(2i) &= L2(\pi(2i-1), \pi(2i)),1 \leq i \leq k.
\end{align*}
Then,
\begin{align}\label{genmoment-XXt}
\mathbb{E}\big[\Tr(S^{k})\big]& =\mathbb{E}\big[\Tr(XX^*)^{k}\big]\nonumber \\
 & =\sum_{\pi:\ell(\pi)=k}x_{L1(\pi(0),\pi(1))}x_{L2(\pi(1),\pi(2))}\cdots x_{L2(\pi(2k-1),\pi(2k))} = \sum_{\pi:\ell(\pi)=2k}\mathbb{E}[Y_{\pi}],
\end{align}
where $Y_{\pi}= \displaystyle \prod_{i=1}^k y_{\xi_{\pi(2i-1)}}y_{\xi_{\pi(2i)}}$.\\
From \eqref{genmoment-XXt}, observe that the $k$th moment of an entry of $S$ involves the $2k$th moment of the entries of $X$.  Hence the circuits that are required to deal with the $k$th moment of the $S$ matrix are of length $2k$.

For any $\pi$, the values $Lt(\pi(i-1),\pi(i)),t=1,2$ will be called \textit{edges} or $L$-values. When an edge appears more than once in a circuit $\pi$, then it is called 
\textit{matched}. Any $m$ circuits $\pi_1,\pi_2,\ldots,\pi_m$ are said to be \textit{jointly-matched} if each edge occurs at least twice across all circuits. They are said to be \textit{cross-matched} if each circuit has an edge which occurs in at least one of the other circuits. 
Circuits $\pi_1$ and $\pi_2$ are said to be \textit{equivalent} if
\begin{align*}
Lt(\pi_1(i-1),\pi_1(i))=Lt(\pi_1(j-1),\pi_1(j))
\Longleftrightarrow Lt(\pi_2(i-1),\pi_2(i))=Lt(\pi_2(j-1),\pi_2(j)), t=1,2.
\end{align*} 

The above is an equivalence relation on $\{\pi:\ell(\pi)=k\}$. Any equivalence class  of circuits can be indexed by an element of $\mathcal{P}(k)$. The positions where the edges match are identified by each block of a partition of $[k]$. 
Also, an element of $\mathcal{P}(k)$ can be identified with a \textit{word} of length $k$ of letters.  Given a partition, we represent the integers of the same partition block by the same letter, and the first occurrence of each letter is in alphabetical order and vice versa. For example, the partition $\{\{1,5\}, \{2,3,4\}\}$ of $[5]$ corresponds to the word $abbba$. On the other hand, the word $abcccaa$ represents the partition $\{\{1,6,7\},\{2\},\{3,4,5\}\}$  of $[7]$. A typical word will be denoted by $\boldsymbol{\omega}$ and its $i$-th letter as $\boldsymbol{\omega}[i]$. 

\noindent \textbf{The class $\Pi_S(\boldsymbol {\omega})$}:  For a given word $\boldsymbol {\omega}$, this is the set of all circuits which correspond to $\boldsymbol {\omega}$.
For any word $\boldsymbol {\omega}$,  $\boldsymbol {\omega}[i]=\boldsymbol {\omega}[j] \ \Leftrightarrow \xi_{\pi}(i)= \xi_{\pi}(j)\}$. This  implies 
\begin{align*}
Lt(\pi(i-1),\pi(i))& = Lt(\pi(j-1),\pi(j)) \ \ \ \text{ if } i  \ \text{ and } j \ \text{ are of same parity}, \ t=1,2 \\
Lt(\pi(i-1),\pi(i))& = Lt^{\prime}(\pi(j-1),\pi(j)) \ \ \ \text{ if } i  \ \text{ and } j \ \text{ are of different parity}, \ t,t^{\prime}\in \{1,2\}, t\neq t^{\prime}. 
\end{align*}
Therefore the class $\Pi_S(\boldsymbol {\omega})$ is given as follows:
\begin{align}\label{S-link}
\Pi_S(\boldsymbol {\omega}) & = \{\pi; \boldsymbol {\omega}[i]=\boldsymbol {\omega}[j] \ \Leftrightarrow \xi_{\pi}(i)= \xi_{\pi}(j)\} \nonumber \\
& = \Big\{\pi: \boldsymbol {\omega}[i]=\boldsymbol {\omega}[j] \Leftrightarrow (\pi(i-1),\pi(i))= (\pi(j-1),\pi(j))\ \ \text{ or } (\pi(i-1),\pi(i))= (\pi(j),\pi(j-1))\Big\}.
\end{align}

From \eqref{genmoment-XXt} observe that,
\begin{align}
\lim_{p \rightarrow \infty}\frac{1}{n}\mathbb{E}[\Tr(S^{k}] =\displaystyle\lim_{n \rightarrow \infty}\frac{1}{p} \sum_{\pi:\ell(\pi)=2k}\mathbb{E}[Y_{\pi}] 
 = \displaystyle\lim_{p \rightarrow \infty} \displaystyle \sum_{b=1}^k  \sum_{\underset{\text{with b distinct letters}}{\omega \ \text{matched of length }2k}}\frac{1}{p} \sum_{\pi \in \Pi_{S}(\boldsymbol{\omega})} \mathbb{E}(Y_{\pi}).    
\end{align}
\vskip5pt
Note that all words that appear above are of length $2k$. For every $k \geq 1$, the words of length $2k$ corresponding to the circuits of $S$ and generalised Wigner matrix, $W$, are related (see Observation 1 below). We will find a connection between the $k$th moments of the LSD of $S$ and $2k$th moments of the LSD of the generalised Wigner matrix. We will also discover that the partitions that contribute to the latter plays a crucial role in the former.

Let $L_{W}(i,j)=(\min(i,j),\max(i,j))$ denote the link function of the Wigner matrix. For words corresponding to $L_W$, the class $\Pi_W(\boldsymbol {\omega})$ is given by 
\begin{align}\label{Wignerlink}
\Pi_W(\boldsymbol {\omega}) &= \Big\{\pi: \boldsymbol {\omega}[i]=\boldsymbol {\omega}[j] \Leftrightarrow L_W(\pi(i-1),\pi(i))= L_W(\pi(j-1),\pi(j))\Big\} \nonumber \\
&= \Big\{\pi: \boldsymbol {\omega}[i]=\boldsymbol {\omega}[j]  \Leftrightarrow (\pi(i-1),\pi(i))= (\pi(j-1),\pi(j))
 \text{ or } (\pi(i-1),\pi(i))= (\pi(j),\pi(j-1))\Big\}.
\end{align}
Next, we make a key observation about $\Pi_S(\boldsymbol {\omega})$ and $\Pi_W(\boldsymbol {\omega})$. 
\vskip5pt
\noindent \textbf{Observation 1}: Let $\tilde{\Pi}_W(\boldsymbol {\omega})$ be the possibly larger class of the circuits for the Wigner Link function with range $1 \leq \pi(i) \leq \max(p,n), 0 \leq i \leq 2k$. Then for a word $\boldsymbol{\omega}$ of length $2k$, 
\begin{equation}\label{W-S}
\Pi_S(\boldsymbol {\omega})\subset \tilde{\Pi}_W(\boldsymbol {\omega}). 
\end{equation}

Now we recall the definition of special symmetric words from \citep{bose2022random}. Towards that, we first define the following:\\

\noindent \textbf{Pure block of a word}: Any string of length $m (m > 1)$ of same letter in a word will be called a pure block of size $m$.
For example, in the word $aabcddddeeeb$, $a$, $d$ and $e$ appear in pure blocks of sizes 2, 4 and 3 respectively.\\

\noindent\textbf{Special Symmetric word}
 A word $\boldsymbol {\omega}$ is \textit{special symmetric} if the following conditions hold: \vskip3pt

\noindent (a) the last new letter appears in pure blocks of even sizes, \vskip3pt

\noindent (b) between two successive appearances of any letter:\vskip3pt

\noindent (b)(i) each of the other letters appears an even number of times, and  \vskip3pt

\noindent (b)(ii) each of the other letters appears an equal number of times in the odd and even positions. \\

For example, $abbccaabba$ is a special symmetric word of length 10 with 3 distinct letters.\\

Condition (b)(i) actually implies both 
Conditions (a) and (b)(ii). 
So the special symmetric words could as well be defined as those words which satisfy Condition (b)(i).
This fact was observed in \citep{pernici2021noncrossing}. We retain all three conditions in the definition 
for clarity and ease of use. 
\vskip5pt

We denote the set of all special symmetric partitions of $[k]$ by $SS(k)$, and its subset where each partition has $b$ distinct blocks by $SS_b(k)$. Clearly $SS(k)=\emptyset$ when $k$ is odd. Let $\mathcal{P}_{2}(2k)$ and $NC_2(2k)$ be respectively the set of pair-partitions and non-crossing pair-partitions of $[2k]$.
Then it is easy to check that 
$$\left(SS(2k)\cap \mathcal{P}_{2}(2k)\right)=NC_2(2k) \subset SS(2k).$$ 
Next, recall the definition of \textit{generating and non-generating vertices} from \citep{bose2022random} and \citep{bose2021some}.
\begin{definition} 
If $\pi$ is a circuit then any $\pi(i)$  will be called a \textit{vertex}. This vertex is \textit{generating} if $i=0$ or $\boldsymbol {\omega}[i]$ is the first occurrence of a letter in the  word $\boldsymbol {\omega}$ corresponding to $\pi$. All other vertices are \textit{non-generating}.
\end{definition} 

For example, for the word $abbc$,  $\pi(0), \pi(1), \pi(2)$ and $\pi(4)=\pi(0)$  are generating.
For the word $aaaa$, $\pi(0)$ and $\pi(1)$ are generating. 
 \vskip5pt
 
\noindent \textbf{Even and odd generating vertices:} 
 A generating vertex $\pi(i)$ is called even (odd) if $i$ is even (odd). Any word has at least one of each, namely $\pi(0)$ and $\pi(1)$. So for a matched word with $b(\leq k/2)$ distinct letters there can be $(r+1)$ even generating vertices where $0 \leq r \leq b-1$.

Observe that  
\begin{align}\label{S-Pi(omega)}
\big|\Pi_S(\boldsymbol {\omega})\big| = \big|\big\{ &\big(\pi(0), \pi(1),\ldots,\pi(2k)\big): 1\leq \pi(2i)\leq p, 1 \leq \pi(2i-1)\leq n \text{ for } i=0,1,\ldots,k,  \nonumber\\
 &\pi(0)=\pi(2k), \quad \xi_{\pi}(i)=\xi_{\pi}(j) \text{ if and only if }\boldsymbol {\omega}[i]=\boldsymbol {\omega}[j]  \big\}\big|.
\end{align}

Circuits corresponding to a word $\boldsymbol {\omega}$ are completely determined by the generating vertices. 
\ $\pi(0)$ is always generating, and there is one generating vertex for each new letter in $\boldsymbol{\omega}$. So, if $\boldsymbol {\omega}$ has $b$ distinct letters then it has $(b+1)$ generating vertices. Hence the growth of $|\Pi_S(\boldsymbol {\omega})|$ is determined by the number of generating vertices that can be chosen freely. For some words, depending on the link function and the nature of the word, some of these vertices may not have a free choice, that is some of the generating vertices might be a linear combination of the other generating vertices. In any case, as $p/n\rightarrow y>0$,
\begin{equation}\label{cardiality of word-S}
| \Pi_S(\boldsymbol {\omega})| = \mathcal{O}(p^{r+1}n^{b-r})\ \ \text{whenever} \ \omega \text{ has } \\   b \ \ \text{distinct letters and }(r+1) \ \ \text{ even generating vertices.}
\end{equation} 
The existence of 
\begin{align}\label{word limit}
\lim_{p,n \rightarrow \infty} \frac{| \Pi_S(\boldsymbol {\omega})|}{p^{r+1}n^{b-r}} 
\end{align}
for every word $\boldsymbol {\omega}$ is tied very intimately to the LSD of $S$. To see this observe that if the variables are centered (see \eqref{genmoment-XXt}), 
 \begin{align}\label{momentnoniid-XXt}
\lim_{p \rightarrow \infty}\frac{1}{n}\mathbb{E}[\Tr(S_{Z})^{k}]& =\displaystyle\lim_{n \rightarrow \infty}\frac{1}{p} \sum_{\pi:\ell(\pi)=2k}\mathbb{E}[Y_{\pi}] \nonumber\\
& = \displaystyle\lim_{p \rightarrow \infty} \displaystyle \sum_{b=1}^k  \sum_{\underset{\text{with b distinct letters}}{\omega \ \text{matched}}}\frac{1}{p} \sum_{\pi \in \Pi_{S}(\boldsymbol{\omega})} \mathbb{E}(Y_{\pi})\nonumber\\
 & = \displaystyle\lim_{p \rightarrow \infty} \displaystyle \sum_{b=1}^k \sum_{r=0}^{b} \sum_{\underset{\underset{\text{and }(r+1)\text{ even generating vertices}}{\text{with b distinct letters}}}{\omega \ \text{matched}}}\frac{1}{p} \sum_{\pi \in \Pi_{S}(\boldsymbol{\omega})} \mathbb{E}(Y_{\pi}).    
\end{align}
As the entries $x_{ij}$ are independent, $\mathbb{E}[Y_{\pi}]= \prod_{l=1}^b \mathbb{E}[x_{(\pi(l-1),\pi(l))}]$, where $(\pi(l-1),\pi(l))$ denotes the distinct $L-$values corresponding to each distinct letter in $\boldsymbol{\omega}$. Now from \eqref{cardiality of word-S}, it can be seen that $\lim_{p,n \rightarrow \infty} \frac{| \Pi_S(\boldsymbol {\omega})|}{p^{r+1}n^{b-r}}$ determines whether or not all generating vertices of $\boldsymbol{\omega}$ have free choice. Therefore, from \eqref{momentnoniid-XXt}, it is easy to see that $\lim_{n \rightarrow \infty} \frac{| \Pi_S(\boldsymbol {\omega})|}{p^{r+1}n^{b-r}}$ will determine when $\boldsymbol{\omega}$ can have a positive contribution to the limiting moments.

  In the next section, we identify for which words the above limit is positive for $S$.
  \vskip3pt

We shall use the the \textit{$L\acute{e}vy$ metric}.
Let $F$ and $G$ be two distribution functions. Then the \textit{$L\acute{e}vy$ distance} between $F$ and $G$ is given by 
$$L(F,G)= \inf\{ \epsilon : F(x-\epsilon) - \epsilon \leq G(x) \leq F(x+\epsilon) + \epsilon\}.$$
It is well-known that if $\{\mu_n\}$ and $\mu$ are probability measures, then $L(\mu_n,\mu) \rightarrow 0$ as $n \rightarrow \infty$, implies $\mu_n$ converges  to $\mu$.

The next lemma is a well-known result that is useful in the proof of Theorem \ref{res:XXt}. For a proof see Corollary A.42 in \citep{bai2010spectral}.
\begin{lemma}\label{lem:metric}
Suppose $A$ and $B$ are real $p \times n$ matrices and $F^{S_A}$ and $F^{S_B}$ denote the ESDs of $AA^T$ and $BB^T$ respectively. Then the \textit{$L\acute{e}vy$ distance}, $L$ between the distributions $F^{S_A}$ and $F^{S_B}$ satisfy the following inequality: 
\begin{equation}\label{levy}
L^4(F^A,F^B)\leq \frac{2}{p^2}(\Tr(AA^T+BB^T))(\Tr[(A-B)(A-B)^T]).
\end{equation}
\end{lemma}	

\noindent Next, we state the following elementary result that helps us conclude the a.s. convergence of the ESD of matrices. 
See Section 1.2 of \citep{bose2018patterned} for a proof.
\begin{lemma}\label{lem:genmoment}
Suppose $A_n$ is any sequence of symmetric $n \times n$ random matrices such that the following conditions hold:
\begin{enumerate}
\item[(i)] For every $k\geq 1$, $\frac{1}{n}\mathbb{E}[\Tr (A_n)^k] \rightarrow \alpha_k$ as $n \rightarrow \infty$.
\item[(ii)] $\displaystyle \sum_{n=1}^{\infty}\frac{1}{n^4}\mathbb{E}[\Tr(A_n^k) \ - \ \mathbb{E}(\Tr(A_n^k))]^4  < \infty$ for every $k \geq 1$.
\item[(iii)] The sequence $\{\alpha_k\}$ is the moment sequence of a unique probability measure $\mu$.
\end{enumerate}
Then $\mu_{A_n}$ converges to $\mu$
weakly a.s.
\end{lemma}

Condition (i) and (ii) of Lemma \ref{lem:genmoment} will be referred to as the \textit{first moment condition} and the \textit{fourth moment condition}, respectively. 


\subsection{Relation to existing results on the $S$ Matrix}\label{realtion to existing results} 

\subsubsection{ \textbf{I.I.D. entries}}\label{iid}
 Suppose $X=((x_{ij}/\sqrt{n}))$ where $\{x_{ij}\}$ are i.i.d. with distribution  $F$ which has mean zero and variance $1$. 
 It is known that $\mu_S$ converges a.s. to $MP_y$. For a brief history and precursors of this result, see \citep{bose2018patterned} and \citep{bai2010spectral}. Here, we show how this result follows as a special case of Theorem \ref{res:XXt}.
 
 First, let us verify that the conditions of Assumption A are satisfied in this case. Towards that, let $t_n=n^{-1/3}$. 
Using the same line of reasoning as in Section 5.1 (a) in \citep{bose2022random}, it follows that $g_2\equiv 1$ and $g_{2k}\equiv 0, k>1$. Thus $M_{2}=1$, $M_{2k}=0, k \geq 2$ (see (iii) in Assumption A) and $\alpha_k= \sum_{\sigma \in \mathcal{P}(2k)} 1 $ clearly satisfies Carleman's condition. 
Now for any $ t > 0$, 
\begin{align*}
\frac{1}{p}\displaystyle \sum_{i,j} \ \big(x_{ij}/\sqrt{n}\big)^2[\boldsymbol {1}_{[|x_{ij}/\sqrt{n}| > t_n]}]= & \frac{1}{np}\displaystyle \sum_{i,j} \ x_{ij}^2[\boldsymbol {1}_{[|x_{ij}| > t_n \sqrt{n}]}]\\
& \leq \frac{1}{np}\displaystyle \sum_{i,j} \ x_{ij}^2[\boldsymbol {1}_{[|x_{ij}| > t]}] \ \ \text{for all large} \ n,\\
& \overset{a.s.}{\longrightarrow} \ \mathbb{E}\big[ x_{11}^2[\boldsymbol {1}_{[|x_{11}| > t]}]\big].  
\end{align*}
As $\mathbb{E}[x_{11}^2]=1$, taking $t$ to infinity, the above limit is 0 a.s. 
 Hence applying Theorem \ref{res:XXt}, the ESD of $S$ converges a.s. to $\mu$ whose $k-$th moment is given by 
\begin{align}
\beta_k(\mu)&= \displaystyle \sum_{r=0}^{k-1} \sum_{\underset{\underset{\text{ even generating vertices}}{\text{with }(r+1)}}{\pi \in SS_k(2k)}} y^r 
\end{align}
 
\noindent Now the number of pair-matched words of length $2k$ with $r+1$ even generating vertices is shown to be $\frac{1}{r+1}{k \choose r}{{k-1} \choose r}$ in 
Theorem 5(a) of \citep{bose2008another}. Hence the rhs of the above equation reduces to the $k$th moment of the $MP_y$ law. Hence $\mu_S$ converges to the $MP_y$ law a.s. 
\subsubsection{ \textbf{Heavy-tailed entries}} Suppose $\{x_{ij}, 1 \leq i \leq p, 1 \leq j\leq n\}$ are i.i.d. with an $\alpha$-stable distribution ($0<\alpha <2$) and $n/p \rightarrow \gamma \in (0,1]$. 
Let $X=((x_{ij}/a_p))$ where $$a_p= \inf\Big\{u: \mathbb{P}[|x_{ij}|\geq u]\leq \frac{1}{p}\Big\}.$$ The existence of LSD of $S$ using Stieltjes transform, has been proved in \citep{belinschi2009spectral}. Theorem \ref{res:XXt} may be used to  give an alternative proof. We recall that for the Wigner matrix with heavy tailed entries, a proof using truncation and moments is available in \citep{bose2022random}. That proof can be easily adapted here. For a fixed constant $B$, let $X^B=
((\frac{x_{ij}}{a}\boldsymbol {1}_{[|x_{ij}|\leq B a]}))$. Then 
$X^B$ satisfies Assumption A. Hence from Theorem \ref{res:XXt}, the ESD of $S^B= X^B(X^{B})^{T}$, a.s. converges a.s., to say $\mu_B$. 
The rest of the arguments are as in Section 5.2 of \citep{bose2022random}.
 Thus  $\mu_{S}$ converges to $\tilde{\mu}$ in probability and yields the convergence in Theorem 1.10 of \citep{belinschi2009spectral}.

\subsubsection{ \textbf{Triangular i.i.d. (size dependent matrices)}}\label{triangular iid}
 Suppose $\{x_{ij,n}; 1\leq i \leq p,1 \leq j \leq n \}$ is a sequence of i.i.d. random variables with distribution $F_n$ that has finite moments of all orders, for every $n$. Also assume that for every $k \geq 1$,
\begin{align}\label{ck}
n \beta_k(F_n) \rightarrow C_k
\end{align}
where $\beta_k(F_n)$ denotes the $k$th moment of $F_n$. Suppose $\{C_2,C_4, \ldots \}$ is the cumulant sequence of a probability distribution whose moment sequence satisfies Carleman's condition.
\begin{remark}\label{infinte divisible}
The condition \eqref{ck} is equivalent to the statement that $\displaystyle\sum_{i=1}^na_{i,n}$, where $a_{i,n}$ are i.i.d. $F_n$ converges to some limit distribution $F$, whose cumulants are $\{C_k\}$. In particular, if $F$ is infinitely divisible, then the existence of such variables $\{a_{i,n}\}$ are guaranteed. See, p.766 (characterization 1) in \citep{bose2002contemporary}. Also then, $\{C_2,C_4, \ldots\}$ is indeed a cumulant sequence. 
\end{remark} 

Let 
$X=((x_{ij,n}))$
and  
$p/n \rightarrow y>0$.
Condition \eqref{ck} implies that Assumption A holds with $t_n=\infty$ and $g_{2k}\equiv C_{2k}, k \geq 1$. Therefore by Theorem \ref{res:XXt},  
$\mu_S$ converges a.s. to $\mu$ with moments 
\begin{align}\label{Zak-limit}
\beta_{k}(\mu)=\displaystyle \sum_{r=0}^{k-1}\sum_{\underset{\underset{\text{ even generating vertices}}{\text{with }(r+1)}}{\pi \in SS(2k)}} y^r C_{\pi}.
\end{align}

We now show how Theorem 3.2 in \citep{Benaych-Georges2012} and Proposition 3.1 in \citep{noiry2018spectral} can be deduced from 
Theorem \ref{res:XXt}. So, suppose the entries of $X$ are i.i.d. with distribution $\mu_n$ that has mean zero and all moments finite and $$\displaystyle{\lim_{n\to\infty} \frac{\beta_k(\mu_n)}{n^{k/2-1}\beta_2(\mu_n)^{k/2}}}= c_k,  \ \ \text{say,  exists for all} \ \ k\geq 1.$$ 
The LSD of $S=XX^T$ 
when $c_k^{1/k}$ is bounded, was considered in Theorem 3.2 of \citep{Benaych-Georges2012} and Proposition 3.1 of \citep{noiry2018spectral}. Clearly Assumption A is satisfied with $t_n=\infty$, $g_2\equiv 1$ and $g_{2k} \equiv C_{2k}, k \geq 2$. Hence Theorem \ref{res:XXt} can be applied and the resulting LSD, say, $\mu$ has moments as in \eqref{Zak-limit}. In Section \ref{connection-ss2k}, we shall verify that this LSD is the same as those obtained in the above references. 
\vskip3pt

\noindent \textbf{Connection to the limiting moments of the generalised Wigner matrices}: As observed above
suppose the entries are triangular i.i.d. that satisfy \eqref{ck}. Then $g_{2k}\equiv C_{2k}, k \geq 1$.
Then Remark \ref{p=n}, applies and hence $X \overset{\mathcal{D}}{=} Y^2$ where $X \sim \mu$, $Y \sim \mu^{\prime}$ , $\mu$ and $\mu^{\prime}$ are the LSDs of the $S$ matrix and the generalised Wigner matrix respectively. 

\subsubsection{ \textbf{Sparse S}} \label{subsubparse}
A well-studied sparse matrix model is where the entries of $X$ have $Bernoulli$ distribution with parameter $p_n$ such that $n p_n \rightarrow \lambda>0$. 
Thus, 
\eqref{ck} holds with $C_{k}\equiv \lambda$ for all $k\geq 1$. Hence by Theorem \ref{res:XXt} 
$\mu_S$ converges a.s. to $\mu$ whose moments are (see \eqref{Zak-limit}):
\begin{equation}\label{sparse-hom}
\beta_{k}(\mu)= \displaystyle \sum_{r=0}^{k-1}\sum_{\underset{\underset{\text{ even generating vertices}}{\text{with }(r+1)}}{\pi \in SS(2k)}} y^r \lambda^{|\pi|}.
\end{equation}

Explicit description of $\mu$ is not available. 
However, we can say the following.
Let $E(2k)$ (and $NCE(2k)$) be the set of partitions (and non-crossing partitions) whose blocks are all of even sizes. Then it is easliy seen that $NCE(2k) \subset SS(2k) \subset E(2k).$ Therefore we have the following:\\
\vskip3pt

\noindent Case 1: $y \leq 1$. Then from \eqref{sparse-hom}
\begin{equation}\label{y<=1}
\displaystyle \sum_{\pi \in NCE(2k)} (\lambda y)^{|\pi|} < \beta_{k}(\mu) < \sum_{\pi \in E(2k)} \lambda^{|\pi|}
\end{equation}
Case 2: $y > 1$. Then from \eqref{sparse-hom}
\begin{equation}\label{y>1}
\displaystyle \sum_{\pi \in NCE(2k)} y^{|\pi|} < \beta_{k}(\mu) < \sum_{\pi \in E(2k)} (\lambda y)^{|\pi|}
\end{equation}

Now suppose $P_1(\gamma)$ is a free Poisson variable with mean $\gamma$ and $P_2(\gamma)$  is a Poisson variable with mean $\gamma$. Let $Y$ be a random variable which takes value $1$ and $-1$ with probability $\frac{1}{2}$ each. Suppose $Y$ is independent of $P_1(\gamma)$ and $P_2(\gamma)$. Consider $Q_1(\gamma)= P_1(\gamma)Y$ and $Q_2(\gamma)= P_2(\gamma)Y$. Then the moments of $Q_1(\gamma)$ and $Q_2(\gamma)$ are give as follows:
\begin{align}
\mathbb{E}[Q_1^k(\gamma)]&=\begin{cases}
0 \ & \ \text{ if } k \text{ is odd },\\
\displaystyle \sum_{\pi \in NCE(k)} \gamma^{|\pi|} & \ \text{ if } k \text{ is even }. 
\end{cases} \label{Q_1}\\
\mathbb{E}[Q_2^k(\gamma)]& =\begin{cases}
0 \ & \ \text{ if } k \text{ is odd },\\
\displaystyle \sum_{\pi \in E(k)} \gamma^{|\pi|} & \ \text{ if } k \text{ is even }. 
\end{cases}\label{Q_2}
\end{align}

Hence 
\eqref{y<=1} and \eqref{y>1}, can be rewritten as 
\begin{align}
\mathbb{E}[(Q_1(\lambda y))^{2k}]< \beta_k(\mu)< \mathbb{E}[(Q_2(\lambda ))^{2k}] \ \ \text{ for every } k\geq 1, y\leq 1, \label{poisson relation-1}\\
\mathbb{E}[(Q_1(\lambda))^{2k}]< \beta_k(\mu)< \mathbb{E}[(Q_2(\lambda y))^{2k}] \ \ \text{ for every } k\geq 1, y> 1.\label{poisson relation-2}
\end{align} 
Thus the LSD of $S$ lies between the free Poisson and Poisson distributions in the above sense.
\subsubsection{ \textbf{Matrices with a variance profile}}\label{var-prof-smatrix}
The $S$ matrix, where the entries of $X$ are independent but not necessarily identically distributed have been considered in  \citep{yin1986limiting}, \citep{lytova2009central} and \citep{bai2010spectral} where a common theme has been to assume that the entries have equal variances. 
Recent works such as \citep{zhu2020graphon}, \citep{jin2020result} drop this assumption. 
We now show that if $X$ has a suitable variance profile, then $\mu_S$ converges a.s.  
We consider two profiles.

In the first, $X$ has a \textit{discrete variance profile} so that $X=((y_{ij}=\sigma_{ij}x_{ij})); 1\leq i \leq p, 1 \leq j\leq n \}$ where $\{x_{ij}\}$ are i.i.d. random variables and $\{\sigma_{ij}\}$ satisfy certain conditions. 
A similar model for the Wigner matrix was considered 
in 
Result 5.1 of \citep{bose2022random}.
We state a similar result 
for $S$ whose proof uses arguments similar to the proof of Theorem \ref{res:XXt} and Result 5.1 in \citep{bose2022random} and we omit the details. 

\begin{result}
Consider the matrix $X$ with entries $\{\frac{y_{ij}}{\sqrt{n}}= \frac{\sigma_{ij}x_{ij}}{\sqrt{n}} : 1\leq i \leq p, 1\leq j \leq n\}$ that are independent and satisfy the following conditions:
\begin{enumerate}
\item[(i)] $\mathbb{E}x_{ij}=0$ and $\mathbb{E}[x_{ij}^2]= 1$.
\item[(ii)] $\sigma_{ij}$ satisfy the following:
\begin{align}\label{s-variance}
\underset{1 \leq i \leq p}{\sup}\ \displaystyle \bigg|\frac{1}{n}\sum_{j=1}^n {\sigma^2}_{ij} -1\bigg| \rightarrow 0 \ \ \  \text{ as } n \rightarrow \infty.
\end{align}
\item[(iii)] $\displaystyle \lim_{n \rightarrow \infty} \frac{1}{n^2} \displaystyle \sum_{i,j} \mathbb{E}\big[x_{ij}^2]\boldsymbol{1}_{[|x_{ij}|>\eta \sqrt{n}]}\big]=0$ for every $\eta>0$.
\end{enumerate}
Then the ESD of $S$ converges a.s. to the $MP_y$ law, where $0< y= \lim p/n$. 
\end{result}

\begin{remark}
Theorem 1.2 in \citep{jin2020result} states a similar result where \eqref{s-variance} is replaced by $\displaystyle \frac{1}{n}\sum_{i}\bigg|\frac{1}{n}\sum_{j=1}^n {\sigma^2}_{ij} -1\bigg| \rightarrow 0.$ However, the proof equation of (2.6) there is not very clear. 
\end{remark}

Next, we consider a \textit{continuous variance profile}. 
Suppose $X=((y_{ij,n}=\sigma(i/p,j/n)x_{ij,n}))$
where $\{x_{ij,n}\}$ are i.i.d. for every fixed $n$ and satisfy the conditions given in \eqref{ck}, and $\sigma$ is a bounded piecewise continuous function on $[0,1]^2$. 
 Then $\mu_S$ converges a.s. to a symmetric probability measure $\nu$ whose $k$th moment is determined by $\sigma$ and $\{C_{2m}\}_{1\leq m\leq 2k}$.

To see this, note that $\{y_{ij,n}\}$ satisfy Assumption A with $g_{2k}\equiv \sigma^{2k}C_{2k}$. 
 By Theorem \ref{res:XXt}, 
$\mu_S$ converges a.s. to a probability measure $\nu$. 
The expressions for the moments of $\nu$ are quite involved and shall be given in Section \ref{proofs-smatrix}
after the proof of Theorem \ref{res:XXt}. Incidentally, from those expressions, it is evident that the contribution to the moment from distinct special symmetric partitions may be different even when they 
the same number of blocks and block sizes.
\vskip3pt 

Consider the special case $p=n$, with $X= ((y_{ij}= \sigma(i/n,j/n)x_{ij}))$ with $\sigma : [0,1]^2 \longrightarrow [0,1]$, $$\sigma(x,y)= \begin{cases}
1, & \ \  x \leq y,\\
0 & \ \ \text{ otherwise}
\end{cases}$$  and $\{x_{ij}; 1\leq i,j \leq n \}$ are i.i.d. with mean zero and variance 1. 
Then $X_p$ equals  
\begin{align}\label{triangular matrix}
X_p= \begin{bmatrix}
x_{11} & x_{12} & x_{13} & \cdots & x_{1n}\\
0 & x_{22} & x_{23} & \cdots & x_{2n}\\
\ & \ &   \  & \vdots  & \ &   \ \\
0 & 0 & 0 & \cdots & x_{nn} 
\end{bmatrix}.
\end{align}
The spectral distribution of $n^{-1}X_pX_p^{*}$ 
were studied in \citep{Dykema2002DToperatorsAD} 
in the Gaussian case. Later, \citep{basu2012spectral} studied the LSD of this and similar other models 
where the entries are i.i.d. with mean zero and variance 1.
If all moments of $\{x_{ij}\}$ are finite, then the  a.s. convergence of $\mu_S$ is immediate from Theorem \ref{res:XXt}. When only the variance is known to be finite, 
a truncation argument similar to that given in Section \ref{iid}, can be used.
As a consequence, 
$\mu_S$ converges a.s. to a non-random probability measure. 

\subsection{Proofs for the $S$ matrix}\label{proofs-smatrix}
As discussed in Section \ref{preliminaries-smatrix}, the existence of $\lim_{n\to\infty}\frac{| \Pi_S(\boldsymbol {\omega})|}{p^{r+1}n^{b-r}}$ is crucial in finding the LSD of the $S$ matrix. So we look into it first.

%
\begin{lemma}\label{lem:SS2k}
 Suppose $\boldsymbol {\omega}\in SS_b(2k)$ with $(r+1)$ even generating vertices $(0 \leq r \leq k)$. Then, $\big|\Pi_S(\boldsymbol {\omega}) \big| = p^{r+1}n^{b-r}$.\\
\end{lemma}
\begin{proof}
We argue by induction on $b$, the number of distinct letters.
If $b=1$, then $r=0$ and $\boldsymbol {\omega}=aa\cdots aa$. Therefore $\pi(0)$ and $\pi(1)$ are the generating vertices and both can be chosen freely. Thus, $\big|\Pi_S(\boldsymbol {\omega}) \big|=pn$. 
 Now assume that the result is true upto $b-1$.
Then it is enough to prove that if $\boldsymbol {\omega}$ has $b$ distinct letters with $(r+1),\ (0 \leq r \leq b-1)$ even generating vertices, then $\big|\Pi_S(\boldsymbol {\omega}) \big|=p^{r+1}n^{b-r}$.

First let $0 \leq r \leq b-2$. 
Suppose the last distinct letter of $\boldsymbol {\omega}$, say, $z$ appears for the first time at the $i$th position, that is at $(\pi(i-1),\pi(i))$ or $(\pi(i),\pi(i-1))$ (depending on whether $i$ is odd or even). Then 
$z$ appears in pure even blocks. Let $m$ ($m$ even) be the length of the first pure block.
Then we have the following two cases:
\vskip3pt

\noindent Case 1: $i$ is odd. Then we have
\begin{align}\label{even z}
& \pi(i-1)= \pi(i+1)= \cdots = \pi(i+m-1), \nonumber \\
& \pi(i)= \pi(i+2)= \cdots = \pi(i+m-2).
\end{align}
Similar identities can be shown for all other pure blocks of $z$. Hence $\pi(i)$ can be chosen freely with $1 \leq \pi(i) \leq n$ as it does not appear elsewhere in $\boldsymbol {\omega}$ other than the letter $z$. Let $\boldsymbol {\omega}^{\prime}$ be the word with $(b-1)$ distinct letters and $(r+1)$ even generating vertices, obtained by dropping all $z$s from $\boldsymbol{\omega}$. It is easy to see that  
$\boldsymbol {\omega}^{\prime}$ is a special symmetric word with $(b-1)$ distinct letters. Therefore, by induction hypothesis, $\big|\Pi_S(\boldsymbol {\omega}^{\prime}) \big|=p^{r+1}n^{b-(r+1)}$. Now as $\pi(i)$ is another odd vertex that can be chosen freely, we have $\big|\Pi_S(\boldsymbol {\omega}) \big|=p^{r+1}n^{b-(r+1)}n= p^{r+1}n^{b-r}$.
\vskip3pt

\noindent Case 2: $i$ is even. Then we have
\begin{align*}
& \pi(i-1)= \pi(i+1)= \cdots = \pi(i+m-1),\\
& \pi(i)= \pi(i+2)= \cdots = \pi(i+m-2).
\end{align*}
  As in Case 1, the generating vertex $\pi(i)$ can be chosen freely with $1 \leq \pi(i) \leq p$. As before, dropping all $z$s from $\boldsymbol{\omega}$  leads to a word special symmetric word$\boldsymbol {\omega}^{\prime}$ with $(b-1)$ distinct letters and $r$ even generating vertices. 
	Therefore, by induction hypothesis, $\big|\Pi_S(\boldsymbol {\omega}^{\prime}) \big|=p^{r}n^{b-r}$. Now as $\pi(i)$ is another even vertex that can be chosen freely, we have $\big|\Pi_S(\boldsymbol {\omega}) \big|=p^{r}p n^{b-r}= p^{r+1}n^{b-r}$. 
    \vskip3pt
  
\noindent Now let $r=b-1$. Then there are $r+1=b$ even generating vertices (one of them being $\pi(0)$) and $b$ distinct letters in $\boldsymbol{\omega}$. Therefore all letters except the first appear for the first time at even positions in $\boldsymbol{\omega}$. So, if $z$ is the last distinct letter of $\boldsymbol{\omega}$, then $z$ appears for the first time at $(\pi(i-1),\pi(i))$ where $i$ is even. Thus similar to  Case 1, \eqref{even z} holds and $\pi(i)$ can be chosen freely with $1 \leq \pi(i) \leq n$. If we drop all $z$s as before from $\boldsymbol{\omega}$, then we get a special symmetric word $\boldsymbol {\omega}^{\prime}$ with $(b-1)$ distinct letters and $(b-2)$ even generating vertices. 
Therefore, $\big|\Pi_S(\boldsymbol {\omega}^{\prime}) \big|=p^{b-1}n^{b-(b-1)}$. As $\pi(i)$ is another even vertex that can be chosen freely, we have $\big|\Pi_S(\boldsymbol {\omega}) \big|=p^{b-1}n p= p^{b} n= p^{r+1}n^{b-r}$, $r=b-1$.  
This completes the proof of the lemma.
\end{proof}

\begin{lemma}\label{S-words}
Let $\boldsymbol{\omega}$ be a word with $b$ distinct letters and $(r+1)$ even generating vertices $ (0 \leq r \leq b-1)$. Then
\begin{equation}\label{limit S-words}
\lim_{n \rightarrow \infty}\frac{| \Pi_S(\boldsymbol {\omega})|}{n^{b+1}}= \begin{cases}
y^{r+1}, & \ \  \omega\in SS_b(2k)\\
0, & \ \  \omega \notin  SS_b(2k).
\end{cases}
\end{equation}
Thus, $\lim_{n \rightarrow \infty}\frac{| \Pi_S(\boldsymbol {\omega})|}{p^{r+1}n^{b-r}}= 1$ if and only if $\boldsymbol{\omega}$ is a special symmetric word.
\end{lemma}

\begin{proof}
 First suppose $\boldsymbol {\omega} \in \mathcal{P}(2k) \setminus SS_b(2k)$. Then from \eqref{W-S} and Lemma 3.1 and 3.3 in \citep{bose2022random}, it is easy to see that $$\lim_{n \rightarrow \infty}\frac{| \Pi_S(\boldsymbol {\omega})|}{n^{b+1}}=0.$$ If $\boldsymbol{\omega} \in SS_b(2k)$, then from Lemma \ref{lem:SS2k}, it immediately follows that $\lim_{n \rightarrow \infty}\frac{| \Pi_S(\boldsymbol {\omega})|}{n^{b+1}}= y^{r+1}$. This completes the proof of the lemma.
\end{proof}

\begin{lemma}\label{lem:moment}
Let 
\begin{align}\label{def-Q}
Q_{k,4}^b = | \{& (\pi_1,\pi_2,\pi_3,\pi_4): \ell(\pi_i)=2k; \pi_i, 1 \leq i \leq 4 \ \text{jointly- and cross-matched with }\\
 & b \text{ distinct edges or } b \text{ distinct letters across all } (\pi_i)_{1\leq i \leq 4}\}|.
\end{align}
 Then there exists a constant C, such that,
\begin{equation}
Q_{k,4}^b \leq  C n^{b +2}\ .
\end{equation} 
\end{lemma}
This was proved for the Wigner link function in Lemma 4.2 in \citep{bose2022random}. The arguments in that proof can be used for the $S$ link function here as $1 \leq \pi(2i)\leq p$ and $1 \leq \pi(2i-1)\leq n$, and $p$ and $n$ are comparable for large $n$. We omit the details.

\subsubsection{Proof of Theorem \ref{res:XXt} and Remark \ref{p=n}}\label{proof of theorem and remark}

\begin{proof}[\textit{Proof of Theorem \ref{res:XXt} }]
 (a) We make use of Lemma \ref{lem:genmoment} and use the notion of words and circuits in order to calculate the moments. We break the proof into a few steps.

\vskip5pt
\noindent\textbf{Step 1:} (Reduction to mean zero)  Consider the zero mean matrix 
$\widetilde{Z}=((y_{ij}-\mathbb{E}y_{ij}))$. Now
\begin{align}\label{meanzero-noiid}
n\ \mathbb{E}[(y_{ij}- \mathbb{E} y_{ij})^{2k}]= n\ \mathbb{E}[y_{ij}^{2k}] + n \displaystyle \sum_{t=0}^{2k-1} {{2k} \choose {t}}\mathbb{E}[y_{ij}^{t}]\ (\mathbb{E}y_{ij})^{2k-t}.  
\end{align}
The first term of the r.h.s. equals $g_{2k}(i/p,j/n)$ by \eqref{gkeven}. The second term is tackled as follows:
\begin{align*}
\text{For } t\neq {2k-1}, \ \ n \ \mathbb{E} [y_{ij}^{t}]\ (\mathbb{E}y_{ij})^{2k-t}& = (n^{\frac{1}{2k-t}}\ \mathbb{E}y_{ij})^{2k-t} \ \mathbb{E}[y_{ij}^{t}]\\
&\overset{n \rightarrow \infty}{\longrightarrow} 0, \ \ \ \text{ by condition } \eqref{gkodd}.  
\end{align*}
\begin{align*}
\text{For } t={2k-1}, \ \ n \ \mathbb{E} [y_{ij}^{2k-1}]\ \mathbb{E}y_{ij}& = (\sqrt{n} \ \mathbb{E} [y_{ij}^{2k-1}]) \ (\sqrt{n}\  \mathbb{E}y_{ij})\\
&\overset{n \rightarrow \infty}{\longrightarrow} 0, \ \ \ \text{ by condition } \eqref{gkodd}.
\end{align*}
Hence from \eqref{meanzero-noiid}, we see condition \eqref{gkeven} is true for the matrix $\widetilde{Z}$. Similarly we can show that \eqref{gkodd} is true for $\widetilde{Z}$. Hence, Assumption A holds for the matrix $\widetilde{Z}$.

Now from Lemma \ref{lem:metric}, 
\begin{align}\label{ineq-S}
L^4\big(F^{S_{Z}}, F^{S_{\widetilde{Z}}}\big) & \leq  \frac{2}{p^2}(\Tr(ZZ^T+\widetilde{Z}\widetilde{Z}^T))(\Tr[(Z-\widetilde{Z})(Z-\widetilde{Z})^T]) \nonumber \\
&\leq \frac{2}{p}\bigg(\displaystyle \sum_{i,j}\big( 2y_{ij}^2+ (\mathbb{E}y_{ij})^2- 2y_{ij}\mathbb{E}y_{ij}\big)\bigg)\bigg(\frac{1}{p}\sum_{i,j}(\mathbb{E}y_{ij})^2\bigg).
\end{align}
The second factor of the rhs in \eqref{ineq-S} is bounded by $$n (\sup_{i,j}\mathbb{E}y_{ij})^2=(\sup_{i,j}\sqrt{n}\mathbb{E}y_{ij})^2\rightarrow 0  \ \ \text{ as } n \rightarrow \infty \ \ \text{ by } \eqref{gkodd}. $$ Now it can be seen applying Borel-Cantelli lemma that $\frac{1}{p}\displaystyle \sum_{i,j} (y_{ij}^2-\mathbb{E}[y_{ij}^2]) \rightarrow 0$ a.s. as $p \rightarrow \infty$ (proof is given in Section \ref{appendix}). Also $\mathbb{E}\big[\frac{1}{p}\sum_{ij}y_{ij}^2\big] \rightarrow \int_{[0,1]^2} g_2(x,y)\ dx\ dy$. Hence, $$\mathbb{P}\big[\{\omega; \limsup_p\frac{1}{p}\displaystyle \sum_{i,j}y_{ij}^2(\omega)= \infty\}\big]=0.$$ Therefore the first term of the rhs in \eqref{ineq-S} also tends to zero a.s.
Hence,  the LSD of $S_{Z}$ and $S_{\widetilde{Z}}$ are same a.s. Thus we can assume that the entries of $Z$ have mean 0.

 We will now verify the conditions (i), (ii) and (iii) of Lemma \ref{lem:genmoment}. 
\vskip5pt
\noindent\textbf{Step 2:} (Verification of the fourth moment condition for $S_{Z}$) Observe that 
 \begin{align}\label{fourthmoment-XXt}
 \frac{1}{p^{4}} \mathbb{E}[\Tr(ZZ^T)^k \ - \ \mathbb{E}(\Tr(ZZ^T)^k)]^4\ = \frac{1}{p^4} \displaystyle \sum_{\pi_1,\pi_2,\pi_3,\pi_4} \mathbb{E}[\displaystyle \Pi_{i=1}^4 (Y_{\pi_i}\ - \ \mathbb{E}Y_{\pi_i})].
 \end{align}
 If $(\pi_1,\pi_2,\pi_3,\pi_4)$ are not jointly-matched, then one of the circuits has a letter that does not appear elsewhere. 
	Hence by independence and mean zero assumption, $\mathbb{E}[\displaystyle \Pi_{i=1}^4 (Y_{\pi_i}\ - \ \mathbb{E}Y_{\pi_i})]=0$.
  If $(\pi_1,\pi_2,\pi_3,\pi_4)$ are not cross-matched, then one of the circuits say $\pi_j$ is only self-matched. Then we have $\mathbb{E}[Y_{\pi_j}\ - \ \mathbb{E}Y_{\pi_j}]=0$. So again we have  $\mathbb{E}[\displaystyle \Pi_{i=1}^4 (Y_{\pi_i}\ - \ \mathbb{E}Y_{\pi_i})]=0$.

So we consider only circuits $(\pi_1,\pi_2,\pi_3,\pi_4)$ that are jointly- and cross-matched. Here each circuit is of length $2k$, so the total number of edges ($L-$ values) is $8k$. As the circuits are at least pair-matched, the number of distinct edges is at most $4k$. 

Suppose $\pi_i$ has $k_i$ distinct letters, $1\leq i \leq 4$ with $k_1+k_2+k_3+k_4=b$. Suppose the $j$th distinct letter appears $s_j$ times across $\pi_1,\pi_2,\pi_3,\pi_4$ and first at the $i_j-$th position.
Let $b_1$ and $b_2$ ($b_1+b_2=b$) be respectively the number of even and odd $s_i$'s, denoted by $s_{i_1},s_{i_2},\ldots,s_{i_{b_1}}$and $s_{i_{b_1+1}},s_{i_{b_1+2}},\ldots,s_{i_{b_2}}$.  Each term can then be written as 
\begin{align*}
 \frac{1}{p^4}  \displaystyle \sum_{b=1}^{4k} p^{-{b_1}} p^{-(b_2-\frac{1}{2})}
 \prod_{j=1}^{b_1} \  g_{s_{i_j},n}(\pi(i_j-1)/p,\pi(i_j)/n) \ \prod_{m=b_1+1}^{b_1+b_2} n^{\frac{b_2-(1-1/2)}{b_2}} \mathbb{E}[y_{\pi(i_{m}-1)\pi(i_{m})}^{s_{i_m}}]. 
 \end{align*}

We note that $g_{s_{i_j},n}   \rightarrow  g_{s_{i_j}}$  for all $1\leq j \leq b_1$. Therefore, the sequence $\|g_{s_{i_j},n}\|$ is bounded by  a constant $M_{j}$. Also as $\frac{b_2-(1-1/2)}{b_2}<1$, by \eqref{gkodd}, we have $n^{\frac{b_2-(1-1/2)}{b_2}} \mathbb{E}[y_{\pi(i_{m}-1)\pi(i_{m})}^{s_{i_m}}]  $ is bounded by $1$ for $n$ large when $b_1+1\leq m \leq b_1+b_2$. Let
$$M^{\prime} = \underset{b_1+b_2=b}{\max} \{M_{t},1:  1 \leq t \leq b_1 \}\ \mbox{ and }\ M_0^{\prime}=  \max \{{M^{\prime}}^b:  1 \leq b \leq 2k \}. 
$$
By \eqref{W-S} and Lemma \ref{lem:moment}, the number of such circuits that have $b$ distinct letters ($b=1,\ldots, k$) is bounded by $C_1n^{b+2}$ for some constant $C_1>0$. Therefore with $C_2=y^{b+2}C_1$,
 \begin{align*}
\frac{1}{p^{4}} \mathbb{E}[\Tr(ZZ^T)^k \ - \ \mathbb{E}(\Tr(ZZ^T)^k)]^4\
& \leq C_2 M_0^{\prime} \displaystyle \sum_{b=1}^{4k} \frac{1}{p^{b+3\frac{1}{2}}} p^{b+2} 
= \mathcal{O}(p^{-\frac{3}{2}}).
\end{align*}
This completes  the  proof of Step 2.\vskip3pt

\noindent\textbf{Step 3:} (Verification of  the first moment condition (i) of Lemma \ref{lem:genmoment})  
By Lemma \ref{lem:genmoment} and the previous step, it is now enough to show  that for every $k\geq 1$, 
   $\displaystyle \lim_{n\to\infty}\frac{1}{p}\mathbb{E}[\Tr(ZZ^T)^{k}]$ exists and is given by $\beta_k( \mu^{\prime})$ for each $k \geq 1$. 
First note that, we can write \eqref{genmoment-XXt} as \begin{align}\label{moment-XXt1}
\lim_{n \rightarrow \infty}\frac{1}{p}\mathbb{E}[\Tr(ZZ^T)^k]
 &=   \displaystyle\lim_{n \rightarrow \infty} \displaystyle \sum_{b=1}^k  \Big[\frac{1}{p} \sum_{\omega \in SS_b(2k)} \sum_{\pi \in \Pi(\boldsymbol{\omega})}\ \mathbb{E}(Y_{\pi}) + \frac{1}{n} \sum_{\underset{\boldsymbol {\omega} \text{ with b letters}}{\omega \notin SS(2k)}}  \sum_{\pi \in \Pi(\boldsymbol{\omega})}\ \mathbb{E}(Y_{\pi})\Big] .\nonumber \\
& = T_1+T_2.
\end{align}
Suppose that $\boldsymbol {\omega}$ has $b$ distinct letters and let $\pi \in \Pi_S(\boldsymbol {\omega})$. Suppose the 
$j$th new letter appears at the $(\pi(i_j-1),\pi(i_j))-$th position for the first time, $1\leq j\leq b$. Let $D$ denote the set of all distinct generating vertices. Thus $|D|\leq (b+1)$.

Suppose $\boldsymbol {\omega}$ has $b$ distinct letters but does not belong to $SS(2k)$.
Then from Lemma \ref{S-words},  
$|D|\leq b$. Hence $\boldsymbol {\omega}$, and as a consequence, $T_2$  has no contribution to \eqref{moment-XXt1}.

Now suppose $\boldsymbol {\omega} \in SS_b(2k)$ with $(r+1)$ even generating vertices. By Lemma \ref{S-words}, $\boldsymbol {\omega}$ has $(b+1)$ distinct generating vertices. For each $j \in \{1,2,\ldots , b\}$ denote $(\pi(i_j-1),\pi(i_j))$ as $(t_j,l_j)$. Then $t_1=\pi(0)$ and $l_1=\pi(1)$. It is easy to see that any distinct $(t_j,l_j)$ corresponds to a distinct letter in $\boldsymbol {\omega}$. Suppose the $j$th new letter appears $s_j$ times in $\boldsymbol {\omega}$. Clearly all the $s_j$ are even. So the total contribution of this $\boldsymbol {\omega}$ to $T_1$ in \eqref{moment-XXt1} is: 
\begin{align}\label{finitesum-XXt}
\frac{1}{pn^b} \displaystyle \sum_{S} \prod_{j=1}^b g_{s_j,n}(t_j/p,l_j/n)
\end{align} 
Recall that there are $(r+1)$ even generating vertices in $D$ with range between $1$ and $p$, and $(b-r)$ vertices (odd generating) with range between $1$ and $n$. So as $n \rightarrow \infty$, \eqref{finitesum-XXt} converges to
\begin{equation}\label{ss2k-limit}
y^r \int_{[0,1]^{b+1}} \displaystyle \prod_{j=1}^b g_{k_j}(x_{t_j},x_{l_j})\  
\prod_{i \in S} dx_{i}.
\end{equation}
Hence we obtain
\begin{align}\label{moment-XXt}
\displaystyle \lim_{p \rightarrow \infty} \frac{1}{p}\mathbb{E}[\Tr S^k]= \sum_{b=1}^k\sum_{r=0}^{b-1} \sum_{\underset{\underset{\text{ even generating vertices}}{\text{with }(r+1)}}{\pi \in SS_b(2k)}} y^r \int_{[0,1]^{b+1}} \displaystyle \prod_{j=1}^b g_{k_j}(x_{t_j},x_{l_j})\ \prod_{i \in S} dx_{i}.
 \end{align}
 
This completes the verification of the first moment condition.
 \vskip5pt
\noindent \textbf{Step 4:} (Uniqueness of the measure)  We have obtained 
\begin{align*}
\gamma_{k}= &\lim_{p \rightarrow \infty}\frac{1}{p}\mathbb{E}[\Tr(S)^{k}]\leq \sum_{b=1}^k\sum_{r=0}^{b-1} \sum_{\underset{\underset{\text{ even generating vertices}}{\text{with }(r+1)}}{\sigma \in SS_b(2k)}} y^r M_{\sigma}
\end{align*}
Let $c=\max(y,1)$. Then 
\begin{align*}
\gamma_{k} \leq \displaystyle \sum_{\sigma \in SS(2k)} c^{k}M_{\sigma}
 \leq \displaystyle \sum_{\sigma \in \mathcal{P}(2k)} c^k M_{\sigma} 
 =c^k \alpha_{k}.
\end{align*}
 As $\{\alpha_{k}\}$ satisfies Carleman's condition, $\{\gamma_{k}\}$ also does so. By Lemma \ref{lem:genmoment}, we see that there exists a measure $ \mu$ with  moments $\{\gamma_{k}\}_{k \geq 1}$ such that $\mu_{S_{Z}}$ converges a.s. to $ \mu$.
  This completes the proof of part (a).
\vskip5pt
 
\noindent(b)  From Lemma \ref{lem:metric}, we have 
\begin{align}\label{levy-XXt}
L^4(F^{S},F^{S_{Z}}) &\leq \frac{2}{p^2}(\Tr(XX^T+ZZ^T))(\Tr[(X-Z)(X-Z)^T]) \nonumber\\
& = \frac{2}{p} \bigg(\displaystyle2 \sum_{i,j}y_{ij}^2+\sum_{ij}x_{ij}^2\boldsymbol{1}_{[|x_{ij}|>t_n]} \bigg)\bigg(\frac{1}{p}\sum_{ij}x_{ij}^2\boldsymbol{1}_{[|x_{ij}|>t_n]}\bigg).
\end{align}
The second factor in the above equation tends to zero a.s. (or in probability) as $n \rightarrow \infty$ due to the condition 
(\ref{truncation-ineq}).
Now $\displaystyle\frac{1}{p} \sum_{i,j} (y_{ij}^2-\mathbb{E}[y_{ij}^2]) \rightarrow 0 $ a.s. (see Section \ref{appendix}) and $\mathbb{E}\big[\frac{1}{p}\sum_{ij}y_{ij}^2\big] \rightarrow \int_{[0,1]^2} g_2(x,y)\ dx\ dy$, and hence is finite. This implies that $\frac{1}{p}\sum_{i,j}y_{ij}^2$ is bounded a.s. 
 Therefore the first factor in \eqref{levy-XXt} is bounded a.s. and thus the rhs of \eqref{levy-XXt} tends to 0 as $p,n\rightarrow \infty$.

 From the discussion above, we infer that the ESD of $XX^T$ converges to the probability measure $\mu$ a.s. (or in probability).
This completes the proof of the theorem. 
\end{proof}
\begin{proof}[\textbf{Proof of Remark \ref{p=n}}]
In the case $p=n$, if $\{g_{2k,n}\}$ are symmetric functions, then the assumption on the entries of $X_n$ are no different from that on the entries of $W_n$ in Theorem 2.1 of \citep{bose2022random}. Now from \eqref{moment-XXt}, and equation (4.11) in \citep{bose2022random} we see that $\mathbb{E}[Y^k]=\mathbb{E}[Y^{\prime 2k}], k \geq 1$. 

However observe that even though $\{g_{2k,n}\}$ are not symmetric for every $n$, the  functions $\{g_{2k}\}$ are symmetric and hence $\mathbb{E}[Y^k]=\mathbb{E}[Y^{\prime 2k}], k \geq 1$ still holds, (see \eqref{moment-XXt}) as the limiting moments depend only on $\{g_{2k}\}$. Therefore, by the uniqueness criterion of a probability distribution via moments, we have $Y \overset{\mathcal{D}}{=} Y^{\prime 2}$. As the limiting moments (\eqref{moment-XXt}) depend on $\{g_{2k}\}$ and $\lim p/n$, if $p/n \rightarrow 1$, and $\{g_{2k}\}$ are symmetric, we have $Y \overset{\mathcal{D}}{=} Y^{\prime 2}$.
\end{proof}

\begin{proof}[Proof of Remark \ref{unbounded support}]
Consider $k=mt$ for some $t \geq 1$. Then from \eqref{moment-XXt}, we have 
\begin{align}\label{moment-unbounded support1}
\beta_{k}(\mu)= \sum_{b=1}^k\sum_{r=0}^{b-1} \sum_{\underset{\underset{\text{ even generating vertices}}{\text{with }(r+1)}}{\pi \in SS_b(2k)}} y^r \int_{[0,1]^{b+1}} \displaystyle \prod_{j=1}^b g_{k_j}(x_{t_j},x_{l_j})\  \prod_{i \in S} dx_{i}.
\end{align}
Recall that $\pi$ in the above expression could be described as a word in $SS_{b}(2k)$ with $(r+1)$ even generating vertices. Let us focus on words $\omega\in SS_t(2k)$ with $t$ distinct letters and where each letter appears $2m$ times in pure even blocks. Clearly $\omega$ has only one even generating vertex $\pi(0)$. Therefore as $n \rightarrow \infty$, the contribution of $\omega$ in the limiting moment is (see \eqref{ss2k-limit}):
\begin{equation}\label{ss2k-unbounded}
\int_{[0,1]^{t+1}} g_{2m}(x_0,x_1)g_{2m}(x_0,x_2)\cdots g_{2m}(x_0,x_t) \ dx_0dx_1\cdots dx_t= \int_{[0,1]} \big(f_{2m}(x_0)\big)^t \ dx_0 .
\end{equation}
Next, observe that the number of such words 
\begin{equation}\label{number of such words}
\frac{1}{t!} {{mt} \choose m}{{mt-t} \choose m}\cdots {m \choose m}= \frac{1}{t!}\frac{(mt)!}{(m!)^t}. 
\end{equation}
Since the integrand in \eqref{moment-unbounded support1} is non-negative, 
using \eqref{ss2k-unbounded} and \eqref{number of such words}, we have
\begin{align*} 
\beta_{k}(\mu)& > \frac{1}{t!}\frac{(mt)!}{(m!)^t} \int_{[0,1]} \big(f_{2m}(x_0)\big)^t \ dx_0 . \nonumber\\
& = \frac{(mt)!}{t!} \int_{[0,1]}\bigg(\frac{f_{2m}(x_0)}{m!}\bigg)^t \ dx_0 > c \frac{(mt)!}{t!} , \ k=mt.
\end{align*}
Therefor for $t$ sufficiently large (with $k=mt$), 
\begin{equation*}
(\beta_k(\mu))^{1/k}> K \ t^{\eta} \ \ \text{ for some constant } K>0 \  \text{ and } \eta>0.
\end{equation*}
Therefore $(\beta_k(\mu))^{1/k} \rightarrow \infty$ as $k=mt \rightarrow \infty$. Hence $\mu$ has unbounded support.
\end{proof}

\noindent \textbf{Moments of the variance profile matrices}: Now we give a description of the moments of LSD $\nu$ for $S$ matrices with variance profile. From Step 3 in the proof of Theorem \ref{res:XXt}, for each word in $SS_{b}(2k)$ with $(r+1)$ even generating vertices and where the distinct letters appear $s_1,s_2,\ldots,s_b$ times, its contribution to the limiting moments is (see \eqref{ss2k-limit})
$$y^r\int_{[0,1]^{b+1}} \prod_{j=1}^b \sigma^{s_j}(x_{t_j},x_{l_j}) \ \prod_{i\in S} dx_i \prod_{j=1}^b C_{s_j},$$
where $(t_j,l_j)$ denotes the position of the first appearance of the $j$th distinct letter in the word.
Hence the $k$th moment of $\nu$ is
\begin{align*}
\beta_{k}(\nu)= \sum_{b=1}^k\sum_{r=0}^{b-1} \sum_{\underset{\underset{\text{ even generating vertices}}{\text{with }(r+1)}}{\pi \in SS_b(2k)}} y^r \int_{[0,1]^{b+1}} \displaystyle \prod_{j=1}^b \sigma^{s_j}(x_{t_j},x_{l_j}) dx_i \prod_{j=1}^b C_{s_j}.
\end{align*}

\subsection{Hypergraphs, Noiry-words and $SS(2k)$}\label{connection-ss2k}
The distribution of the $S$ matrix with triangular i.i.d. entries, was studied in \citep{Benaych-Georges2012} and \citep{noiry2018spectral}, where the authors used the concepts of \textit{Hypergraphs} and \textit{words}, which we call \textit{Noiry words} here. In Section \ref{triangular iid}, we discussed the triangular i.i.d. cases and showed how Theorem \ref{res:XXt} is used in this situation. We also described the limiting moments via $SS(2k)$ partitions. Now we verify that the moments that we have obtained are identical with those obtained in \citep{Benaych-Georges2012} and \citep{noiry2018spectral}. 

\begin{definition}
Let $G$ be a graph with vertex set $V$. Let $\pi$ and $\tau$  be partitions, respectively, of $V$ and the edge set. Then the hypergraph $H(\pi,\tau)$ is a graph with vertex set $G_{\pi}$ (i.e. $\pi$) and edges $\{E_W; W\in \tau\}$, where each edge $E_W$ is the set of blocks $J \in \pi$ such that at least one edge of $G_{\pi}$ starting or ending at $J$ belongs to $W$.
Further if no two of the edges can have more than one common vertex, then $H(\pi,\tau)$ is said to be a \textit{hypergraph with no cycles}.
\end{definition}
For details on Hypergraphs, see Sections 5.3 and 12.3.2 in \citep{Benaych-Georges2012} and \citep{berge1984hypergraphs}. In \citep{Benaych-Georges2012}, their equation (22) describes the moments of the LSD of $S$ as a sum on \textit{Hypergraphs with no cycle}. 
\begin{lemma}\label{hypergraph-ss2k}
For every word $\boldsymbol{\omega} \in SS_b(2k)$, there exists partitions $\pi,\tau \in \mathcal{P}(k)$ such that there is a unique hypergraph $H(\pi,\tau)$ which has no cycle with $|\pi|+|\tau|=b+1$. The converse is also true.
\end{lemma}
\begin{proof}
Let $\boldsymbol{\omega} \in SS_b(2k)$ with $(r+1)$ and $(b-r)$ even and odd generating vertices respectively.
Suppose the even and the odd generating vertices are respectively $\pi(i_{t_0})=\pi(0),\pi(i_{t_1}),\ldots,\pi(i_{t_r})$ and 
$\pi(i_{m_1})=\pi(1),\pi(i_{m_2}),\ldots ,\pi(i_{m_{b-r}})$.
Let $V_j= \{\pi(2i): \pi(2i)=\pi(i_{t_j}), 1 \leq i \leq k\}$, $0 \leq j \leq r$ and $W_j= \{\pi(2i-1): \pi(2i-1)=\pi(i_{m_j}), 1 \leq i \leq k\}$, $1 \leq j \leq (b-r)$. Clearly, $\sigma=\{V_j; 0 \leq j \leq r\}$ and $\tau=\{W_j; 1 \leq j \leq (b-r)\}$ are two partitions of $\{1,2,\ldots,k\}$. Therefore, we can construct a hypergraph $H(\sigma,\tau)$ where $\sigma$ is the vertex set and $\{E_W;W\in \tau\}$ is the edge set (see (\ref{hypergraph-ss2k})). 

Now suppose if possible, $H(\sigma,\tau)$ has a cycle. That means by construction, there exists $a,b \ (a\neq b)\in \{1,2,\ldots, (b-r)\}$ and $q,l \ (q\neq l)\in \{0,1, \ldots,r\}$ such that $V_q, V_l \in W_a \cap W_b$. That is, there are edges $(\pi(k_1-1),\pi(k_1)),(\pi(k_2-1),\pi(k_2)),(\pi(k_3-1),\pi(k_3)),(\pi(k_4-1),\pi(k_4))$ with $k_i,1\leq i \leq 4$ odd such that $\pi(k_1-1) \in V_q, \pi(k_1)\in W_a$, $\pi(k_2-1) \in V_q, \pi(k_2)\in W_b$, $\pi(k_3-1) \in V_l, \pi(k_3)\in W_a$ and $\pi(k_4-1) \in V_l, \pi(k_4)\in W_b$. As the positions $(\pi(k_i-1),\pi(k_i),i=1,2,3,4$ are all distinct, there are four distinct letters that appear at these four positions in $\boldsymbol{\omega}$. Without loss of generality suppose, from left to right $(\pi(k_4-1),\pi(k_4))$ is the rightmost (among the four positions mentioned above) in $\boldsymbol{\omega}$. Since $\pi(k_4-1)\in V_l$ and $\pi(t_l)$ comes before $\pi(k_4-1)$, it cannot be chosen freely. Using a similar argument,  $\pi(k_4)$ also cannot be chosen freely. Also they have been chosen as generating vertices of three different letters that have appeared in the positions $(\pi(k_i-1),\pi(k_i)), 1 \leq i \leq 3$. Using Lemma \ref{S-words}, this is not possible as the letter at $(\pi(k_4-1),\pi(k_4))$ is different from the previous three letters. Thus 
$H(\sigma,\tau)$  does not have a cycle. 
Moreover, it is evident by construction that every special symmetric word we get a unique $H(\sigma,\tau)$ without any cycles.\vskip5pt

\noindent Conversely, suppose $H(\sigma,\tau)$ is a hypergraph with no cycle and $|\sigma|+|\tau|=b+1$. We form a word of length $2k$ from it in the following manner. Now $\sigma ,\tau \in \mathcal{P}(k)$. Let $\sigma=\{V_0,V_1,\ldots,V_r\}$ and $\tau=\{W_1,\ldots,W_{b-r}\}$ (as $|\sigma|+|\tau|=b+1$). Then we choose the even vertices $\pi(2i),0 \leq i\leq k-1$ from $\sigma$ and odd vertices $\pi(2i-1), 1 \leq i \leq k$ from $\tau$ and $\pi(i)=\pi(j)$ if $i$ and $j$ belong to the same block of $\sigma$ or $\tau$ (depending on $i$ and $j$ both being even or odd respectively). Thus we get a word $\boldsymbol{\omega}$ of length $2k$ whose even and odd generating vertices are $\{\pi(\min \{V_s\})\}_{0 \leq s \leq r}$ and $\{\pi(\min \{W_t\})\}_{1 \leq t \leq (b-r)}$ respectively. Thus there are $b$ distinct letters in $\boldsymbol{\omega}$. Now as $H(\sigma,\tau)$ does not have a cycle, using the same arguments as the previous paragraph, it can be shown that all the generating vertices can be chosen freely. This can happen only if the word is special symmetric. Thus we obtain $\boldsymbol{\omega} \in SS_b(2k)$ with $(r+1)$ even generating vertices. It is easy to see that two hypergaphs with no cycle cannot give rise to the same special symmetric word. 

Hence there is a one-one correspondence between special symmetric words and hypergraphs with no cycles. This completes the proof of this lemma.
\end{proof}

Thus we see that \eqref{Zak-limit} can be written as 
\begin{align}\label{equality-hypergraph}
\beta_{k}(\mu)& =\displaystyle \sum_{r=0}^{k-1}\sum_{\underset{\underset{\text{ even generating vertices}}{\text{with }(r+1)}}{\pi \in SS(2k)}} y^r C_{\pi} \nonumber \\
& = \displaystyle \sum_{r=0}^{k-1} \sum_{\underset{\text{with }r+1 \text{ blocks}}{\sigma\in \mathcal{P}(k)}} \sum_{\underset{H(\sigma,\tau) \text{ has no cycle}}{\tau \in \mathcal{P}(k)}} \prod_{i=1}^{b-r} y^{\frac{r}{b-r}} f(W_i)
\end{align} 
where $W_i$ are the blocks of $\tau$ and $f$ is some function determined by $(C_{2k})_{k\geq 1}$ (and not necessarily multiplicative in the sense of partitions.) Therefore, when the entries satisfy \eqref{ck}, using \eqref{equality-hypergraph} and Remark \ref{unbounded support}, we obtain the conclusions of Theorem 3.2 of \citep{Benaych-Georges2012}. As this is a special case (as the entries are iid for every $n$) of our result, Theorem \ref{res:XXt}, it indeed generalises Theorem 3.2 of \citep{Benaych-Georges2012}.\vskip5pt

In Proposition 3.1 in \citep{noiry2018spectral}, the author describes the limiting moments via equivalence class of words. His notion of words is different from ours  and so we call the former \textit{Noiry words}. 
  \vskip5pt

\noindent \textbf{Noiry words}:  Suppose $G=(V,E)$ is a graph with labelled  vertices. A word of length $k\geq 1$ on $G$ is a sequence of labels $i_1,i_2,\ldots,i_k$ such that for each $j\in \{1,2,\ldots,k-1\}$, $\{i_j,i_{j+1}\}$ is a pair of adjacent labels, i.e., the associated vetrices are neighbours in $G$. A word of length $k$ is closed if $i_1=i_k$. Such closed words will be called Noiry words. See Section 3 in \citep{noiry2018spectral} for more details.\\

\noindent \textbf{Equivalence of Noiry words}: Let $\boldsymbol{i}=i_1,i_2,\ldots,i_k$ and $\boldsymbol{i^{\prime}}=i_1^{\prime},i_2^{\prime},\ldots,i_k^{\prime}$ be two Noiry words on two labeled graphs $G$ and $G^{\prime}$ with vertex set $V$. These words are said to be equivalent if there is a bijection $\sigma$ of $\{1,2,\ldots,|V|\}$ such that $\sigma(i_j)=i_j^{\prime},1 \leq j \leq k$. This defines an equivalence relation on the set of all Noiry words, thereby giving rise to equivalence classes of Noiry words. 

Using the developments in Section 3 and equation (3.2) of \citep{noiry2018spectral})$\mathbf{W}_k(a,a+1,l,\boldsymbol{b}), \boldsymbol{b}=(b_1,b_2,\ldots,b_a)\in \mathbb{N}^a, b_i\geq 2, \displaystyle \sum_{i=1}^a b_i=2k$,  denotes an equivalence class of Noiry words on a labeled rooted planar tree with $a$ edges, of which $l$ are odd and each edge is traversed $b_i$ times, $1 \leq i \leq a$. Then the $k$th moment of the LSD is given in equation (3.2) of \citep{noiry2018spectral} 
as   
\begin{equation*}
\beta_k(\mu)= \displaystyle \sum_{a=1}^k\sum_{l=1}^a \alpha^{l} \sum_{\underset{b_i\geq 2,b_1+\cdots+ b_{a}=2k}{\boldsymbol{b}=(b_1,b_2,\ldots,b_a)}}|\mathbf{W}_k(a,a+1,l,\boldsymbol{b})| \prod_{i=1}^a C_{b_i}
\end{equation*}

In the next lemma we show how each of these equivalence classes of words correspond to special symmetric words.

\begin{lemma}
Each equivalence class $\mathbf{W}_k(a,a+1,l,\boldsymbol{b}),\boldsymbol{b}=(b_1,b_2,\ldots,b_a)\in \mathbb{N}^a, b_i\geq 2, \displaystyle \sum_{i=1}^a b_i=2k$ is a word $\boldsymbol{\omega} \in SS_a(2k)$ with $l$ odd generating vertices and where each letter appears $b_i, 1\leq i \leq a$ times in $\boldsymbol{\omega}$. 
\end{lemma}
\begin{proof}
Recall from Section \ref{preliminaries-smatrix} that we have defined words to be equivalence classes of circuits with the relation arising from the link functions (see \eqref{S-link}). Now Noiry words are not equivalence classes to begin with, they form equivalence classes if they are relabeled in a certain way as described above. From this and how we have defined equivalence of circuits, observe that an equivalence class of Noiry words is nothing but a word in our case.  Now the only words  with $a$ distinct letters for which $a+1$ generating vertices can be chosen freely are the special symmetric words with $a$ distinct letters (see Lemma \ref{S-words}). Thus $\mathbf{W}_k(a,a+1,l,\boldsymbol{b}),\boldsymbol{b}=(b_1,b_2,\ldots,b_a)\in \mathbb{N}^a, b_i\geq 2, \displaystyle \sum_{i=1}^a b_i=2k$ is a word $\boldsymbol{\omega} \in SS_a(2k)$ with $l$ odd generating vertices where each letter appears $b_i, 1\leq i \leq a$ times in $\boldsymbol{\omega}$.
\end{proof}

Using this lemma it readily follows that 
\begin{align*}
\displaystyle\sum_{a=1}^k \sum_{l=1}^a \sum_{\underset{b_i\geq 2,b_1+\cdots+ b_{a}=2k}{\boldsymbol{b}=(b_1,b_2,\ldots,b_a)}}|\mathbf{W}_k(a,a+1,l,\boldsymbol{b})|  =  \displaystyle \sum_{a=1}^k \sum_{l=1}^{a}\sum_{\underset{\underset{\text{ with block sizes } b_1,\ldots,b_a }{\text{with }l \text{ odd generating vertices}}}{\pi \in SS_a(2k)}} 1 =  \displaystyle \sum_{l=1}^k \sum_{\underset{\text{with }l \text{ odd generating vertices}}{\pi \in SS(2k)}} 1.
\end{align*}
Hence it follows that \eqref{Zak-limit} is same as the moment expression  in equation (3.2) of \citep{noiry2018spectral}.

\section{Details for the $S_{A}$ matrices}\label{SA-matrix}
This section deals with the details for $S_A$. We first describe some notions and definitions followed by a detailed proof of Theorem \ref{res:main}. Finally, we discuss how this theorem can deal with triangular (size dependent) entries, sparsity, i.i.d., variance profile, band and block matrices.
  
\subsection{Premliminaries}\label{preliminaries-SA}
 We recall link functions, circuits and words from Section 4 in \citep{bose2021some}, as applied to $S_{A}$. \vskip5pt

\noindent \textbf{Link function}: The link functions for the eight choices of $A$ are (here $1 \leq i \leq p, 1 \leq j \leq n$):
\begin{itemize}
\item[(i)] Symmetric reverse circulant, $R^{(s)}$: 
$L_{R^{(s)}}(i,j)= (i+j-2)(\text{mod }n), $.
\item[(ii)] Asymmetric reverse circulant, $R$: 
$L_R(i,j)= \begin{cases}
(i+j-2)(\text{mod }n) & i\leq j,\\
-[(i+j-2)(\text{mod }n) ] & i> j.
\end{cases}.$
\item[(iii)] Symmetric circulant, $C^{(s)}$: 
$L_{C^{(s)}}(i,j)= n/2-|n/2-|i-j||$, 
\item[(iv)] Circulant, $C$: 
$L_C(i,j)=(j-i) (\text{mod }n)$. 
\item[(v)] Symmetric Toeplitz, $T^{(s)}$: 
$L_{T^{(s)}}(i,j)= |i-j|$, 
\item[(vi)] Asymmetric Toeplitz, $T$: 
$L_T(i,j)= i-j$. 
\item[(vii)] Symmetric Hankel, $H^{(s)}$: 
$L_{H^{(s)}}(i,j)= i+j$. 
\item[(viii)] Asymmetric Hankel, $H$: 
$L_H(i,j)= \begin{cases}
(i+j) & i\geq j,\\
-(i+j) & i< j.
\end{cases}.$ 
\end{itemize}

\noindent \textbf{Circuits and Words}: Recall the definition of circuits and words from Section \ref{preliminaries-smatrix} for the $S$ matrix. In this case those definitions remain unaltered.
Suppose $A$ is one of the eight mentioned patterned matrices.
Then as before, notice that circuits $\pi$ with $\ell(\pi)=2k$ are required to deal with the $k$th moment of $S_A$.
For any choice of the link $L$, 
\begin{align*}
\xi_{\pi}(2i-1) &= L(\pi(2i-2),\pi(2i-1)), 1 \leq i \leq k\\
\xi_{\pi}(2i) &= L(\pi(2i), \pi(2i-1)),1 \leq i \leq k,
\end{align*}
\begin{align}\label{genmoment-AAt}
\mathbb{E}\big[\Tr(S_{A}^{k})\big] =\mathbb{E}\big[\Tr(AA^T)^{k}\big]
= \sum_{\pi:\ell(\pi)=2k}\mathbb{E}[Y_{\pi}],
\end{align}
where $Y_{\pi}= \displaystyle \prod_{i=1}^k x_{\xi_{\pi(2i-1)}}x_{\xi_{\pi(2i)}}$.\vskip5pt

\noindent \textbf{The class $\Pi(\boldsymbol {\omega})$}: 
For $\boldsymbol {\omega}$, 
\begin{equation}\label{SA-link}
\Pi_{S_{A}}(\boldsymbol {\omega})=\{\pi : \ \boldsymbol {\omega}[i]=\boldsymbol {\omega}[j] \ \Leftrightarrow \xi_{\pi}(i)= \xi_{\pi}(j) \text{ for all } i,j\}.
\end{equation}
Now, 
 \begin{align}\label{momentnoniid-AAt}
\lim_{p \rightarrow \infty}\frac{1}{n}\mathbb{E}[\Tr(S_{A})^{k}]& =\displaystyle\lim_{n \rightarrow \infty}\frac{1}{p} \sum_{\pi:\ell(\pi)=2k}\mathbb{E}[Y_{\pi}] \nonumber\\
& = \displaystyle\lim_{p \rightarrow \infty} \displaystyle \sum_{b=1}^k  \sum_{\underset{\text{with b distinct letters}}{\omega \ \text{matched of length }2k}}\frac{1}{p} \sum_{\pi \in \Pi_{S_A}(\boldsymbol{\omega})} \mathbb{E}(Y_{\pi}).
    \end{align}
Note that all words that appear above are of length $2k$. For every $k \geq 1$, the words of length $2k$ corresponding to the circuits of $A$ and $S_A$, are related. Here we make a key observation in that regard. 
\vskip3pt

\noindent \textbf{Observation (i)}: Let $A^{(s)}$ stand for any of the symmetric matrices  $R^{(s)},H^{(s)},C^{(s)}$ or $T^{(s)}$ and let
$\Pi_{A^{(s)}}(\boldsymbol {\omega})$ be the possibly larger class of  circuits for $A^{(s)}$
with range $1 \leq \pi(i) \leq \max(p,n), 0 \leq i \leq 2k$. Let $\Pi_{S_{A^{(s)}}}(\boldsymbol {\omega})$ and $\Pi_{A^{(s)}}(\boldsymbol {\omega})$ denote the set of all circuits corresponding to a word $\boldsymbol {\omega}$ arising from the circuits corresponding to $A^{(s)}$ and $S_{A^{(s)}}$, respectively. Then, for every $k \geq 1$ and any word $\boldsymbol{\omega}$ of length $2k$, 
\begin{equation}\label{SA-S}
\Pi_{S_A}(\boldsymbol {\omega})\subset \Pi_{S_{A^{(s)}}}(\boldsymbol {\omega})\subset \Pi_{A^{(s)}}(\boldsymbol {\omega}). 
\end{equation}

Now we recall the definition of  even and symmetric words from \citep{bose2021some}.\vskip3pt
 
\noindent\textbf{Even word:}
A word $\boldsymbol {\omega}$ is called \textit{even} if each distinct letter in $\boldsymbol {\omega}$ appears an even number of times. 
We shall denote the set of all even words of length $2k$ as $E(2k)$ and the set of all even words of length $2k$ with $b$ distinct letters as $E_b(2k)$. For example, $ababcc$ is an even word of length 6 with 3 letters. The corresponding partition of $[6]$ is $\{\{1,3\},\{2,4\},\{5,6\}\}$.\vskip3pt

\noindent\textbf{Symmetric word:}
A word $\boldsymbol {\omega}$ is  \textit{symmetric} if each distinct letter appears equal number of times in odd and even positions. 
 We shall denote the set of all symmetric  words of length $2k$ as $S(2k)$ and the set of all symmetric words of length $2k$ with $b$ distinct letters as $E_b(2k)$. 
\vskip3pt
  
\noindent \textbf{Even and odd generating vertices} are defined exactly as before. 
Observe that  
\begin{align}\label{SA-Pi(omega)}
\big|\Pi_{S_A}(\boldsymbol {\omega})\big| = \big|\big\{ &\big(\pi(0), \pi(1),\ldots,\pi(2k)\big): 1\leq \pi(2i)\leq p, 1 \leq \pi(2i-1)\leq n \text{ for } i=0,1,\ldots,k,  \nonumber\\
 &\pi(0)=\pi(2k), \quad \xi_{\pi}(i)=\xi_{\pi}(j) \text{ if and only if }\boldsymbol {\omega}[i]=\boldsymbol {\omega}[j] , 1\leq i,j \leq 2k  \big\}\big|.
\end{align}
 As $p/n\rightarrow y>0$,
\begin{equation}\label{cardiality of word}
| \Pi_{S_A}(\boldsymbol {\omega})| = \mathcal{O}(p^{r+1}n^{b-r})\ \ \text{whenever} \ \omega \text{ has } \\   b \ \ \text{distinct letters and }(r+1) \ \ \text{ even generating vertices.}
\end{equation} 
As before, the existence of the following limit is tied very intimately to the LSD of $S_{A}$.  
\begin{align}\label{word limit}
\lim_{p,n \rightarrow \infty} \frac{| \Pi_{S_A}(\boldsymbol {\omega})|}{p^{r+1}n^{b-r}} \ \ \text{ as } p/n \rightarrow y. 
\end{align}
  In Section \ref{proofs-SA}, we identify the words for which the above limit is positive.
  \vskip3pt
  
As before we will use the moment method. The $L\acute{e}vy$ metric defined in Section \ref{preliminaries-smatrix} will be helpful in dealing with non-centered variables and truncation inequalities.

\begin{lemma}[Theorem A.38 in \citep{bai2010spectral}]\label{L-inequality-set}
Let $\lambda_k$ and $\delta_k$, $1 \leq k \leq n$ be two sets of complex numbers and $F$ and $\bar{F}$ denote their empirical distributions. Then for any $\alpha>0$,
\begin{equation}
L^{\alpha+1}(F,\bar{F})\leq \min_\pi \frac{1}{n} \displaystyle \sum_{k=1}^n |\lambda_k- \delta_{\pi(k)}|^{\alpha}
\end{equation}
where $L$ is the L\'{e}vy distance and $\pi=(\pi(1), \ldots, \pi(n))$ is any permutation of $1,2,\ldots,n$.
\end{lemma}
\begin{lemma}[Theorem A.37 in \citep{bai2010spectral}]\label{trace inequality}
Suppose $A$ and $B$ are real $p \times n$ matrices and $\lambda_k$ and $\delta_k$, $1 \leq k \leq p$ are the singular values of $A$ and $B$ arranged in descending order. Then,
\begin{equation}
\displaystyle \sum_{k=1}^{\min(p,n)} |\lambda_k- \delta_k|^2 \leq \Tr[(A-B)(A-B)^T].
\end{equation} 
\end{lemma}

The following inequality on the L\'{e}vy distance between the EESDs of two matrices can be easily deduced from the above lemmas.
\begin{lemma}\label{lem:Emetric}
Suppose $A$ and $B$ are real $p \times n$ matrices and $\mathbb{E}F^{S_A}$ and $\mathbb{E}F^{S_B}$ denote the EESDs of $AA^T$ and $BB^T$ respectively. Then the \textit{$L\acute{e}vy$ distance}, $L$ between these distributions satisfy the following inequality: 
\begin{equation}\label{levy inequality}
L^4(\mathbb{E}F^{S_A},\mathbb{E}F^{S_B})\leq \frac{2}{p^2}(\mathbb{E}\Tr(AA^T+BB^T))(\mathbb{E}\Tr[(A-B)(A-B)^T]).
\end{equation}
\end{lemma}	
\begin{proof}
Let $\sigma_i(A), \sigma_i(B), i= 1,2 \ldots,p$ denote respectively the singular values of $A$ and $B$, each set arranged in descending order. Then with $\alpha=1$, from Lemma \ref{L-inequality-set}, we have,
\begin{align*}
L^2(\mathbb{E}F^{S_A},\mathbb{E}F^{S_B})  & \leq \frac{1}{p} \displaystyle \sum_{i=1}^p |\mathbb{E}\lambda_i^2- \mathbb{E}\delta_i^2| \nonumber\\ 
& \leq \frac{1}{p} \mathbb{E}\bigg[\displaystyle \sum_{i=1}^p\big|\lambda_i^2- \delta_i^2\big|\bigg]\\
& \leq \frac{1}{p}\bigg(\mathbb{E}\big[\displaystyle \sum_{i=1}^p (\lambda_i + \delta_i)^2 \big]\bigg)^{1/2} \bigg(\mathbb{E}\big[\displaystyle \sum_{i=1}^p |\lambda_i - \delta_i|^2 \big]\bigg)^{1/2}\\
&\leq \frac{1}{p}\bigg(2\mathbb{E}\big[\displaystyle \sum_{i=1}^p (\lambda_i^2 + \delta_i^2) \big]\bigg)^{1/2} \bigg(\mathbb{E}\big[\displaystyle \sum_{i=1}^p |\lambda_i - \delta_i|^2 \big]\bigg)^{1/2}.
\end{align*}
Now using Lemma \ref{trace inequality} on the second factor of the above inequality, we get \eqref{levy inequality}.
\end{proof}

\subsection{Proofs for $S_A$ matrices}\label{proofs-SA}
We look at the words that contribute to the limiting moments and determine their contribution for each of the matrices.

\subsubsection{$\lim_{n\to\infty}\frac{| \Pi(\boldsymbol {\omega})|}{p^{r+1}n^{b-r}}$ for $S_A$ matrices}

\begin{lemma}\label{lem:toe-S}
 Suppose $\boldsymbol {\omega}$ is a word with $b$ distinct letters and $(r+1)$ even generating vertices. Then $\displaystyle \lim_{n \rightarrow \infty}\frac{1}{p^{r+1}n^{b-r}} | \Pi_{S_{T^{(s)}}}(\boldsymbol {\omega})|= \alpha_{T^{(s)}}(\boldsymbol {\omega})>0$ if and only if $\boldsymbol {\omega}$ is an even word.
 \end{lemma}
 \begin{proof}
 First suppose $\boldsymbol {\omega} \in \mathcal{P}(2k) \setminus E_b(2k)$. Then from \eqref{SA-S} and Lemma 5.3 in \citep{bose2021some}, and using the fact that $p/n \rightarrow y>0$ as $n\rightarrow \infty$, it is easy to see that $$\displaystyle \lim_{n \rightarrow \infty}\frac{1}{p^{r+1}n^{b-r}} | \Pi_{S_{T^{(s)}}}(\boldsymbol {\omega})|=0.$$  
Now suppose $\boldsymbol{\omega}$ is an even word with $b$ distinct letter and $(r+1)$ even generating vertices. Let $i_1,i_2,\ldots,i_b$ be the positions where new letters made their first appearances. First we fix the generating vertices $\pi(i_j), 0 \leq j\leq b$ where $\pi(i_0)=\pi(0)$. Let
\begin{align*} 
s_i= \pi(i)- \pi(i-1)  \  \  \text{for }1 \leq i \leq 2k.
\end{align*}
Clearly, from \eqref{SA-link}, $\boldsymbol {\omega}[i]=\boldsymbol {\omega}[j]$ 
if and only if $\xi_{\pi}(i)=\xi_{\pi}(j)$. That is, 
$|s_i| =  |s_j|$, that is, $s_i - s_j  =  0$ or   $s_i  +  s_j  =  0.$
  Clearly $s_{i_1}=s_1$. If the first letter appears in the $j$-th position, then $s_j=s_1$ or $s_j=-s_1$. Similarly,  for every $i,\ 1 \leq i \leq 2k$, for some $j \in \{1,2,\ldots,b\}$,
\begin{equation}\label{Stplus_minus}
s_i = s_{i_j} \mbox{ or }	s_i =  - s_{i_j}.
\end{equation}

Thus, we have 
 \begin{align}\label{toe-relation}
 \pi(i)  &=  \pm(\pi(i_j)-\pi(i_j-1))+\pi(i-1)=\pm s_{i_j}+ \pi(i-1) \ \ \text{for some }j.
  \end{align}
  Let 
  $$
  v_{2i}=  \frac{\pi(2i)}{p}  \ \text{ for }0 \leq i \leq k \ \mbox{ and }\  
   v_{2i-1}= \frac{\pi(2i-1)}{n} \ \ \text{ for } \ 1 \leq i \leq k,
  $$
  $$
  u_{i}=  \frac{s_i}{n}   \text{ for }1 \leq i \leq 2k.
  $$
Let $y_n=p/n$. Now, $\pi(i)= \pi(i-1) \pm s_{i_j}$ whenever the $i$-th letter in $\boldsymbol{\omega}$ is same as the $j$-th distinct letter that appeared first at the $i_j$-th position. Therefore $v_i= \frac{1}{y_n}(v_{i-1} \pm u_{i_j})$ when $i$ is even and $v_i= {y_n}v_{i-1} \pm u_{i_j}$ when $i$ is odd. 
  
Let $$S=\{\pi(i_j):0 \leq  j \leq b\}\ \text{ and }\ S^{\prime}=\{i:\ \pi(i) \notin S\}.$$ 
That is, $S$ is the set of all distinct generating vertices and $S^{\prime}$ is the set of all indices of the non-generating vertices. We have the following claim.
\vskip3pt  
\noindent \textbf{Claim}: For any $1 \leq i \leq 2k$,
\begin{align}\label{toe-vrelation}
 v_i= \begin{cases} 
v_0+ \frac{1}{y_n} \sum_{j=1}^i \alpha_{ij}u_{i_j} & \ \text{if } i \ \ \text{ is even},\\
y_n v_0+  \sum_{j=1}^i \alpha_{ij}u_{i_j} & \ \text{if } i \ \ \text{ is odd}\end{cases}
\end{align}
where  $\alpha_{ij}$ depends on the choice of sign in  \eqref{toe-relation}.
\vskip3pt
\noindent We prove this by induction on $i$. We know that $\pi(1) \in S$. Clearly, $v_1=u_1 +y_n v_0$. Now either $\pi(2) \in S$ or $2 \in S^{\prime}$. If $\pi(2)\in S$, then $v_2=\frac{1}{y_n}(u_2-v_1)$ and $v_1=u_1-y_nv_0$. Therefore $v_2=\frac{1}{y_n}( u_2-u_1)+v_0$. If $\ 2\in S^{\prime}$, then
  $u_2= \pm u_1$  and $v_2= \frac{1}{y_n}(v_1 \pm u_1)$.
  So either $v_2=\frac{1}{y_n} (u_1 + y_nv_0 + u_1)$ or $ v_2=\frac{1}{y_n} (u_1 + y_nv_0 - u_1)=v_0$.
Hence the claim is true for $i=2$.

Now we assume that the claim is true for all $j<i$ and prove it for $i$. There are two cases:\\

\noindent Case 1: $i$ is even. Now either $\pi(i) \in S$ or $i \in S^{\prime}$. If $\pi(i) \in S$, then 
  \begin{align*}
  v_i&= \frac{1}{y_n} (u_i+v_{i-1})\\
 &= \frac{1}{y_n}(u_i+ y_nv_0+  \sum_{j=1}^{i-1} \alpha_{(i-1)j}u_{i_j})  \   \   \text{(by induction hypothesis as } i-1 \ \ \text{ is odd)}\\
 &=  v_0 + \frac{1}{y_n} \sum_{j=1}^{i} \alpha_{ij}u_{i_j}, 
  \end{align*}
where $\alpha_{ii}=1$.  If $i \in S^{\prime}$, then there exists $j$ such that $i_j<i$ and $u_i=\pm u_{i_j}$. Then either  $ v_i= \frac{1}{y_n}(v_{i-1}+ u_{i_j})$, or     $v_i= \frac{1}{y_n}(v_{i-1}- u_{i_j}).$
Hence either $$ v_i=v_0+ \frac{1}{y_n}(\sum_{j=1}^{i-1} \alpha_{(i-1)j}u_{i_j}  + u_{i_j}), \ \mbox{ or }\  v_i=v_0+  \frac{1}{y_n}(\sum_{j=1}^{i-1} \alpha_{(i-1)j}u_{i_j} - u_{i_j}).$$
  Therefore $v_i=v_0 + \frac{1}{y_n} \sum_{j=1}^{i} \alpha_{ij}u_{i_j}$  where  $\alpha_{ij}=\alpha_{(i-1)j}+1$ or $\alpha_{(i-1)j}-1$ (depending on the sign of the equation).\\ 
  
\noindent Case 2: $i$ is odd. Then using similar argument as above we can show that, $v_i= y_n v_0+  \sum_{j=1}^i \alpha_{ij}u_{i_j}$.

Thus the claim is proved.  
  
Let $u_S=\{u_i:\pi(i)\in S\}$ and $v_S=\{v_i:\pi(i)\in S\}$. From the previous claim, we have, for $1\leq i \leq 2k$,  
  \begin{equation*}
  v_i= \begin{cases}
  v_0+ \frac{1}{y_n}L_{i,u,n}^T(u_{S}),& \ \text{if } i \ \ \text{ is even},\\
y_n v_0+ L_{i,u}^T(u_{S})  & \ \text{if } i \ \ \text{ is odd},
  \end{cases}
  \end{equation*}
  where $L_{i,u,n}^T(u_{S})$ denotes a linear combination of $\{u_i:\pi(i)\in S\}$.

Also, for $1\leq i \leq 2k$, we have 
  \begin{equation}\label{linear-toe}
  v_i=  L_{i,n}^T(v_{S}),
  \end{equation}
  where $L_i^T(v_{S})$ denotes a linear combination of $\{v_i:\pi(i)\in S\}$ arising from \eqref{Stplus_minus}. 
 Now, the linear combinations vary depending on the sign chosen for each $s_i$. As we know, for each block of an even word, the number of positive and negative signs in the relations among the $s_i$'s (i.e., the equations like \eqref{Stplus_minus}) are equal. Therefore there are $\displaystyle \prod_{i=1}^b{ {k_i-1} \choose {\frac{k_i}{2}}}$ different sets of linear combinations corresponding to each word $\boldsymbol{\omega}$, where $k_1,\ldots,k_b$ are the block sizes of $\boldsymbol {\omega}$. 
  
Let $U_p=\{0,1/p,\ldots,(p-1)/p\},U_n=\{0,1/n,\ldots,(n-1)/n\}$. Then it is easy to see that for a word $\boldsymbol{\omega}$ of length $2k$, 
\begin{align*}
\big|\Pi_{S^{T^{(s)}}}(\boldsymbol{\omega})\big|= \big|\big\{& (v_0,v_1,\ldots,v_{2k}): v_{2i} \in U_p, v_{2i-1}\in U_n \text{ for } 0\leq i\leq k, v_0=v_{2k},\\
& L_{T^{(s)}}(v_{i-1},v_i)=L(v_{j-1},v_j) \text{ whenever } \boldsymbol{\omega}[i]=\boldsymbol{\omega}[j]\big \}\big|.
 \end{align*}
Hence
\begin{equation}\label{integral-gen}
  \displaystyle \lim_{n \rightarrow \infty}\frac{1}{p^{r+1}n^{b-r}} |\Pi_{S_{T^{(s)}}}(\boldsymbol{\omega})|= \sum_{L_{\boldsymbol{\omega}}^T}\int_{0}^{1} \int_{0}^{1}\int_{0}^{1} \cdots  \int_{0}^{1} \textbf{1}(0 \leq L_i^T(x_S) \leq 1,\ \forall \  i \in S^{\prime})\ dx_S,
\end{equation}
where $dx_S= \displaystyle \prod_{j=0}^{b}dx_{i_j}$  denotes the $(b+1)$-dimensional Lebesgue measure, $x_{i_0}=x_0$, and $L_{i}^T$ is the limit of $L_{i,n}^T$ as $p/n\rightarrow y$ and $\sum_{L_{\boldsymbol{\omega}}^T}$ is the sum over all the $\displaystyle \prod_{i=1}^b{ {k_i-1} \choose {\frac{k_i}{2}}}$ such different sets of linear combinations corresponding to $\boldsymbol{\omega}$.

As observed in Step 1 in Lemma 5.3 of \citep{bose2021some}, choosing $v_{i_j}, 0 \leq j \leq b$ freely is equivalent to choosing  $v_0$ and $u_{i_j},1 \leq j \leq b$ freely. Note that $-1 \leq u_{2i}\leq y_n$ and $-y_n \leq u_{2i-1}\leq 1$ for $0 \leq i \leq k$. Also by abuse of notation we denote the variables in the limit as $u_{i_j}$. 
 So,
  \begin{align}\label{toe-integral-exp}
  \displaystyle \lim_{n \rightarrow \infty}\frac{1}{p^{r+1}n^{b-r}} |\Pi_{S_{T^{(s)}}}(\boldsymbol{\omega})|=  \sum_{L_{\boldsymbol{\omega}}^T} &\int_{0}^{1} \int_{-1}^{y} \cdots  \int_{-y}^{1} 
   \textbf{1}(0 \leq x_0+\frac{1}{y}L_{i,u}^T(u_S) \leq 1,\ \forall \  2i \in S^{\prime})\nonumber \\
  &\textbf{1}(0 \leq y x_0+ L_{i,u}^T(u_S) \leq 1,\ \forall \  (2i-1) \in S^{\prime}) \ dx_0du_S,
\end{align}    
where $du_S= \displaystyle \prod_{j=1}^{b}du_{i_j}$  denotes the $(b+1)$-dimensional Lebesgue measure on $[-1,y]^{r}\times [-y,1]^{b-r}$.

  Suppose, a particular set of linear combinations $L_{i,u}^T$ is given,
  i.e., for $i \in S^{\prime}$, \eqref{toe-vrelation} holds and the values of $\alpha_{im},1 \leq i \leq 2k, 1\leq j \leq b$ are known. Here we show that the integral in \eqref{toe-integral-exp} is positive on a certain region in $[0,1] \times [-1,y]^{r}\times [-y,1]^{b-r}$. We divide this proof into two cases.\\
  
\noindent Case 1: $y>1$.  First let 
\begin{equation}\label{def-C}
C = \max \{|\alpha_{ij}|: 1 \leq j \leq b \text{ and } i \in S^{\prime} \}.
\end{equation}
Next we choose $\epsilon$ such that  $Cb\epsilon  <  1/2$. 
 Now, let $|u_{i_j}|< \epsilon$ for $1 \leq j \leq b$ and $\frac{Cb\epsilon}{y} \leq v_0 \leq \frac{1-Cb\epsilon}{y}$. 
 Then, for all $2i \in S^{\prime}$, $0 \leq v_0+ \frac{1}{y} L_{2i,u}^T (u_S) \leq 1$ and for all $2i-1 \in S^{\prime}$, $0 \leq y v_0+  L_{2i-1,u}^T (u_S) \leq 1$. Also the circuit condition is automatically satisfied. Note that as observed before we cannot choose the $s_{i_j}$'s freely in case of words that are not even.\\
 
 \noindent Case 2: $y\leq 1$.  First let $C$ be as in \eqref{def-C}. Next we choose $\epsilon$ such that  $\frac{Cb\epsilon}{y}  <  1/2$. 
 Now, let $|u_{i_j}|< \epsilon$ for $1 \leq j \leq b$ and $\frac{Cb\epsilon}{y} \leq v_0 \leq 1- \frac{Cb\epsilon}{y}$. 
 Then, for all $2i \in S^{\prime}$, $0 \leq v_0+ \frac{1}{y} L_{2i,u}^T (u_S) \leq 1$ and for all $2i-1 \in S^{\prime}$, $0 \leq y v_0+  L_{2i-1,u}^T (u_S) \leq 1$.

 Thus
$$\displaystyle \lim_{n \rightarrow \infty}\frac{1}{n^{b+1}} |\Pi_{S_{T^{(s)}}}(\boldsymbol {\omega})|=\alpha(\boldsymbol{\omega})> 0 \ \ \text{for any even word}\ \boldsymbol{\omega},$$ 
where $\alpha_{T^{(s)}}(\boldsymbol{\omega})$ is the sum of the integrals defined in \eqref{toe-integral-exp}.

 This completes the proof of the lemma. 
 \end{proof}
\begin{lemma}\label{lem:toe}
Suppose $\boldsymbol {\omega}$ is a word with $b$ distinct letters and $(r+1)$ even generating vertices. Then $\displaystyle \lim_{n \rightarrow \infty}\frac{1}{p^{r+1}n^{b-r}} | \Pi_{S_T}(\boldsymbol {\omega})|= \alpha_{T}(\boldsymbol {\omega})>0$ if and only if $\boldsymbol {\omega}$ is symmetric.
\end{lemma} 
 
 \begin{proof}
 
 Let $$s_i=\pi(i)-\pi(i-1), \ \ \text{ for } 1 \leq i \leq 2k.$$
 From \eqref{SA-link}, we know that $\boldsymbol {\omega}[i]= \boldsymbol {\omega}[j]$ if and only if $\xi_{\pi}(i)= \xi_{\pi}(j)$. This implies
 \begin{align}\label{T-relation}
 s_i= & s_j \ \ \text{ when } i \ \text{ and } j \ \text{ are of same parity}, \nonumber \\
  s_i= & -s_j \ \ \text{ when } i \ \text{ and } j \ \text{ are of opposite parity}.
 \end{align}
Now we fix an $\boldsymbol {\omega}$ with $b$ distinct letters which appear at $i_1,i_2,\ldots,i_b$ positions for the first time. Also let  $\boldsymbol {\omega}$ have $(r+1)$ even generating vertices. Using the same arguments as in Step 1 of Lemma 5.3 of \citep{bose2021some}, being able to choose $\pi(i_j), 0 \leq j \leq b$ is equivalent to choosing $\pi(0), s_{i_j}, 1\leq j \leq b$ freely. Next we show that if $\pi(0)$ and $s_{i_j}, 1 \leq j \leq b$ can be chosen freely, then the word is symmetric.

To see this, observe that the circuit condition gives 
\begin{equation}\label{T-s-relation}
\displaystyle \sum_{i=1}^{2k}s_i=\pi(0)-\pi(2k)= 0.
\end{equation}
Using \eqref{T-relation}, we see that there exists $\alpha_j, 1\leq j \leq b$ such that $$\displaystyle \sum_{i=1}^b\alpha_js_{i_j}=0.$$
Since we need the $s_{i_j}$'s to have free choice, we must have $\alpha_j=0$ for all $j \in \{1,2,\ldots,b\}$.	 Therefore for each $j$, 
 \begin{align}\label{T-s-relation1}
\big|\{l:s_l=s_{i_j}\}\big|=\big|\{l:s_l=-s_{i_j}\}\big|.
 \end{align} 
 Now from the definition of $\xi_{\pi}$, $\xi_{\pi}(2i)= s_{2i}, 0 \leq i \leq k$ and  $\xi_{\pi}(2i-1)= -s_{2i-1}, 0 \leq i \leq k$. Therefore for each $j \in \{1,2 , \ldots, b\}$, to satisfy \eqref{T-s-relation1}, we must have  
 \begin{align*}
\big|\big\{l:l \text{ even and }\xi_{\pi}(l)=\xi_{\pi}(i_j) \big\}\big|=\big|\big\{l:l \text{ odd and }\xi_{\pi}(l)=\xi_{\pi}(i_j)\big\}\big|.
\end{align*} 
That is, each letter appears equal number of times at odd and even places. Hence the word is symmetric.

Now if $\boldsymbol {\omega}$ is not symmetric, at least one of the generating vertices is a linear combination of the others, and hence cannot be chosen freely. So, 
$$\displaystyle \lim_{n \rightarrow \infty}\frac{1}{p^{r+1}n^{b-r}} | \Pi_{S_T}(\boldsymbol {\omega})|=0\ \ \text{if}\ \boldsymbol {\omega}\  \text{is not symmetric}.$$

Next, suppose $\boldsymbol{\omega}$ is a symmetric word with $b$ distinct letters and $(r+1)$ even generating vertices. We shall show that $\displaystyle \lim_{n \rightarrow \infty}\frac{1}{p^{r+1}n^{b-r}} | \Pi_{S_T}(\boldsymbol {\omega})|=\alpha_T(\boldsymbol {\omega})>0$.

Suppose the letters make their first appearances at $i_1,i_2, \ldots,i_b$ positions in $\boldsymbol{\omega}$. First we fix the generating vertices $\pi(i_j), 0 \leq j \leq b$. Suppose $S=\{\pi(i_j):0 \leq  j \leq b\}  \ \mbox{ and }\ S^{\prime}=\{i : \pi(i) \notin S\}.$ For $i\in S^{\prime}$, $\xi_{\pi}(i)= \xi_{\pi}(i_j)$ for some $j \in \{1,2, \ldots b\}$. Then  
\begin{align}\label{T-pi-relation}
\pi(i)=&  s_{i_j} + \pi(i-1) \ \ \text{ if } i \ \text{ and } i_j \  \text{ are of same parity}\nonumber \\
\pi(i)=&  -s_{i_j} + \pi(i-1) \ \ \text{ if } i \ \text{ and } i_j \  \text{ are of opposite parity}.
\end{align}
Thus, \eqref{T-pi-relation} is nothing but \eqref{toe-relation} where the sign has been determined depending on the parity of $i$ and $i_j$. Hence, for $1 \leq i \leq 2k$, $v_{i}=L_{i,n}^{T}(v_S)$, (the notations are the same as in the proof of Lemma \ref{lem:toe-S}) where $L_{i,n}^{T}(\cdot)$ is a particular set of linear combinations that has been determined by \eqref{T-pi-relation}. As a result, the rest of the proof is same as that of Lemma \ref{lem:toe-S}. Therefore, 
\begin{equation}\label{T-integral-gen}
  \displaystyle \lim_{n \rightarrow \infty}\frac{1}{p^{r+1}n^{b-r}} |\Pi_{S_{T}}(\boldsymbol{\omega})|= \int_{0}^{1} \int_{0}^{1}\int_{0}^{1} \cdots  \int_{0}^{1} \textbf{1}(0 \leq L_i^T(x_S) \leq 1,\ \forall \  i \in S^{\prime})\ dx_S,
\end{equation}
where $dx_S= \displaystyle \prod_{j=0}^{b}dx_{i_j}$  denotes the $(b+1)$-dimensional Lebesgue measure, $x_{i_0}=x_0$.

That the integral is positive now follows from the proof of the same fact in Lemma \ref{lem:toe-S}. Thus, 
$$\displaystyle \lim_{n \rightarrow \infty}\frac{1}{p^{r+1}n^{b-r}} |\Pi_{S_{T}}(\boldsymbol{\omega})|=\alpha_T(\boldsymbol{\omega})>0 \ \ \text{ for every symmetric word } \boldsymbol{\omega},$$
where $\alpha_T$ is the value of an individual integral in the rhs of \eqref{toe-integral-exp}.
\end{proof}

 \begin{lemma}\label{lem:han}
 Suppose $\boldsymbol {\omega}$ is a word with $b$ distinct letters and $(r+1)$ even generating vertices. Then 
 \begin{itemize}
 \item[(i)]$\displaystyle \lim_{n \rightarrow \infty}\frac{1}{p^{r+1}n^{b-r}} | \Pi_{S_{H^{(s)}}}(\boldsymbol {\omega})|= \alpha_{H^{(s)}}(\boldsymbol {\omega})>0$ if and only if $\boldsymbol {\omega}$ is a symmetric word.
 \item[(ii)] $\displaystyle \lim_{n \rightarrow \infty}\frac{1}{p^{r+1}n^{b-r}} | \Pi_{S_H}(\boldsymbol {\omega})|= \alpha_{H}(\boldsymbol {\omega})$ can only be positive if $\boldsymbol {\omega}$ is a symmetric word. In case of a symmetric word the value of $\alpha_{H}(\boldsymbol {\omega})$ is determined by an integral given in \eqref{H-integral}.
 \end{itemize}
 \end{lemma}
 \begin{proof}
  First suppose $\boldsymbol {\omega} \in \mathcal{P}(2k) \setminus S_b(2k)$. Then from \eqref{SA-S} and Lemma 5.4 in \citep{bose2021some}, it is easy to see that $$\displaystyle \lim_{n \rightarrow \infty}\frac{1}{p^{r+1}n^{b-r}} | \Pi_{S_{H^{(s)}}}(\boldsymbol {\omega})|=\lim_{n \rightarrow \infty}\frac{1}{p^{r+1}n^{b-r}} | \Pi_{S_{H}}(\boldsymbol {\omega})|= 0.$$  
Now suppose $\boldsymbol{\omega}$ is a symmetric word with $b$ distinct letters and $(r+1)$ even generating vertices. Suppose $i_1,i_2,\ldots,i_b$ are the positions where new letters made their first appearances. First we fix the generating vertices $\pi(i_j), 0 \leq j\leq b$ where $\pi(i_0)=\pi(0)$. Let 
\begin{align*}
t_i= \pi(i)+ \pi(i-1)  \text{ for } 1 \leq i \leq 2k.
\end{align*}
Now let us first consider the symmetric Hankel link function. Clearly from \eqref{SA-link}, $\boldsymbol {\omega}[i]=\boldsymbol {\omega}[j]$ 
 if and only if  $t_i -  t_j = 0$. Further $t_{i_1}=t_1$. If the first letter again appears at the $j$-th position, then $t_j = t_1$. Similarly,  for every $i,\ 1 \leq i \leq 2k$, 
\begin{equation}\label{S-han1}
t_i = t_{i_j}\ \ \text{ for some }j \in \{1,2,\ldots,b\}.
\end{equation}

First we fix the generating vertices $\pi(i_j),\ j=0,1,2,\ldots,b$.
 Let 
  $$
  v_{2i}=  \frac{\pi(i)}{p},\ \ v_{2i-1}=  \frac{\pi(i)}{n}  \text{ for }0 \le i \leq k,\ \\
   S=\{\pi(i_j):0 \leq  j \leq b\}  \ \mbox{ and }\ S^{\prime}=\{i : \pi(i) \notin S\}.$$
   For $1\leq i \leq 2k$, from the link function and the formula for $t_i$ we have 
  \begin{equation}\label{linear-han}
  v_i= L_{i,n}^H(v_{S}),
  \end{equation}
  where $L_{i,n}^H(v_{S})$ denotes a linear combination of $\{v_i:\pi(i)\in S\}$.  
  
 Let $U_p=\{0,1/,\ldots,(p-1)/p\}$ and $U_n=\{0,1/n,\ldots,(n-1)/n\}$. From \eqref{SA-Pi(omega)}, it is easy to see that for a word $\boldsymbol{\omega}$ of length $2k$, 
\begin{align*}
\big|\Pi_{S_{H^{(s)}}}(\boldsymbol{\omega})\big|= \big|\big\{ (v_0,v_1,\ldots,v_{2k}):v_{2i}\in U_p v_{2i-1}\in U_n \text{ for } 0\leq i\leq k, v_0=v_{2k},
v_i= L_{i,n}^H(v_S)\big \}\big|.
 \end{align*}
Transforming $v_i \mapsto x_i= v_i -\frac{1}{2}$, we get that 
\begin{align*}
\big|\Pi_{S_{H^{(s)}}}(\boldsymbol{\omega})\big|= \big|\big\{ &(x_0,x_1,\ldots,x_{2k}):  \ x_{2i}\in \{-1/2,-1/2+1/p, \ldots,-1/2+(p-1)/p\},\\
&x_{2i-1}\in \{-1/2,-1/2+1/n, \ldots,-1/2+(n-1)/n\} \text{ for } 0\leq i\leq k,\\
&\quad x_0=x_{2k} \mbox{ and }
x_i= L_{i,n}^H(y_S)\big \}\big|.\end{align*}  
Hence
 \begin{equation}\label{S-han2}
  \displaystyle \lim_{n \rightarrow \infty}\frac{1}{p^{r+1}n^{b-r}}\Pi_{S_{H^{(s)}}}(\boldsymbol{\omega})= \int_{-1/2}^{1/2} \int_{-1/2}^{1/2}\int_{-1/2}^{1/2} \cdots  \int_{-1/2}^{1/2} \textbf{1}(-1/2 \leq L_i^H(x_S) \leq 1/2,\ \forall\ i \in S^{\prime})\ dx_S,
\end{equation}
where $dx_S= \displaystyle \prod_{j=0}^{b}dx_{i_j}$  denotes the $(b+1)$-dimensional Lebesgue measure on $[-\frac{1}{2},\frac{1}{2}]^{b+1}$.

Let $y_n=p/n$ and for $1 \leq i \leq k$,
\begin{align}\label{p-relation}
p_{2i}=& x_{2i-1}+ y_n x_{2i}, \nonumber\\
 p_{2i-1}=& y_n x_{2i-2}+  x_{2i-1},
\end{align}
 and 
\begin{align}\label{q-relation}
q_{2i}=& x_{2i-1}- y_n x_{2i},\nonumber \\
 q_{2i-1}=& y_n x_{2i-2}-  x_{2i-1}.
\end{align} 
 Now we have the following claim.

\vskip3pt  
\noindent \textbf{Claim}: For any $1\leq i \leq 2k$,
$$x_i= \begin{cases}
x_0+\frac{1}{y_n}  \sum_{j=1}^{i} \alpha_{ij}p_{i_j} & \text{ if } \ i \ \mbox{ is  even},\\
 - y_n x_0+  \sum_{j=1}^{i} \alpha_{ij}p_{i_j} & \text{ if }\  i \  \text{ is odd}.
\end{cases} $$
\vskip3pt
\noindent We prove this by induction in $i$. 
We know that $\pi(1) \in S$. Clearly, $x_1=p_1 -y_n x_0$. Now either $\pi(2) \in S$ or $2 \in S^{\prime}$. If $\pi(2)\in S$, then $x_2=\frac{1}{y_n}(p_2-x_1)$. Therefore $x_2= \frac{1}{y_n}(p_2-p_1) + x_0$. If  $2\in S^{\prime}$, then
  $p_2=  p_1$  and $x_2= \frac{1}{y_n}(p_1 -y_1)=y_0$.
So the claim is true for $i=2$.
  
  Now we assume that the claim is true for all $j<i$ and try to prove it for $i$.
Then either $\pi(i) \in S$ or $i \in S^{\prime}$. 

\noindent (a) If $\pi(i) \in S$ and $i$ is even, then $y_i= \frac{1}{y_n}(p_i-y_{i-1})$. Now, $i-1$ is odd and hence $y_{i-1}= -y_n y_0+ \displaystyle \sum_{j=1}^{i-1}\alpha_{(i-1)j}p_{i_j}$ by induction hypothesis. Therefore $y_i= y_0+\frac{1}{y_n} \displaystyle \sum_{j=1}^{i}\alpha_{ij}p_{i_j}$ where $\alpha_{ii}=1$. The case where $i$ is odd can be tackled similarly.

\noindent (b)  If $i \in S^{\prime}$, then there exists $m$ such that $i_m<i$ and $p_i=p_{i_m}$. Now if $i$ is even, $y_i=\frac{1}{y_n}(p_{i_m}-y_{i-1})$. As $i-1$ is odd, $y_{i-1}= -y_n y_0+ \displaystyle \sum_{j=1}^{i-1}\alpha_{(i-1)j}p_{i_j}$ and therefore $y_i= y_0+ \frac{1}{y_n}\displaystyle \sum_{j=1}^{i-1}\alpha_{ij}p_{i_j}$ where $\alpha_{im}=\alpha_{(i-1)m}+1$. The case where $i$ is odd can be tackled similarly.
   
Thus the claim is proved.
\vskip3pt

Now we perform the following change of variables in \eqref{S-han2}:
\begin{align*}
(x_0,x_1,x_2,x_3,\ldots,x_{2k}) \longrightarrow (x_0,-x_1,x_2,-x_3,\ldots,x_{2k})=(z_0,z_1,z_2,z_3,\ldots,z_{2k}) \ (\text{ say }) .
\end{align*}
Observe that this transformation does not alter the value of the integral in \eqref{S-han2}. Also observe that using \eqref{p-relation} and \eqref{q-relation}, under this transformation, 
\begin{align*}
(p_1,p_2,p_3,\ldots,p_{2k}) \longrightarrow (q_1,-q_2,q_3,\ldots,-q_{2k}).
\end{align*}
Then from the claim it follows that $$z_i= \begin{cases}
z_0+ \frac{1}{y_n}\displaystyle \sum_{j=1}^{i}\beta_{ij}q_{i_j} &  \ \text{ if } i \ \text{ is even},\\
y_n z_0+ \displaystyle \sum_{j=1}^{i}\beta_{ij}q_{i_j} &  \ \text{ if } i \ \text{ is odd}
\end{cases}$$ 
where $\beta_{ij}=\pm \alpha_{ij}$ according as $i_j$ is odd or even.  We shall use the notation $z_i=L_{i,q,n}^H(z_S)$ to denote this linear relation.

Also note that choosing $x_{i_j}, 0 \leq j \leq b$ freely is equivalent to choosing $p_{i_j}, 0 \leq j \leq b $ (where $p_{i_0}=x_0$). Further, $-\frac{y_n+1}{2}\leq q_{i_j} \leq  \frac{y_n+1}{2}$. Therefore we can write \eqref{S-han2} as 
\begin{align*}
 \displaystyle \lim_{n \rightarrow \infty}\frac{1}{p^{r+1}n^{b-r}}\Pi_{S_{H^{(s)}}}(\boldsymbol{\omega})= \int_{-1/2}^{1/2} \int_{-\frac{y+1}{2}}^{\frac{y+1}{2}}\int_{-\frac{y+1}{2}}^{\frac{y+1}{2}} \cdots  \int_{-\frac{y+1}{2}}^{\frac{y+1}{2}} \textbf{1}(-1/2 \leq L_{i,q}^H(z_S) \leq 1/2,\ \forall\ i \in S^{\prime})\ dq_S,
\end{align*}
where $dq_S= \displaystyle \prod_{j=0}^{b}dq_{i_j}$  denotes the $(b+1)$-dimensional Lebesgue measure on $[-\frac{1}{2},\frac{1}{2}]\times [-\frac{y+1}{2},\frac{y+1}{2}]^{b}$ and $L_{i,q}^H$ denotes the limit of the linear combination $L_{i,q,n}^H$ as $y_n \rightarrow y>0$.

Now it can be proved that the above integrand is positive on a region of positive measure on $[-\frac{1}{2},\frac{1}{2}]\times [-\frac{y+1}{2},\frac{y+1}{2}]^{b}$- the proof is similar to the proof that the integral in the rhs of \eqref{toe-integral-exp} is positive. So we omit the details.

This prove part (i) of the lemma. \\

\vskip3pt 

%

\noindent To prove part (ii), note that for the asymmetric Hankel link function, 
\begin{align*}
& \xi_{\pi}(i)= \xi_{\pi}(j) \ \ \text{ if and only if }\ \ t_i=t_j \ \ \text{ and }\\
&  \sgn(\pi(i)-\pi(i-1))= \sgn(\pi(j)-\pi(j-1)) \ \ \text{ if } i \ \text{ and } j \ \text{ are of same parity, or}\\
& \sgn(\pi(i)-\pi(i-1))= \sgn(\pi(j-1)-\pi(j)) \ \ \text{ if } i \ \text{ and } j \ \text{ are of opposite parity}.
\end{align*}
Let 
\begin{equation}\label{def-even gen}
\mathcal{E}_{\boldsymbol{\omega}}= \{0,i_j; i_j \ \text{ is even }, 1\leq j \leq b\}
\end{equation}
and 
\begin{equation}\label{def-odd gen}
\mathcal{O}_{\boldsymbol{\omega}}= \{i_j; i_j \ \text{ is odd }, 1\leq j \leq b\}.
\end{equation}
  For every $j \in \{1,2, \ldots, b\}$, let 
  \begin{align}\label{def-odd-even}
  C_{i_j}^{o} & = \{i;   \xi_{\pi}(i)= \xi_{\pi}(i_j), i,i_j \ \text{ are of opposite parity },0 \leq i \leq 2k \}\\
  C_{i_j}^{e} & = \{i;  \xi_{\pi}(i)= \xi_{\pi}(i_j),i,i_j \ \text{ are of same parity },0 \leq i \leq 2k \}
  \end{align}
  Using the notations as in the proof of part (i), we now have that 
\begin{align*}
\big|\Pi_{S_{H}}(\boldsymbol{\omega})\big|= \big|\big\{ & (v_0,v_1,\ldots,v_{2k}):v_{2i}\in U_p, v_{2i-1}\in U_n \text{ for } 0\leq i\leq k, v_0=v_{2k},
v_i= L_{i,n}^H(v_S),\\
& \sgn(y_n L_{i,n}^H(v_S)- L_{i-1,n}^H(v_S))= \sgn(y_n v_{i_j}- L_{i-1,n}^H(v_S))\ \text{ when } i_j\in \mathcal{E}_{\boldsymbol{\omega}} \ \text{ and } i \in C_{i_j}^{e} \ \\
&\text{ or } \sgn(y_nL_{i-1,n}^H(v_S)- L_{i,n}^H(v_S) )= \sgn(y_n v_{i_j}- L_{i_j-1,n}^H(v_S))\ \text{ when } i_j\in \mathcal{E}_{\boldsymbol{\omega}} \ \text{ and } i \in C_{i_j}^{o},\\
& \sgn(y_n L_{i,n}^H(v_S)- L_{i-1,n}^H(v_S))= \sgn(y_n L_{i_j-1,n}^H(v_S)-  v_{i_j-1} )\ \text{ when } i_j\in \mathcal{O}_{\boldsymbol{\omega}} \ \text{ and } i \in C_{i_j}^{e} \ \\
&\text{ or } \sgn(y_nL_{i-1,n}^H(v_S)-  L_{i,n}^H(v_S) )= \sgn(y_n v_{i_j}- L_{i-1,n}^H(v_S))\ \text{ when } i_j\in \mathcal{O}_{\boldsymbol{\omega}} \ \text{ and } i \in C_{i_j}^{o}  \big \}\big|.
 \end{align*}
 As $\Pi_{S_{H}}(\boldsymbol{\omega})\subset \Pi_{S_{H^{(s)}}}(\boldsymbol{\omega})$. If $\boldsymbol {\omega}$ is a word with $b$ distinct letters but not symmetric, by part (i), $\frac{1}{p^{r+1}b^{n-r}}\big|\Pi_{S_{H}}(\boldsymbol{\omega})\big| \rightarrow 0$ as $n \rightarrow \infty$.
 
Next let $\boldsymbol {\omega} \in S_b(2k)$ with $(r+1)$ even generating vertices. Clearly for $\boldsymbol {\omega}$, $|\mathcal{E}_{\boldsymbol{\omega}}|=r+1$ and $|\mathcal{O}_{\boldsymbol{\omega}}|=b-r$. Now suppose, 
\begin{align}\label{H-sgn-indicator}
f_{n}^H(v_S)= \displaystyle\prod_{j=1}^b \bigg[ \prod_{i_j \in \mathcal{E}_{\boldsymbol{\omega}}} &\big (\prod_{i \in C_{i_j}^e}\textbf{1}(\sgn(y_n L_{i,n}^H(v_S)- L_{i-1,n}^H(v_S))= \sgn(y_n v_{i_j}- L_{i-1,n}^H(v_S)))\nonumber\\
&\prod_{i \in C_{i_j}^o} \textbf{1}(\sgn(y_nL_{i-1,n}^H(v_S)- L_{i,n}^H(v_S) )= \sgn(y_n v_{i_j}- L_{i_j-1,n}^H(v_S)))  \big)\nonumber\\
& \prod_{i_j \in \mathcal{O}_{\boldsymbol{\omega}}}\big( \prod_{i \in C_{i_j}^e} \textbf{1}( \sgn(y_n L_{i,n}^H(v_S)- L_{i-1,n}^H(v_S))= \sgn(y_n L_{i_j-1,n}^H(v_S)-  v_{i_j-1} ))\nonumber \\
& \prod_{i \in C_{i_j}^o}\sgn(y_nL_{i-1,n}^H(v_S)-  L_{i,n}^H(v_S) )= \sgn(y_n v_{i_j}- L_{i-1,n}^H(v_S))
\big) \bigg].
\end{align}
and let $f^H$ be the limit of $f_n^H$ as $y_n\rightarrow y>0$. Then 
\begin{align}\label{H-integral}
\displaystyle \lim_{n \rightarrow \infty} \frac{1}{p^{r+1}n^{b-r}} \Pi_{S_H}(\boldsymbol{\omega})= \int_{0}^1\int_0^1 \cdots \int_0^1\textbf{1}(0 \leq L_{i}^H(v_S)\leq 1)f^H(v_S) \ dv_S
\end{align}
where $dv_S= \prod_{j=0}^bdv_{i_j}$ is the $(b+1)-$dimensional Lebesgue integral on $[0,1]^{b+1}$, $L_{i}^H$ is the limit of the linear combination $L_{i,n}^H$ as $y_n\rightarrow y$ and $f^H$ is the function defined above via \eqref{H-sgn-indicator}. 

This completes the proof of part (ii).

 
 \end{proof}

  Let $\lfloor \cdot\rfloor$ denotes the greatest integer function.
\begin{lemma}\label{lem:rev}
Suppose $\boldsymbol {\omega}$ is a word of length $2k$ with $b$ distinct letters and $(r+1)$ even generating vertices. Then 
 \begin{itemize}
 \item[(i)]$\displaystyle \lim_{n \rightarrow \infty}\frac{1}{p^{r+1}n^{b-r}} | \Pi_{S_{R^{(s)}}}(\boldsymbol {\omega})|= {\lfloor y \rfloor}^{k-(r+1)} + \alpha_{R^{(s)}}(\boldsymbol {\omega})>0$ if and only if $\boldsymbol {\omega}$ is a symmetric word. 
 \item[(ii)] $\displaystyle \lim_{n \rightarrow \infty}\frac{1}{p^{r+1}n^{b-r}} | \Pi_{S_R}(\boldsymbol {\omega})|= {\lfloor y \rfloor}^{k-(r+1)}+ \alpha_{R}(\boldsymbol {\omega})$ if and only if $\boldsymbol {\omega}$ is a symmetric word.
 \end{itemize}
\end{lemma} 
The proof of this lemma borrows the main ideas from the proof of Lemma \ref{lem:han} and is given in details in Section \ref{appendix}.\\

\vskip3pt
 
Recall the sequence $a_{2n}= \frac{1}{2}{{2n} \choose n}$ from Lemma 5.2 in \citep{bose2021some}. 
 \begin{lemma}\label{lem:c-sc}
Suppose $\boldsymbol {\omega}$ is a word of length $2k$ with $b$ distinct letters and $(r+1)$ even generating vertices. Then 
 \begin{itemize}
 \item[(i)]$\displaystyle \lim_{n \rightarrow \infty}\frac{1}{p^{r+1}n^{b-r}} | \Pi_{S_{C^{(s)}}}(\boldsymbol {\omega})|= a_{\boldsymbol{\omega}}\big[ {\lfloor y \rfloor}^{k-(r+1)} + \alpha_{C^{(s)}}(\boldsymbol {\omega})\big]>0$ if and only if $\boldsymbol {\omega}$ is an even word and $a_{\boldsymbol{\omega}}$ is the multiplicative extension of the sequence $a_{2n}$ when $\boldsymbol{\omega}$ is considered as a partition in $\{1,2,\ldots,2k\}$.
 \item[(ii)] $\displaystyle \lim_{n \rightarrow \infty}\frac{1}{p^{r+1}n^{b-r}} | \Pi_{S_C}(\boldsymbol {\omega})|= {\lfloor y \rfloor}^{k-(r+1)}+ \alpha_{C}(\boldsymbol {\omega})>0$ if and only if $\boldsymbol {\omega}$ is a symmetric word.
 \end{itemize}
\end{lemma} 

The proof of this lemma borrows a lot of ideas from the proofs of Lemmas \ref{lem:toe-S} and lemma \ref{lem:toe} and is given in details in Section \ref{appendix}.

\subsubsection{Proof of Theorem \ref{res:main}}\label{proof of other theorem}

 \begin{lemma}\label{meanzero}
Recall the matrix $Z_A$ from Theorem \ref{res:main}. Under the assumptions of Theorem \ref{res:main}, suppose, $\widetilde{Z_A}$ is the $p \times n $ matrix whose entries are $(y_{i}-\mathbb{E}y_{i})$ and thus have mean 0. Then the LSD of $S_{Z_A}$ and $S_{\widetilde{Z_A}}$ are same.
 \end{lemma}
 \begin{proof}
In Step 1 of the proof of Theorem \ref{res:XXt}, we dealt with the same problem but for a bi-sequence of random variables $\{x_{ij,n}\}$. The same proof can be adapted in this case replacing by the sequence $\{x_{i,n}\}$. Therefore, condition \eqref{fkeven} is true for $\widetilde{Z_A}$. Similarly we can show that \eqref{fkodd} is true for $\widetilde{Z_A}$. Hence, Assumption B holds for $\widetilde{Z_A}$.

Now from Lemma \ref{lem:Emetric}, 
\begin{align*}
L^4\big(\mathbb{E}F^{S_{Z_A}}, \mathbb{E}F^{S_{\widetilde{Z_A}}}\big) & \leq  \frac{2}{p^2}(\mathbb{E}\Tr(Z_AZ_A^T+\widetilde{Z_A}\widetilde{Z_A}^T))(\mathbb{E}\Tr[(Z_A-\widetilde{Z_A})(Z_A-\widetilde{Z_A})^T]) \\
&\leq \frac{2}{p}\bigg(\displaystyle \sum_{i=-(n+p)}^{n+p}cn \mathbb{E}\big( 2y_{i}^2+ (\mathbb{E}y_{i})^2- 2y_{ij}\mathbb{E}y_{ij}\big)\bigg)\frac{1}{p}\bigg(\sum_{i=-(n+p)}^{n+p}cn(\mathbb{E}y_{i})^2\bigg)
\end{align*}
where $c$ is a constant depending on the link function of the matrix. Observe that for all matrices with link functions (i)-(viii), the second inequality is true due to the structure of the link functions. The second factor of the rhs in the above inequality is bounded by $$\frac{2}{y_n} (n+p)(\sup_{i}\mathbb{E}y_{i})^2=\frac{2}{y_n}(\sup_{i}\sqrt{n}\mathbb{E}y_{i})^2+\frac{2}{y_n}(\sup_{i}\sqrt{p}\mathbb{E}y_{i})^2 \rightarrow 0  \ \ \text{ as } n \rightarrow \infty, \  p/n \rightarrow y>0 \ \text{ by } \eqref{fkodd}.$$
Again, $\mathbb{E}\big[\frac{1}{p}\sum_{i}y_{i}^2\big] \rightarrow \int_{[0,1]^2} f_2(x,y)\ dx\ dy$.  Therefore the first term of the rhs in the inequality is bounded uniformly and hence $L^4\big(\mathbb{E}F^{S_{Z_A}}, \mathbb{E}F^{S_{\widetilde{Z_A}}}\big) \rightarrow 0$ as $p \rightarrow \infty$. Thus we can assume that the entries of $Z_A$ have mean 0.
 \end{proof}

 \begin{lemma}\label{truncation}
 Under the conditions of Theorem \ref{res:main}, if the EESD of the matrices $S_{Z_A}$ converges weakly to $\mu_{A}$, then, with the assumption \eqref{S-truncationtoe-XXt}, the EESD of $S_{A}$ (where $A$ is the non-truncated version of $Z$) converges weakly to $\mu_A$.
 \end{lemma}
\begin{proof}
Observe that from Lemma \ref{lem:Emetric}, we have 
\begin{align}\label{levy-EXXt}
L^4(\mathbb{E}F^{S},\mathbb{E}F^{S_{Z_A}}) &\leq \frac{2}{p^2}(\mathbb{E}\Tr(AA^T+Z_AZ_A^T))(\mathbb{E}\Tr[(A-Z_A)(A-Z_A)^T]) \nonumber\\
& \leq \frac{2}{p} \bigg(\displaystyle 2cn \sum_{i=-(n+p)}^{n+p}\mathbb{E}[y_{i}^2]+ cn \sum_{i=-(n+p)}^{n+p}\mathbb{E}[x_{i}^2\boldsymbol{1}_{[|x_{i}|>t_n]}] \bigg)\bigg(\frac{1}{p}\sum_{i=-(n+p)}^{n+p} cn \mathbb{E}[x_{i}^2\boldsymbol{1}_{[|x_{i}|>t_n]}]\bigg).
\end{align}
The second factor in the above inequality tends to zero a.s. (or in probability) as $n \rightarrow \infty$ from \eqref{S-truncationtoe-XXt}. 
 Again, the first factor is uniformly bounded as in the proof of Lemma \ref{meanzero}. Thus $L^4(\mathbb{E}F^{S_A},\mathbb{E}F^{S_{Z_A}})\rightarrow 0$ as $p \rightarrow \infty$. 

This completes the proof of the lemma.
\end{proof} 

Now we will prove Theorem \ref{res:main}. The arguments in the proof of the different parts are often repetitive. So we prove part (i) in details and omit the elaborate arguments for the other parts.
\begin{proof}[\textbf{Proof of Theorem \ref{res:main}}]
We shall prove the theorem in different parts.  \\
\noindent \textbf{(i)}: Let $A=T^{(s)}$. First observe that from Lemma \ref{truncation}, it is enough to prove that the EESD of $S_{Z_A}$ converges to $\mu_{T^{(s)}}$. Further, from Lemma \ref{meanzero} we may assume that $\mathbb{E}(y_{i})=0$. Therefore it suffices to verify the first moment condition and the Carleman's condition for $S_{Z_A}$.

As $\mathbb{E}(y_{i})=0$, from \eqref{genmoment-AAt}, if $\displaystyle\lim_{p \rightarrow \infty} \frac{1}{p} \sum_{\pi \in \Pi_{S_{T^{(s)}}}(\boldsymbol{\omega})} \mathbb{E}(Y_{\pi})$ exists for every matched word $\boldsymbol{\omega}$ of length $2k$ with $b$ distinct letters and $(r+1)$ even generating vertices ($k \geq 1, 1 \leq b \leq b, 0 \leq r \leq (b-1)$), then the first moment condition would follow.

Suppose $\boldsymbol {\omega}$ is a word with $b$ distinct letters, $(r+1)$ even generating vertices and the distinct letters appear $k_1,k_2,\ldots,k_b$ times. Let the $j$th distinct letter appear at $(\pi(i_j-1),\pi(i_j))$th position for the first time. Denote $(\pi(i_j-1),\pi(i_j))$ as $(m_j,l_j)$. 
Let us now recall $v_{i}, 1 \leq i \leq 2k$ and $s_i, u_i, 1 \leq i \leq 2k$ as defined in Lemma \ref{lem:toe-S}.

First, let $\boldsymbol {\omega} \notin E(2k)$. Suppose $\boldsymbol{\omega}$ contains $b_1$ distinct letters that appear even number of times and $b_2$ number of distinct letters that appear odd number of times and $b=b_1+b_2$. So
we  assume that for each $\pi \in \Pi(\boldsymbol {\omega})$, 
$k_{j_p}$, $1\leq p \leq b_1$ are even and $k_{j_q}$, $b_1+1\leq q \leq b_1+b_2$ are odd. Hence the contribution of this $\boldsymbol {\omega}$ to  \eqref{momentnoniid-XXt} is as follows:
\begin{align}\label{finitesum-nons2k}
 \frac{1}{pn^{b_1+b_2-\frac{1}{2}}} \sum_{S} \prod_{p=1}^{b_1} f_{k_{j_p}}(|s_{j_p}|)   \prod_{q=b_1+1}^{b_1+b_2} n^{\frac{b_2-1/2}{b_2}} \mathbb{E}\Big[y_{|s_{j_q}|}^{k_{j_q}}\Big].
\end{align}
For $n$ large, $n^{\frac{b_2-1/2}{b_2}} \mathbb{E}[y_{(s_{j_q}-2)\ (\text{mod }n)}^{k_{j_q}}]<1$ for any $b_1+1\leq q\leq b_1+b_2$ and $\prod_{p=1}^{b_1} f_{k_{j_p}}(|s_{j_p}|) \leq M$ (independent of $n$). Now as $\boldsymbol{\omega} \notin E_b(2k)$ and $p/n \rightarrow y >0$, from Lemma \ref{lem:toe-S} we have, $|S|\leq b$. Hence, as $p,n \rightarrow \infty$ and $p/n \rightarrow y>0$, \eqref{finitesum-nons2k} goes to 0. Thus any word that is not even, contributes 0 to the limiting moments.

Now let $\boldsymbol {\omega}\in E_b(2k)$. Let $\mathcal{E}_{\boldsymbol{\omega}}= $ and $\mathcal{O}_{\boldsymbol{\omega}}$ be as in \eqref{def-even gen} and \eqref{def-odd gen}. Clearly, as observed in  Lemma \ref{lem:toe}, there are $ \prod_{i=1}^b {{k_i-1} \choose \frac{k_i}{2}}$ combination of equations for the $s_j$'s (and hence $v_j$'s) for determining the non-generating vertices, once the generating vertices are chosen. Let us denote a generic combination of the $v_j$'s by $L_{\boldsymbol{\omega}}^T$ (see \eqref{integral-gen}). For each of the combination of equations we get positive (possibly different) contribution (see Lemma \ref{lem:toe}). 
 Then the contribution of each  combination $L_{\boldsymbol{\omega}}^T$ corresponding to the word $\boldsymbol {\omega}$ can be written as 
\begin{align}\label{finitesum-toep}
 y_n^{r}\frac{1}{p^{r+1}n^{b-r}} \sum_S  \prod_{j=1}^b  \bigg(\prod_{i_j \in \mathcal{E}_{\boldsymbol{\omega}}} & f_{k_j,n}\big(|v_{m_j}-y_nv_{l_j}|\big)\prod_{i_j \in \mathcal{O}_{\boldsymbol{\omega}}}f_{k_j,n}\big(|y_nv_{m_j}-v_{l_j}|\big)\bigg) \nonumber\\
&  \textbf{1}(0 \leq L_{i}^T(v_S) \leq 1,\ \forall \  i \in S^{\prime}),  
\end{align}
where $S$ is the set of  distinct generating vertices and $S^{\prime}$ is   the set of indices of the non-generating vertices of $\boldsymbol {\omega}$.
 By abuse of notation let $m_1$ and $l_j, 1\leq j\leq b$ denote the indices of the generating vertices. Therefore as $p \rightarrow \infty$, the contribution of $\boldsymbol {\omega}$ in \eqref{momentnoniid-XXt} is given by 
\begin{align}\label{limit-toe-S}
 y^r\sum_{L_{\boldsymbol{\omega}}^T}\int_{0}^{1} \int_{0}^{1}\int_{0}^{1} \cdots  \int_{0}^{1}  \prod_{j=1}^b  \bigg( & \prod_{i_j \in \mathcal{E}_{\boldsymbol{\omega}}}f_{k_j,n}\big(|x_{m_j}-yx_{l_j}|\big)\prod_{i_j \in \mathcal{O}_{\boldsymbol{\omega}}}f_{k_j,n}\big(|yx_{m_j}-x_{l_j}|\big)\bigg) \nonumber\\
 &  \textbf{1}(0 \leq L_{i}^T(x_S) \leq 1,\ \forall \  i \in S^{\prime})\ dx_S, 
\end{align}
where $dx_S= dx_{m_1}dx_{l_1}\cdots dx_{l_b}$ denotes the $(b+1)$-dimensional Lebesgue measure on $[0,1]^{b+1}$ and $0<y =\lim p/n$. As for each $k \geq 1$, there are finitely many even words, the first moment condition is established. 
Hence we have, 
\begin{align}\label{toe1-S}
 &\displaystyle\lim_{n \rightarrow \infty}\frac{1}{p}\mathbb{E}[\Tr(S_{Z_A})^{k}] =  \displaystyle \sum_{b=1}^k \sum_{r=0}^{b-1}\sum_{\underset{(r+1) \text{ even generating vertics}}{\sigma \in E_b(2k) \text{ with }}} y^r \sum_{L_{\sigma}^T}\int_{0}^{1} \int_{0}^{1}\int_{0}^{1} \cdots  \int_{0}^{1} \nonumber\\
&  \prod_{j=1}^b  \bigg(  \prod_{i_j \in \mathcal{E}_{\boldsymbol{\omega}}}f_{k_j,n}\big(|x_{m_j}-yx_{l_j}|\big) 
\prod_{i_j \in \mathcal{O}_{\boldsymbol{\omega}}}f_{k_j,n}\big(|yx_{m_j}-x_{l_j}|\big)\bigg)
  \textbf{1}(0 \leq L_{i}^T(x_S) \leq 1,\ \forall \  i \in S^{\prime})\ dx_S.
\end{align}

Now we show that the limits, $ \lim_{n \rightarrow \infty}\frac{1}{p}\mathbb{E}[\Tr(S_{Z_A})^{k}]=\gamma_{k}, k \geq 1$ determines a unique distribution.

If $y\leq 1$, 
\begin{align*}
\gamma_{k}= &\lim_{n \rightarrow \infty}\frac{1}{p}\mathbb{E}[\Tr(S_{Z_A})^{k}]
 \leq \displaystyle \sum_{\sigma \in E(2k)} M_{\sigma}
 \leq \displaystyle \sum_{\sigma \in \mathcal{P}(2k)} M_{\sigma} 
 =\alpha_{k}.
\end{align*}
As $\{\alpha_{k}\}$ satisfies Carleman's condition, $\{\gamma_{k}\}$  does so. Hence the sequence of moments $\{\gamma_{k}\}$ determines a unique distribution. 

If $y>1$, $y^r\leq y^b, 0 \leq r \leq b, 1 \leq b \leq k$ and hence 
\begin{align*}
\gamma_{k}= &\lim_{n \rightarrow \infty}\frac{1}{p}\mathbb{E}[\Tr(S_{Z})^{k}]
 \leq \displaystyle \sum_{\sigma \in E(2k)} y_{\sigma}M_{\sigma}
 \leq \displaystyle \sum_{\sigma \in \mathcal{P}(2k)} y^{2k}M_{\sigma} 
 =y^{2k}\alpha_{k}.
\end{align*}
As, $y\in (1, \infty)$ and $\alpha_{k}$ satisfies Carleman's condition, $\{\gamma_{k}\}$  does so. Hence the sequence of moments $\{\gamma_{k}\}$ determines a unique distribution.

Therefore, there exists a measure $ \mu_{T^{(s)}} $ with moment sequence $\{\gamma_{k}\}$ such that $\mathbb{E}\mu_{S_{Z_A}}$ converges to $ \mu_{T^{(s)}}$, whose moments are given as in \eqref{toe1-S}.

This completes the proof of part (i).\\

\vskip5pt 

\noindent \textbf{(ii)} Let $A=T$. Just as in part (i), it suffices to verify the first moment condition and the Carleman's condition for $S_{Z}$. 

As $\mathbb{E}(y_{i})=0$, from \eqref{genmoment-AAt}, if $\displaystyle\lim_{p \rightarrow \infty} \frac{1}{p} \sum_{\pi \in \Pi_{S_{T}}(\boldsymbol{\omega})} \mathbb{E}(Y_{\pi})$ exists for every matched word $\boldsymbol{\omega}$ of length $2k$ with $b$ distinct letters and $(r+1)$ even genrating vertices ($k \geq 1, 1 \leq b \leq b, 0 \leq r \leq (b-1)$), then the first moment condition follows.

Now suppose $\boldsymbol {\omega}$ is a word with $b$ distinct letters each letter appearing $k
_1,k_2,\ldots,k_b$ times. From this point, we borrow all notations from part (i).

Let $v_i=\pi({i})/n$ as defined in Lemma \ref{lem:toe} and $U_n= \{0,1/n,2/n,\ldots,(n-1)/n\}$. 

If $\boldsymbol{\omega}$ is not an even word, then its contribution to the limiting moments is 0. This follows using the same argument as part (i). 

Now suppose $\boldsymbol{\omega} \in E_b(2k)\setminus S_b(2k)$. Then the contribution of this $\boldsymbol {\omega}$ can be written as 
\begin{align}\label{finitesum-toe}
 y_n^{r}\frac{1}{p^{r+1}n^{b-r}} \sum_S  \prod_{j=1}^b  \bigg(\prod_{i_j \in \mathcal{E}_{\boldsymbol{\omega}}} & f_{k_j,n}\big(y_nv_{l_j}-v_{m_j}\big)\prod_{i_j \in \mathcal{O}_{\boldsymbol{\omega}}}f_{k_j,n}\big(y_nv_{m_j}-v_{l_j}\big)\bigg) \nonumber\\
&  \textbf{1}(0 \leq L_{i}^T(v_S) \leq 1,\ \forall \  i \in S^{\prime}), 
 \end{align}
where $S$ is the set of  distinct generating vertices and $S^{\prime}$ is   the set of indices of the non-generating vertices of $\boldsymbol {\omega}$ and $y_n=p/n$. From Lemma \ref{lem:toe}, observe that for $j\neq 1$, $m_j$ can be written as a linear combination of $\{l_i; 1\leq i\leq j-1\}$ and $m_1$. By abuse of notation let $m_1$ and $l_j, 1\leq j\leq b$ denote the indices of the generating vertices. Then, as $p \rightarrow \infty$, the above sum goes to 
\begin{align}
y^r \int_{0}^{1} \int_{0}^{1}\int_{0}^{1} \cdots  \int_{0}^{1}  \prod_{j=1}^b  \bigg( & \prod_{i_j \in \mathcal{E}_{\boldsymbol{\omega}}}f_{k_j,n}\big(yx_{l_j}-x_{m_j}\big)\prod_{i_j \in \mathcal{O}_{\boldsymbol{\omega}}}f_{k_j,n}\big(yx_{m_j}-x_{l_j}\big)\bigg) \nonumber\\
 &  \textbf{1}(0 \leq L_{i}^T(x_S) \leq 1,\ \forall \  i \in S^{\prime})\ dx_S, 
\end{align}
where 
$dx_S= dx_{m_1}dx_{l_1}dx_{l_2}\cdots dx_{l_b}$.

By Lemma \ref{lem:toe}, it follows that the above integral reduces to a $c$  dimensional integral where $c \leq b$, if $\boldsymbol{\omega} \notin S_b(2k)$ as $x_{m_1}$ and the $x_{l_j}$'s then satisfy a linear equation (see proof of Lemma \ref{lem:toe}). As a result, the contribution of $\boldsymbol {\omega}$ as described in \eqref{limit-toe} is equal to 0 if $\boldsymbol {\omega}  \in E_b(2k)\setminus S_b(2k)$.

Now let $\boldsymbol {\omega}\in S_{b}(2k)$. Then the contribution  the word $\boldsymbol {\omega}$ to  the limiting moments can be written as
\begin{align}\label{limit-toe}
y^r \int_{0}^{1} \int_{0}^{1}\int_{0}^{1} \cdots  \int_{0}^{1}  \prod_{j=1}^b  \bigg( & \prod_{i_j \in \mathcal{E}_{\boldsymbol{\omega}}}f_{k_j,n}\big(yx_{l_j}-x_{m_j}\big)\prod_{i_j \in \mathcal{O}_{\boldsymbol{\omega}}}f_{k_j,n}\big(yx_{m_j}-x_{l_j}\big)\bigg) \nonumber\\
 &  \textbf{1}(0 \leq L_{i}^T(x_S) \leq 1,\ \forall \  i \in S^{\prime})\ dx_S, 
\end{align}
where 
$dx_S= dx_{m_1}dx_{l_1}dx_{l_2}\cdots dx_{l_b}$ is the $(b+1)$-dimensional Lebesgue measure on $[0,1]^{(b+1)}$ and $0< y = \lim y_n$. As for each $k \geq 1$, there are finitely many symmetric words each of which contributes \eqref{limit-toe} to the $k$th limiting moment, the first moment condition holds true.

Therefore for $k\geq 1$, 
 \begin{align}\label{toe2}
 \displaystyle\lim_{n \rightarrow \infty}\frac{1}{n}\mathbb{E}[\Tr(S_{Z_A})^{k}]= \displaystyle \sum_{b=1}^k \sum_{r=0}^{b-1}\sum_{\underset{(r+1) \text{ even generating vertics}}{\sigma \in S_b(2k) \text{ with }}} y^r \int_{0}^{1} \int_{0}^{1}\int_{0}^{1} \cdots  \int_{0}^{1} \nonumber\\
  \prod_{j=1}^b  \bigg(  \prod_{i_j \in \mathcal{E}_{\boldsymbol{\omega}}}f_{k_j,n}\big(yx_{l_j}-x_{m_j}\big) 
\prod_{i_j \in \mathcal{O}_{\boldsymbol{\omega}}}f_{k_j,n}\big(yx_{m_j}-x_{l_j}\big)\bigg)
 &  \textbf{1}(0 \leq L_{i}^T(x_S) \leq 1,\ \forall \  i \in S^{\prime})\ dx_S.
 \end{align}

\vskip3pt
As $S(2k)\subset E(2k)$ and the integrand in \eqref{toe2} is bounded, the same arguments as in part (i) are applicable. Thus the Carleman's condition for $S_Z$ is satisfied. 
%
%
Therefore, there exists a measure $ \mu_{T} $ with moment sequence $\{\gamma_{k}\}$ such that $\mathbb{E}\mu_{S_{Z}}$ converges to $ \mu_{T}$.
This proves part  (ii).\\
\vskip2pt 

\noindent \textbf{(iii)} Let $A=H^{(s)}$. Since the arguments are similar to the previous parts, we resolve to describe the limiting moments from this point onward.

For $k\geq 1$, 
 \begin{align}\label{han1}
 \beta_k(\mu_{H^{(s)}})= \displaystyle \sum_{b=1}^k \sum_{r=0}^{b-1}\sum_{\underset{(r+1) \text{ even generating vertics}}{\sigma \in S_b(2k) \text{ with }}} y^r \int_{0}^{1} \int_{0}^{1}\int_{0}^{1} \cdots  \int_{0}^{1} \nonumber\\
  \prod_{j=1}^b  \bigg(  \prod_{i_j \in \mathcal{E}_{\boldsymbol{\omega}}}f_{k_j,n}\big(x_{m_j}+yx_{l_j}\big) 
\prod_{i_j \in \mathcal{O}_{\boldsymbol{\omega}}}f_{k_j,n}\big(yx_{m_j}+x_{l_j}\big)\bigg)
 &  \textbf{1}(0 \leq L_{i}^H(x_S) \leq 1,\ \forall \  i \in S^{\prime})\ dx_S.
 \end{align}
%
%
%
%

\vskip5pt

\noindent \textbf{(iv)} Let $A=H$. 
We get that there exist a measure $\mu_H$ such that EESD of $S_{H}$ converges to $\mu_H$, whose moments are as follows:
\begin{align}\label{han2}
 \beta_k(\mu_H)= \displaystyle \sum_{b=1}^k &  \sum_{r=0}^{b-1}\sum_{\underset{(r+1) \text{ even generating vertics}}{\sigma \in S_b(2k) \text{ with }}}  y^r \int_{0}^{1} \int_{0}^{1}\int_{0}^{1} \cdots  \int_{0}^{1} \nonumber\\
 & \prod_{j=1}^b  \bigg(  \prod_{i_j \in \mathcal{E}_{\boldsymbol{\omega}}}f_{k_j,n}\big(\sgn(yx_{l_j}-x_{m_j})(x_{m_j}+yx_{l_j})\big) 
\prod_{i_j \in \mathcal{O}_{\boldsymbol{\omega}}}f_{k_j,n}\big(\sgn(yx_{m_j}-x_{l_j})(yx_{m_j}+x_{l_j})\big)\bigg)\nonumber\\
 &  \textbf{1}(0 \leq L_{i}^H(x_S) \leq 1,\ \forall \  i \in S^{\prime})f^H(x_S)\ dx_S.
 \end{align}

\noindent \textbf{(v)} Let $A=R^{(s)}$.

%




For any $m \geq 1$, let 
\begin{align}\label{h-def}
h_{2m,n}(x_1,x_2)& = f_{2m,n}(x_1+x_2)\boldsymbol{1}(0 \leq x_1+x_2\leq 1) + f_{2m,n}(x_1+x_2-1)\boldsymbol{1}(x_1+x_2> 1),\nonumber\\
h_{2m}(x_1,x_2)& = f_{2m}(x_1+x_2)\boldsymbol{1}(0 \leq x_1+x_2\leq 1) + f_{2m}(x_1+x_2-1)\boldsymbol{1}(x_1+x_2> 1)
\end{align}
Then for $k\geq 1$ (see Lemma \ref{lem:rev}), 
 \begin{align}\label{rev1}
 \beta_k(\mu_{R^{(s)}})=& \displaystyle \sum_{b=1}^k \sum_{r=0}^{b-1}\sum_{\underset{(r+1) \text{ even generating vertics}}{\sigma \in S_b(2k) \text{ with }}} y^r\nonumber\\
 & \Bigg[ \lfloor y \rfloor^{k-(r+1))}\displaystyle \int_{0}^{1} \int_{0}^{1} \cdots  \int_{0}^{1}  \prod_{j=1}^b  \bigg(  \prod_{i_j \in \mathcal{E}_{\boldsymbol{\omega}}}h_{k_j,n}\big(x_{m_j},yx_{l_j}\big)\prod_{i_j \in \mathcal{O}_{\boldsymbol{\omega}}}h_{k_j,n}\big(yx_{m_j},x_{l_j}\big)\bigg)\ dx_S\nonumber\\
&  + \sum_{\phi \neq S_0\subset S^{-}} \lfloor y \rfloor^{|S^{-}- S_0|} \int_{0}^{1} \int_{0}^{1} \cdots  \int_{0}^{1}  \prod_{j=1}^b  \bigg(  \prod_{i_j \in \mathcal{E}_{\boldsymbol{\omega}}}h_{k_j,n}\big(x_{m_j},yx_{l_j}\big)\prod_{i_j \in \mathcal{O}_{\boldsymbol{\omega}}}h_{k_j,n}\big(yx_{m_j},x_{l_j}\big)\bigg) \nonumber\\
 &  \textbf{1}(F(yL_{2i}^H(x_S)) \leq y-\lfloor y \rfloor,\ \forall \  2i \in S
 _0)\ dx_S\Bigg].
 \end{align}
%
%
%
%

\vskip5pt

\noindent \textbf{(vi)} Let $A=R$. Suppose
\begin{align}\label{h-tilde-def}
\tilde{h}_{2m}(x_,x_2)= & f_{2m}(\sgn(x_2-x_1)(x_1+x_2))\boldsymbol{1}(0 \leq x_1+x_2\leq 1) + \nonumber\\
& f_{2m}(\sgn(x_2-x_1)(x_1+x_2-1))\boldsymbol{1}(x_1+x_2> 1)
\end{align}
The moments of $\mu_R$ are as follows (see \eqref{rc-integral}):
\begin{align}\label{rev2}
 \displaystyle\beta_k(\mu_R)= & \displaystyle \sum_{b=1}^k \sum_{r=0}^{b-1}\sum_{\underset{(r+1) \text{ even generating vertics}}{\sigma \in S_b(2k) \text{ with }}} y^r\nonumber\\
 & \Bigg[ \lfloor y \rfloor^{k-(r+1))}\displaystyle \int_{0}^{1} \int_{0}^{1} \cdots  \int_{0}^{1}  \prod_{j=1}^b  \bigg(  \prod_{i_j \in \mathcal{E}_{\boldsymbol{\omega}}}\tilde{h}_{k_j,n}\big(x_{m_j},yx_{l_j}\big)\prod_{i_j \in \mathcal{O}_{\boldsymbol{\omega}}}\tilde{h}_{k_j,n}\big(yx_{m_j},x_{l_j}\big)\bigg)f^H(x_S) \ dx_S\nonumber\\
&  + \sum_{\phi \neq S_0\subset S^{-}} \lfloor y \rfloor^{|S^{-}- S_0|} \int_{0}^{1} \int_{0}^{1} \cdots  \int_{0}^{1}  \prod_{j=1}^b  \bigg(  \prod_{i_j \in \mathcal{E}_{\boldsymbol{\omega}}}\tilde{h}_{k_j,n}\big(x_{m_j},yx_{l_j}\big)\prod_{i_j \in \mathcal{O}_{\boldsymbol{\omega}}}\tilde{h}_{k_j,n}\big(yx_{m_j},x_{l_j}\big)\bigg) \nonumber\\
 &  \textbf{1}(F(yL_{2i}^H(x_S)) \leq y-\lfloor y \rfloor,\ \forall \  2i \in S
 _0)f^H(x_S)\ dx_S\Bigg]. 
 \end{align}

\noindent \textbf{(vii)} Let $A=C^{(s)}$. In this case, we have, 
\begin{align}\label{sc1}
 \beta_k(\mu_{C^{(s)}})=  \displaystyle \sum_{b=1}^k & \sum_{r=0}^{b-1}\sum_{\underset{(r+1) \text{ even generating vertics}}{\sigma \in E_b(2k) \text{ with }}} y^r  a_{\sigma}\nonumber\\
 & \bigg[ \lfloor y \rfloor^{k-(r+1)} \int_{0}^{1}  \cdots  \int_{0}^{1}  \prod_{j=1}^b   \bigg(\prod_{i_j \in \mathcal{E}_{\boldsymbol{\omega}}}  f_{k_j}\big(\big|1/2- |1/2-|x_{m_j}-yx_{l_j}||\big|\big)\nonumber\\
& \prod_{i_j \in \mathcal{O}_{\boldsymbol{\omega}}}f_{k_j}\big(\big|1/2-|1/2-|yx_{m_j}-x_{l_j}||\big|\big)\bigg) + 
 \displaystyle \sum_{\phi \neq S_0\subset S^{-}}\lfloor y \rfloor^{|S^{-}- S_0|}\nonumber\\
  & \int_{0}^{1}  \cdots  \int_{0}^{1}  \prod_{j=1}^b   \bigg(\prod_{i_j \in \mathcal{E}_{\boldsymbol{\omega}}}  f_{k_j}\big(\big|1/2- |1/2-|x_{m_j}-yx_{l_j}||\big|\big)\nonumber\\
& \prod_{i_j \in \mathcal{O}_{\boldsymbol{\omega}}}f_{k_j}\big(\big|1/2-|1/2-|yx_{m_j}-x_{l_j}||\big|\big)\bigg)\boldsymbol{1}\big( F(yL_{2i}^T(x_{S}))\leq y - \lfloor y \rfloor, \forall 2i \in S_0\big)\bigg]
\end{align}

\noindent \textbf{(viii)} Let $A=C$. Suppose 
\begin{align}\label{f-def}
\eta_{2m}(yx_1,x_2)& = f_{2m}(x_2-yx_1)\boldsymbol{1}(0 \leq x_2-yx_1 \leq 1)+  f_{2m}(1-x_2+yx_1)\boldsymbol{1}(x_2-yx_1<0 )
\end{align}
where $(x_1,x_2)\in [0,1]^2$. Then the limiting moments are given by the following formula:
\begin{align}\label{sc2}
 \beta_k(\mu_C) =  \displaystyle \sum_{b=1}^k & \sum_{r=0}^{b-1}\sum_{\underset{(r+1) \text{ even generating vertics}}{\sigma \in S_b(2k) \text{ with }}} y^r  a_{\sigma}\nonumber\\
 & \bigg[ \lfloor y \rfloor^{k-(r+1)} \int_{0}^{1}  \cdots  \int_{0}^{1}  \prod_{j=1}^b   \bigg(\prod_{i_j \in \mathcal{E}_{\boldsymbol{\omega}}}  \eta_{k_j}\big(x_{m_j}-yx_{l_j}\big)\nonumber\\
& \prod_{i_j \in \mathcal{O}_{\boldsymbol{\omega}}}\eta_{k_j}\big(x_{l_j}-yx_{m_j} \big)\bigg) + 
 \displaystyle \sum_{\phi \neq S_0\subset S^{-}}\lfloor y \rfloor^{|S^{-}- S_0|}\nonumber\\
  & \int_{0}^{1}  \cdots  \int_{0}^{1}  \prod_{j=1}^b  \bigg(\prod_{i_j \in \mathcal{E}_{\boldsymbol{\omega}}}  \eta_{k_j}\big(x_{m_j}-yx_{l_j}\big)
 \prod_{i_j \in \mathcal{O}_{\boldsymbol{\omega}}}\eta_{k_j}\big(x_{l_j}-yx_{m_j} \big)\bigg)\nonumber\\
 &  \boldsymbol{1}\big( F(yL_{2i}^T(x_{S}))\leq y - \lfloor y \rfloor, \forall 2i \in S_0\big)\bigg]
\end{align}
\end{proof}

\subsection{Applications of Theorem \ref{res:main}}\label{application-AAt}
 As the entries are dependent on $i,j,n$, the formula for the limiting moments, as seen in the proof of Theorem \ref{res:main},  can often be very complicated. Here we discuss a few special cases where the limiting moment formula is relatively simple. 
 
 Theorem \ref{res:main} concludes the convergence of the EESD of $S_A$. However, as we will see in the upcoming sections, a.s. convergence of the ESD can be obtained in some cases. To establish the a.s. convergence of the ESD in such cases, we will use Lemma \ref{lem:genmoment}, just as we did in case of the $S$ matrix. Recall the set $Q_{k,4}^b$ from \eqref{def-Q} that was used to establish the fourth moment condition for $S$. Analogous version of Lemma \ref{lem:moment} is not true 
for $S_A$ (see erratum \citep{bose2021some}). However, it can be shown that
\begin{equation}\label{four circuits}
|Q_{k,4}^b|\leq n^{2k+2} \ \ \ \text{for any } 1\leq b \leq 2k.
\end{equation}
For a proof of this fact, see Lemma 1.4.3 (a) in \citep{bose2018patterned}. Even though the proof given there is for the case where the entries are i.i.d., the arguments can be used to prove the same for the $S_{A}-$ link function as $1 \leq \pi(2i)\leq p$ and $1 \leq \pi(2i-1)\leq n$ and $p$ and $n$ are comparable for large $n$. Further, the same arguments can be adapted when the entries are independent and bounded. 

\subsubsection{General triangular i.i.d. entries}\label{triangular iid-AAt}

Let $A$ be one of the $p \times n$ patterned matrices mentioned in Section \ref{introduction}. Suppose for each fixed $n$ the input sequence $\{x_{i,n}: i \geq 0\}$   
are i.i.d. for every fixed $n$, with all moments finite. Assume that for all $k \geq 1$,
 \begin{equation}\label{ck-AAt}
  n \mathbb{E}[x_{0,n}^k]\rightarrow C_k \ \ \text{ as }\ n \rightarrow \infty.
 \end{equation}
Also assume that the moments of the random variable whose cumulants are $\{C_2,C_4,\ldots\}$ satisfy Carleman's condition. We can actually find such variables $\{x_{i,n}\}$ as discussed in Remark \ref{infinte divisible} in Section \ref{triangular iid}. Now observe that Assumption B (i), (ii) and (iii)  are satisfied with $t_n=\infty$ and $f_{2k}\equiv C_{2k}$ for $k \geq 1$. Thus Theorem \ref{res:main} can be applied to conclude that the EESD of $S_{A}$ converges to a probability distribution, $\mu_{A}$. A brief description of the limiting moments is given below.\\
\vskip3pt

\noindent (i) Suppose $A= T^{(s)}$ whose entries satisfy \eqref{ck-AAt}. Thus by part (i) of Theorem \ref{res:main}, the EESD of $S_{T_p^{(s)}}$ converges to $\mu_{T^{(s)}}$ whose moment sequence is given as follows (see \eqref{toe1-S}):
\begin{equation}\label{ck-toe-S}
\beta_k(\mu_{T^{(s)}})= \displaystyle \sum_{b=1}^k\sum_{r=0}^b y^r \sum_{\underset{(r+1) \text{ even generating vertices}}{\pi\in E_b(2k) \text{ with}}} \sum_{L_{\pi}^T} C_{\pi}\int_0^1 \int_0^1 \cdots\int_0^1 \boldsymbol{1}(0 \leq L_{i}^T(x_S)\leq 1, \forall i \in S^{\prime}).
\end{equation}
Note that since an even word can be identified as an even partition, for every $\pi\in E_b(2k)$, $L_{\pi}^T= L_{\boldsymbol {\omega}}^T$ for the corresponding even word with $b$ distinct letters.\\
\vskip2pt
\noindent (ii) By part (ii) of Theorem \ref{res:main} (see \eqref{toe2}):
\begin{equation}\label{ck-toe}
\beta_k(\mu_T)= \displaystyle \sum_{b=1}^k\sum_{r=0}^b y^r \sum_{\underset{(r+1) \text{ even generating vertices}}{\pi\in S_b(2k) \text{ with}}}  C_{\pi} \int_0^1 \int_0^1 \cdots\int_0^1 \boldsymbol{1}(0 \leq L_{i}^T(x_S)\leq 1, \forall i \in S^{\prime})\ dx_S.
\end{equation}
\noindent (iii) By part (iii) of Theorem \ref{res:main}, the EESD of $S_{H^{(s)}}$ converges to $\mu_{H^{(s)}}$ whose moment sequence is as in \eqref{ck-toe}, where the integrand is replaced by $\boldsymbol{1}(0 \leq L_{i}^H(x_S)\leq 1, \forall i \in S^{\prime})$ (see \eqref{han1}).\\
\vskip2pt
\noindent (iv) By part (iv) of Theorem \ref{res:main}, the EESD of $S_{H_p}$ converges to $\mu_{H}$ whose moment sequence is as in \eqref{ck-toe}, where the integrand is replaced by $\boldsymbol{1}(0 \leq L_{i}^H(x_S)\leq 1, \forall i \in S^{\prime}) f^{H}(s)$ (see \eqref{han2}).\\
\vskip2pt
\noindent(v) By part (v) of Theorem \ref{res:main}, the EESD of $S_{R_p^{(s)}}$ converges to $\mu_{R^{(s)}}$ whose moment sequence is given as in \eqref{ck-toe}, where the function inside the summation is (see \eqref{rev1}) $$C_{\pi}\big[ \lfloor y \rfloor^{k-(r+1)}\\+ \sum_{\phi\neq S_0\subset S^{-}}\lfloor y \rfloor^{|S^{-}\setminus S_0|}
\int_0^1 \int_0^1 \cdots\int_0^1 \boldsymbol{1}(F(yL_{2i}^H(x_S))\leq y-\lfloor y \rfloor, \forall 2i \in S_0) \ dx_S\big].$$ 

\noindent (vi) By part (vi) of Theorem \ref{res:main}, $\beta_k(\mu_R)$ is same as $\beta_k(\mu_{R^{(s)}})$, with an extra factor $f^H(s)$ in the integrand.\\
\vskip2pt
\noindent (vii) By part (vii) of Theorem \ref{res:main}, the EESD of $S_{C^{(s)}}$ converges to $\mu_{C^{(s)}}$ whose moment sequence is as in \eqref{ck-toe-S}, where the function inside the summation is (see \eqref{sc1}) $$a_{\pi} C_{\pi}\big[ \lfloor y \rfloor^{k-(r+1)}+ \sum_{\phi\neq S_0\subset S^{-}}\lfloor y \rfloor^{|S^{-}\setminus S_0|}
\int_0^1 \int_0^1 \cdots\int_0^1 \boldsymbol{1}(F(yL_{2i}^T(x_S))\leq y-\lfloor y \rfloor, \forall 2i \in S_0)f^H(x_S) \ dx_S\big].$$
\noindent (viii) By part (viii) of Theorem \ref{res:main}, $\beta_k(\mu_C)$ is as in \eqref{ck-toe}, where the function inside the summation is (see \eqref{sc2}) $$C_{\pi}\big[ \lfloor y \rfloor^{k-(r+1)}+ \sum_{\phi\neq S_0\subset S^{-}}\lfloor y \rfloor^{|S^{-}\setminus S_0|}\int_0^1 \int_0^1 \cdots\int_0^1 \boldsymbol{1}(F(yL_{2i}^T(x_S))\leq y-\lfloor y \rfloor, \forall 2i \in S_0)f^H(x_S) \ dx_S\big].$$ 


\begin{remark}\label{equality of LSD}
\noindent (a) The linear combinations $L_{i}^T$ and $L_{i}^H$ from Lemma \ref{lem:toe-S},  \ref{lem:toe} and \ref{lem:han} play crucial role in the moments of the LSD of $S_A$. Observe that for $S_T$ and $S_{H^{(s)}}$ (see Lemmas \ref{lem:toe} and \ref{lem:han}),  only symmetric words can contribute positively to the limiting moments. Now from the proof of part (i) of Lemma \ref{lem:han} and Lemma \ref{lem:toe}, it follows that when the word is symmetric, after having chosen the generating vertices ($\{v_S\}$), for every $i \in S^{\prime}$, $L_{i}^T(v_S)= L_{i,q}^H(z_S)$. As the $z_i$s are derived by elementary transformations that do not alter the integral, we have $L_i^T(v_S)=L_i^H(v_S)$, for symmetric word. This will be useful in finding relations between $\mu_{T},\mu_{H^{(s)}},\mu_{R^{(s)}},\mu_C$, that we discuss next.\\

\noindent (b) Suppose $y \in \mathbb{N}$. Then, observe that the integrals in parts (v) and (viii) above are zero and thus $$\beta_k(\mu_{R^{(s)}})=\beta_k(\mu_C)= \displaystyle \sum_{\pi \in S(2k)}y^{k-1}C_{\pi}.$$

Further, we can say more about these limits even when $y \notin \mathbb{N}$, using part (a). Recall that the words contributing to the limiting moments for $S_{T},S_{H^{(s)}},S_{R^{(s)}}$ ans $S_C$ are symmetric. Now, from (a) and uniqueness of the limit, it is easy to see that when the variables are triangular i.i.d. and satisfy \eqref{ck-AAt}, we have $\mu_{T}=\mu_{H^{(s)}}$ and $\mu_{R^{(s)}}=\mu_C$.
\end{remark}

\begin{remark}\label{sym-nonsym LSD}
However, in general the LSDs of $S_A$ for symmetric and the asymmetric cases are not identical. For instance, for the Toeplitz and the circulant matrices, this is evident from the moment formula, as the set of partitions that contribute positively to the limiting moments are different in the two cases.
For the Hankel and the reverse circulant, there is an extra factor in the integrand for the asymmetric versions that gives rise to the difference in the limit. We illustrate this for $S_{H^{(s)}}$ and $S_H$ below.  

For all words that are special symmetric, the contributions for the symmetric and asymmetric Hankel are same as there are no further restrictions for the signs arising from \eqref{H-sgn-indicator}. However, if $\boldsymbol{\omega}\in S_b(2k)\setminus SS_b(2k)$, some additional conditions do appear in case of asymmetric Hankel. 

For instance, let us consider the word $abcabc \in S_{3}(6)\setminus SS_{3}(6)$. In case of symmetric Hankel, its contribution to $\mu_{H^{(s)}}$ is 
\begin{align}\label{abcabc-S}
C_{2}^3\int_0^1\int_0^1\int_0^1\int_0^1 \boldsymbol{1}(0 \leq x_0+x_1-x_3, x_2x_0+x_3\leq 1) \ dx_0dx_1dx_2dx_3.
\end{align}
On the other hand, the contribution for the word, $abcabc$ (in case of asymmetric Hankel) to $\mu_H$ is 
\begin{align}\label{abcabc}
C_{2}^3 \int_0^1\int_0^1\int_0^1\int_0^1 & \boldsymbol{1}(0 \leq x_0+x_1-x_3, x_2x_0+x_3\leq 1)\boldsymbol{1}\big(\sgn(x_1-x_0)= \nonumber\\
& \sgn(2x_3-x_0-x_1),\sgn(x_1-x_2)=\sgn(x_2-2x_0-x_1+2x_3),\nonumber\\
& \sgn(x_3-x_2)=\sgn(x_2-2x_0+x_3)\big) \ dx_0dx_1dx_2dx_3.
\end{align}
The integrand in \eqref{abcabc} is less than that in \eqref{abcabc-S} due to the extra restrictions arising from the sign functions. Thus, the $k$th moment of $\mu_{H}$ is in general smaller than that of $\mu_{H^{(s)}}$. A very similar thing occurs in case of $\mu_{R^{(s)}}$ and $\mu_R$.
\end{remark}
\subsubsection{Sparse triangular i.i.d. entries}\label{sparse entries} 
Suppose the input sequence  $\{x_{i,n}: \ i \geq 0\}$ are $\mbox{Ber}(p_n)$ where $np_n \rightarrow \lambda>0$. 
Then  \eqref{ck-AAt} is satisfied with $C_k=\lambda$ for all $k\geq 1$.
Therefore from the discussion  in Section \ref{triangular iid-AAt}, the EESD of $S_{A}$ converges to say $\mu_A$ whose moments are as in (i)-(viii) in Section  \ref{triangular iid-AAt}, where $C_{\pi}=\lambda^{|\pi|}$ for all $\pi\in \mathcal{P}(2k)$.

\subsubsection{I.i.d. Entries} 
	\citep{bose2010limiting} established the LSD of $S_A$ when $A$ is the asymmetric or symmetric versions of Toeplitz,  Hankel, circulant and reverse circulant matrices with entries $\{\frac{1}{\sqrt n}{x_{i}}\}$, where $x_{i}$ are independent and identically distributed with mean 0 and variance 1.	
	. 
Here we show how these LSD results of \citep{bose2010limiting} can be obtained as special cases of Theorem \ref{res:main}.  

First, observe that just like the $S$ matrix, but now dealing with a single sequence of random variables, $\{x_{i,n}\}$, we can show that conditions (i), (ii) and (iii) of Assumption B hold with $f_2\equiv 1$ and $f_{2k}\equiv 0$ for all $k \geq 2$. Then from Theorem \ref{res:main}, we obtain the convergence of the EESD.

The moment formulae are given as in (i)-(viii) in Section \ref{triangular iid-AAt}, where $C_{2}=1$ and $C_{2k}= 0$ for all $k \geq 2$. Thus the words that contribute to the limiting moments are now pair matched. Hence the moments are indeed equal to the ones in \citep{bose2010limiting}. 

 Now, as $S_A$ satisfies \eqref{four circuits}, we have 
\begin{align}\label{fourthmoment-noniid}
& \frac{1}{p^{4}} \mathbb{E}\big[\Tr(S_{A}^k)  -  \mathbb{E}(\Tr(S_{A}^k))\big]^4 = \mathcal{O}(p^{-2})\ \ \text{ and therefore,}\nonumber\\
&  \displaystyle \sum_{p=1}^{\infty}\frac{1}{p^{4}} \mathbb{E}\big[\Tr(S_{A}^k)  -  \mathbb{E}(\Tr(S_{A}^k))\big]^4 < \infty \ \ \ \text{ for every } k\geq 1.
\end{align}
Then using Lemma \ref{lem:genmoment}, we can conclude that $\mu_{S_{A}}$  converges a.s.
 
%
%
 
 \subsubsection{Matrices with variance profile}\label{var-prof}
Suppose the input sequence is $\{\sigma(i/n)x_{i,n};i\geq 0\}$, where $\sigma: [0,1]\rightarrow \mathbb{R}$ is a bounded and Riemann integrable function and   $\{x_{i,n};i \in \mathbb{Z}\}$ are  i.i.d. random variables with mean zero and all moments finite. Assume that $\{x_{i,n};i\geq 0\}$ satisfy \eqref{ck-AAt}. Then the EESD of $S_{A}$ for each of the eight patterns of $A$, converges to a probability distribution whose moments are determined by $\sigma$ and $\{C_{2k}, k \geq 1\}$. This follows from Theorems \ref{res:main} as argued below:

 First observe that the entries  of $A$ satisfy  Assumption B (i) and (ii)  with $t_n=\infty$, $f_{2k} = \sigma^{2k} C_{2k}, \ k \geq 1$. Since $\sigma$ is bounded, Assumption B (iii) is also true. Hence from Theorem \ref{res:main}, we can conclude that the EESD of $S_{A}$ converges. 

 Note that in the i.i.d. situation where each $x_{i,n}$ has the same distribution $F$ for all $i$ and $n$, $C_{2k}=0$ for all $k\geq 2$. Hence the EESD of $S_{A}$ converges.
As $\sigma$ is bounded, \eqref{four circuits} and hence \eqref{fourthmoment-noniid} hold true. Thus we can conclude that $\mu_{S_{A}}$ converges a.s. to the respective limits.
 \subsubsection{Triangular Matrices}
As discussed in Section \ref{var-prof-smatrix}, the LSD of triangular matrices have been studied in \citep{Dykema2002DToperatorsAD}, where the entries of the matrix are i.i.d. Gaussian. Later LSD results were proved in \citep{basu2012spectral} for triangular matrices with other patterns such as Hankel, Toeplitz and symmetric circulant, and with i.i.d. input. The matrices that the authors considered are symmetric and hence the entries $y_{L(i,j),n}$ are of the form $y_{L(i,j),n}= x_{L(i,j),n}\boldsymbol{1}(i+j\leq n+1)$. However, the matrix considered in \citep{Dykema2002DToperatorsAD} is upper triangular, as in \eqref{triangular matrix}. It is natural to ask what happens to such matrices when there are other patterns involved.

 Let $A$ be any of the eight matrices that are being discussed in this article. Let $A^U$ be the matrix whose entries $y_{L(i,j),n}$ are as follows:
\begin{align}\label{triangular}
y_{L(i,j),n}= \begin{cases}
x_{L(i,j),n} & \ \ \text{ if } i\leq j\\
0       & \ \ \text{ otherwise}.
\end{cases}
\end{align}

Then we have the following result.
\begin{result}\label{res:triangular}
Consider the matrices $A^U$. Assume that the variables $\{x_{i,n};i\geq 0\}$ in \eqref{triangular} are i.i.d. random variables with all moments finite, for every fixed $n$.  Also assume that $\{x_{i,n};i\geq 0\}$ satisfy \eqref{ck-AAt}. Then, for each of the eight matrices mentioned above, the EESD of $S_{A^U}$ converges to some probability measure $\mu_{A^U}$ that depends on $\{C_{2k}\}_{k\geq 1}$. 
\end{result}
 \begin{proof}
Define the function $\sigma$ on $[0,1]^2$ as 
$$\sigma(x,y)=\left\{\begin{array}{ll}
1 & \mbox{if }x\leq y,\\
0 & \mbox{otherwise}.
\end{array}
\right.$$ 
Now observe that the entries $y_{L(i,j),n}$ of the matrix $A_n^U$ can be written as $\sigma(i/p,j/n)x_{L(i,j),n}$.

Following the  proofs in Theorem \ref{res:main}, it is easy to see that  the first moment condition holds for $S_{A^U}$. As $||\sigma||\leq 1$, the Carleman's condition also holds for the limiting moment sequence.  
Hence the EESD of $S_{A^U}$ converges to a probability measure $\mu_A$.

\end{proof}

\begin{remark}
\noindent (i) If the entries of $A^U$ are $\frac{y_{i,n}}{\sqrt{n}}$ where $\{y_{i,n};i \geq 0\}$ are as in \eqref{triangular} and $\{x_{i,n};i\geq 0\}_{n \geq 1}$ are i.i.d. random variables with mean 0 and variance 1, then using familiar truncation arguments (as in Theorem 8.1.2 in \citep{bose2018patterned}), the variables $\{y_{i,n};i \geq 0\}$ can be assumed to be uniformly bounded and hence satisfy \eqref{ck-AAt} with $C_2= 1$ and $C_{2k}= 0$ for  $k\geq 2$. Hence from Result \ref{res:triangular}, we obtain the convergence of the EESD. Again it can be verified that \eqref{four circuits} and \eqref{fourthmoment-noniid} are true in this case. Thus the ESD of $S_{A^U}$ converges a.s. to a non-random probability measure.
\end{remark}

\subsubsection{Band matrices}
 Band matrices had been discussed previously in \citep{basak2011limiting}, \citep{liu2011limit}, \citep{popescu2009general},  \citep{liu2011limit} and others. In Section 7.4 in \citep{bose2021some}, the LSD of band matrices where the non-zero entries satisfy \eqref{ck-AAt} had been studied. So it was natural to ask what happens to the LSD of the $A^bA^{bT}, A^BA^{BT}$, where $A^b$ and $A^B$ are matrices with entries $y_{L(i,j)}= x_{L(i,j)}\boldsymbol{1}(L(i,j)\leq m_n)$ and $y_{L(i,j)}= x_{L(i,j)}[\boldsymbol{1}(L(i,j)\leq m_n )+\boldsymbol{1}(L(i,j)\geq n-m_n )$ (see Section 7.4 in \citep{bose2021some}). Here we provide an answer to that question. 
 \begin{result}\label{banding-AAt}
Consider the matrices $A^b$ and $A^B$. Assume that the variables $\{x_{i,n};i\geq 0\}$ associated with the matrices $A^U$ (as in \eqref{triangular}) are i.i.d. random variables with all moments finite, for every fixed $n$ and satisfy \eqref{ck-AAt}. Suppose $\displaystyle \alpha=  \lim_{n \rightarrow \infty}\frac{m_n}{n}>0$. Then, for each of the eight matrices, the EESD of $S_{A^b}$ and $S_{A^B}$ converge to some probability measures $\mu_{\alpha}^b$  and $\mu_{\alpha}^B$ that depend on $\{C_{2k}\}_{k\geq 1}$.   
\end{result}
 For $A^b$, the entries $y_{i}$ can be written as $\sigma_n\big({i}/{n}\big)x_{i}$, where $\sigma_n(x)=\boldsymbol{1}(x \leq \frac{m_n}{n})$. Observe that for $k \geq 1$, $\int \sigma_n^{k}(x)\  dx \rightarrow \int \sigma_0^{k}(x)\  dx$ as $n \rightarrow \infty$, where $\sigma_0(x)= \boldsymbol{1}(x \leq \alpha)$.

For the Type II band versions $R^{(s)B}$ of $R^{(s)}$ and $T^B$ of $T^{(s)}$,  the entries $y_{i}$ can be written as $\sigma_n\big({i}/{n}\big)x_{i}$, where $\sigma_n(x)=\boldsymbol{1}(\frac{m_n}{n}\leq x \leq 1-\frac{m_n}{n})$. Clearly, $\int \sigma_{n}^{k}(x) \ dx \rightarrow \int \sigma_1^{k}(x) \ dx$ as $n \rightarrow \infty$ for each $k \geq 1$ where $\sigma_1= \boldsymbol{1}_{[0,\alpha] \cup [1-\alpha,1]}$. For the Type II band versions $H^B$ of $H^{(s)}$,  the entries $y_{i}$ can be written as $\sigma_n\big({i}/{n}\big)x_{i}$, where $\sigma_n(x)=\boldsymbol{1}(1-m_n/n\leq x \leq 1 + m_n/n)$. Clearly, $\sigma_n$ converges to $\sigma_2= \boldsymbol{1}_{[1-\alpha,1+\alpha]}$. 
 
 Thus in all of the above cases, Assumption B is true with $t_n=\infty$ and $g_{2k}= \sigma_t^{2k}C_{2k}$, $t=0,1, \text{ or } 2$. Thus from Theorem \ref{res:main} the result follows.

In this case too if ${x_{i}}$ are all i.i.d. and the entries of the matrices are $\{\frac{y_{i}}{\sqrt{m_n}}\}$, then additionally \eqref{four circuits} and thereby \eqref{fourthmoment-noniid} holds. Thus the a.s. convergence of the ESDs can be concluded.

%

\section{Appendix}\label{appendix}
\begin{lemma}
Suppose $\{x_{ij};1\leq i \leq p, 1 \leq j \leq j\leq n\}$ are independent variables that satisfy Assumption A and $y_{ij}= x_{ij}\boldsymbol{1}_{[|x_{ij}|\leq t_n]}$. Then 
\begin{itemize}
\item[(i)] $\displaystyle\frac{1}{p} \sum_{i,j} (y_{ij}^2-\mathbb{E}[y_{ij}^2]) \rightarrow 0  \ \ \text{ a.s. as } p \rightarrow \infty. $
\item[(ii)] Additionally if, $\frac{1}{p} \displaystyle \sum_{i,j} x_{ij}^2\boldsymbol{1}_{[|x_{ij}|> t_n]} \rightarrow 0$ a.s. (or in probability), then $\limsup_p \displaystyle\frac{1}{p} \sum_{i,j} x_{ij}^2<\infty$ a.s. (or in probability).  
\end{itemize}

\end{lemma}
\begin{proof}
\begin{itemize}
\item[(i)] Let $\epsilon>0$ be fixed. Then 
\begin{align*}
\mathbb{P}\bigg[\big|\frac{1}{p}\displaystyle \sum_{i,j} (y_{ij}^2-\mathbb{E}[y_{ij}]^2)\big|>\epsilon\bigg]& \leq \frac{1}{\epsilon^4p^4} \mathbb{E}\bigg[ \big(\displaystyle \sum_{i,j} y_{ij}^2-\mathbb{E}[y_{ij}^2])\big)^4\bigg]\\
& = \frac{1}{\epsilon^4p^4} \mathbb{E}\bigg[ \displaystyle \sum_{\underset{j_1,j_2,j_3,j_4}{i_1,i_2,i_3,i_4}} \prod_{l=1}^4\big(y_{i_lj_l}^2 - \mathbb{E}[y_{i_lj_l}^2])\big)^4\bigg].
\end{align*}
As $\{y_{ij}\}$ are independent, the above inequality becomes
\begin{align*}
\mathbb{P}\bigg[\big|\frac{1}{p}\displaystyle \sum_{i,j} (y_{ij}^2-\mathbb{E}[y_{ij}]^2)\big|>\epsilon\bigg]& \leq \frac{1}{\epsilon^4 p^4} \sum_{ij} \mathbb{E}\big[(y_{ij}^2- \mathbb{E}[y_{ij}^2])^4\big] \\ & \  + 6 \frac{1}{\epsilon^4 p^4} \sum_{\underset{j_1,j_2}{i_1,i_2}} \mathbb{E}\big[(y_{i_1j_1}^2- \mathbb{E}[y_{i_1j_1}^2])^2 (y_{i_2j_2}^2- \mathbb{E}[y_{i_2j_2}^2])^2\big].
\end{align*}
Now from \eqref{gkeven}, as $\{g_{2k,n}\}$ are bounded integrable, the first term in the rhs of the above inequality is $\mathcal{O}(\frac{1}{p^3})$ and the second term is $\mathcal{O}(\frac{1}{p^2})$. Therefore, 
\begin{align*}
\displaystyle \sum_p \mathbb{P}\bigg[\big|\frac{1}{p}\displaystyle \sum_{i,j} (y_{ij}^2-\mathbb{E}[y_{ij}]^2)\big|>\epsilon\bigg]< \infty.
\end{align*}
Hence by Borel-Cantelli lemma, $\displaystyle\frac{1}{p} \sum_{i,j} (y_{ij}^2-\mathbb{E}[y_{ij}^2]) \rightarrow 0 \ \ \ \text{ a.s. as } p \rightarrow \infty.$
\end{itemize}

\item[(ii)] Observe that $\displaystyle \sum_{i,j} x_{ij}^2= \sum_{i,j} \big(y_{ij}^2 +x_{ij}^2\boldsymbol{1}_{[|x_{ij}|> t_n]}\big)$. Also note that $\frac{1}{p}\displaystyle\sum_{i,j}\mathbb{E}[y_{ij}]^2 \rightarrow \int g_2(x,y)\ dx \ dy$ as $n,p \rightarrow \infty$. Then by the condition $\frac{1}{p}\displaystyle \sum_{i,j} x_{ij}^2\boldsymbol{1}_{[|x_{ij}|> t_n]} \rightarrow 0$ a.s. (or in probability) and (i), (ii) holds true. 
\end{proof} 

\begin{proof}[\textbf{Proof of Lemma \ref{lem:rev}}]
First suppose $\boldsymbol {\omega} \in \mathcal{P}(2k) \setminus S_b(2k)$. Then from \eqref{SA-S} and Lemma 5.1 in \citep{bose2021some}, and using the fact that $p/n \rightarrow y>0$ as $n\rightarrow \infty$, it is easy to see that $$\displaystyle \lim_{n \rightarrow \infty}\frac{1}{p^{r+1}n^{b-r}} | \Pi_{S_{R^{(s)}}}(\boldsymbol {\omega})|=\lim_{n \rightarrow \infty}\frac{1}{p^{r+1}n^{b-r}} | \Pi_{S_{R}}(\boldsymbol {\omega})|= 0.$$  
Now suppose $\boldsymbol{\omega}$ is a symmetric word with $b$ distinct letters and $(r+1)$ even generating vertices. Suppose $i_1,i_2,\ldots,i_b$ are the positions where new letters made their first appearances. First we fix the generating vertices $\pi(i_j), 0 \leq j\leq b$ where $\pi(i_0)=\pi(0)$. Let 
\begin{align*}
t_i= \pi(i)+ \pi(i-1)  \text{ for } 1 \leq i \leq 2k.
\end{align*}
Now let us first consider the symmetric Reverse circulant link function. 

From \eqref{SA-link}, $\boldsymbol{\omega}[i]= \boldsymbol{\omega}[j]$ if and only if $t_i=t_j\ (\text{mod }n).$ Clearly, for $\boldsymbol{\omega}$, $t_{i_1}=t_1$ and for every $1 \leq i \leq 2k$, 
\begin{equation}\label{S-rc1}
t_i\ =\ t_{i_j}\ (\text{mod}\ n) \  \  \ \text{for some }j \in \{1,2,\ldots,b\}.
\end{equation}
First we fix the generating vertices $\pi(i_j),\ j=0,1,2,\ldots,b$. Let $S=\{\pi(i_j):0 \leq  j \leq b\}  \ \mbox{ and }\ S^{\prime}=\{i : \pi(i) \notin S\}$. For every $i \in S^{\prime}$,
\begin{align}\label{S-rc2}
\pi(i)= & (t_{i_j}+ \pi(i-1))\ (\text{mod } n) \ \ \text{ for some } i_j,j\in \{1,2, \ldots, b\}\nonumber\\
i.e., \pi(i)=& \displaystyle \sum_{j<i} \alpha_{ij}\pi(i_j)\ (\text{mod } n)\ \ \text{ for some }\alpha_{ij}\in \mathbb{Z}
\end{align}
Thus for every $i\in S^{\prime}\setminus \{2k\} $, there exists unique integer $m_{i,n}$ such that 
\begin{align}\label{S-rc3}
 1 \leq \displaystyle \sum_{j<i} \alpha_{ij}\pi(i_j)+ m_{i,n}\leq n.
\end{align}
As we have already fixed the generating vertices, from \eqref{S-rc2} and \eqref{S-rc3} it follows that there is a unique choice for $\pi(2i-1),1 \leq i \leq k$ such that $2i-1\in S^{\prime}$. For all $2i\in S^{\prime}, 1 \leq i< k$, we can have $\lfloor y_n \rfloor$ choices as $1 \leq \pi(2i) \leq p$ and $y_n=p/n$. Moreover, there is an additional choice if 
\begin{align}\label{S-rc4}
\displaystyle \sum_{j<2i} \alpha_{2ij}\pi(i_j)+ m_{2i,n}\leq p - \lfloor y_n \rfloor n.
\end{align}
Next let
  $$
  v_{2i}=  \frac{\pi(i)}{p},\ \ v_{2i-1}=  \frac{\pi(i)}{n}  \text{ for }0 \le i \leq k,\ \\
  L(a)= \max\{m \in \mathbb{Z}\}, \ F(a)= a- L(a). 
   $$
 Also let 
 \begin{equation}\label{def-S-}
  S^{-}= \{2i: 2i\notin S^{\prime} \ \text{ and \eqref{S-rc4} holds true}\}.
 \end{equation}
Now observe that from \eqref{S-rc3} and \eqref{S-rc4} it follows that for every $i \in S^{-}$,
 \begin{align}
 F(y_n L_{2i,n}^H(v_{S}))\leq y_n - \lfloor y_n \rfloor
 \end{align}
where $L_{2i,n}^H$ is the set of linear combinations defined in \eqref{linear-han}. 
 
From \eqref{SA-Pi(omega)} and the discussion above, it is easy to see that for a word $\boldsymbol{\omega}$ of length $2k$, 
\begin{align*}
\frac{1}{p^{r+1}n^{b-r}}\big|\Pi_{S_{R^{(s)}}}(\boldsymbol{\omega})\big|={\lfloor y_n \rfloor}^{k-(r+1)} + \displaystyle \sum_{\phi \neq S_0\subset S^{-}}\lfloor y_n \rfloor^{|S^{-}- S_0|}\big|\big\{v_S: F(y_n L_{2i,n}^H(v_{S}))\leq y_n - \lfloor y_n \rfloor, \textbf{1}(2i \in S_0)\big\}\big|.
 \end{align*} 
 Therefore as $n \rightarrow \infty$,
 \begin{align}\label{S-rc-integral}
 \displaystyle \lim_{n \rightarrow \infty}\frac{1}{p^{r+1}n^{b-r}}\big|\Pi_{S_{R^{(s)}}}(\boldsymbol{\omega})\big|=& \lfloor y \rfloor^{k-(r+1)} + \displaystyle \sum_{\phi \neq S_0\subset S^{-}}\lfloor y \rfloor^{|S^{-}- S_0|} \nonumber\\ 
 &\int_0^1\int_0^1 \cdots \int_0^1 \textbf{1}\big(F(y L_{2i}^H(v_{S}))\leq y - \lfloor y \rfloor) \ \forall 2i\in S_0\big) \ dv_S
 \end{align}
 where $d{v_S}= \prod_{j=0}^b dv_{i_j}$ is the $(b+1)-$dimensional Lebesgue integral on $[0,1]^{b+1}$ and $S^{-}$ is as in \eqref{def-S-}.
 
When $y\geq 1$, the rhs of \eqref{S-rc-integral} is positive. We next show that when $y<1$, the value of the integral $\int_0^1\int_0^1 \cdots \int_0^1 \textbf{1}\big(F(y L_{2i}^H(v_{S}))\leq y - \lfloor y \rfloor) \ \forall 2i\in S^{-}\big) \ dv_S$ is positive. 

First note that as $y<1, \lfloor y \rfloor=0$. Now note that we had previously established in the proof of part (i) of Lemma \ref{lem:han} that for certain values of $v_S\in [0,1]^{b+1}$, $\textbf{1}(0 \leq L_{i}^H(v_S)\leq 1, \forall i \in S^{\prime})=1$. As, $\{v_S: 0 \leq L_{2i}^H(v_S)\leq 1, \forall i \in S^{-}\} \subset \{v_S: 0 \leq L_{i}^H(v_S)\leq 1, \forall i \in S^{\prime}\}$, for these chosen values of $v_S$, we have   $\textbf{1}(0 \leq L_{2i}^H(v_S)\leq 1, \forall 2i \in S^{-})=1$. Therefore, with this choice of $v_S\in [0,1]^{b+1}$,$$yL_{2i}^H(v_S)\leq y< 1 \implies F(yL_{2i}^H(v_S))\leq y.$$
 That is, the integral $\int_0^1\int_0^1 \cdots \int_0^1 \textbf{1}\big(F(y L_{2i}^H(v_{S}))\leq y - \lfloor y \rfloor) \ \forall 2i\in S^{-}\big) \ dv_S$ is positive. 

Hence the proof of part (i) is complete.

To prove part (ii), observe that 
\begin{align*}
& \xi_{\pi}(i)= \xi_{\pi}(j) \ \ \text{ if and only if }\ \ t_i=t_j \ (\text{mod } n) \ \ \text{ and }\\
&  \sgn(\pi(i)-\pi(i-1))= \sgn(\pi(j)-\pi(j-1)) \ \ \text{ if } i \ \text{ and } j \ \text{ are of same parity, or}\\
& \sgn(\pi(i)-\pi(i-1))= \sgn(\pi(j-1)-\pi(j)) \ \ \text{ if } i \ \text{ and } j \ \text{ are of opposite parity}.
\end{align*}

As $\Pi_{S_{R}}(\boldsymbol{\omega})\subset \Pi_{S_{R^{(s)}}}(\boldsymbol{\omega})$. If $\boldsymbol {\omega}$ is a word with $b$ distinct letters but not symmetric, by part (i), $\frac{1}{p^{r+1}b^{n-r}}\big|\Pi_{S_{R}}(\boldsymbol{\omega})\big| \rightarrow 0$ as $n \rightarrow \infty$.
 
Next let $\boldsymbol {\omega} \in S_b(2k)$ with $(r+1)$ even generating vertices. Clearly for $|\mathcal{E}_{\boldsymbol{\omega}}|=r+1$ and $|\mathcal{O}_{\boldsymbol{\omega}}|=b-r$. Now suppose, 

Recall the sets $\mathcal{E}_{\boldsymbol{\omega}}, \mathcal{O}_{\boldsymbol{\omega}},C_{i_j}^e, C_{i_j}^{o}$ from \eqref{def-even gen}, \eqref{def-odd gen} and \eqref{def-odd-even}. Similarly we can define the functions $f_n^H$ and $f^H$. Thus we can conclude 
\begin{align}\label{rc-integral}
 \displaystyle \lim_{n \rightarrow \infty}\frac{1}{p^{r+1}n^{b-r}}\big|\Pi_{S_{R}}(\boldsymbol{\omega})\big|=& \lfloor y \rfloor^{k-(r+1)} + \displaystyle \sum_{\phi \neq S_0\subset S^{-}}\lfloor y \rfloor^{|S^{-}- S_0|} \nonumber\\ 
 &\int_0^1\int_0^1 \cdots \int_0^1 \textbf{1}\big(F(y L_{2i}^H(v_{S}))\leq y - \lfloor y \rfloor) \ \forall 2i\in S_0\big)f^H(v_S) \ dv_S
 \end{align}
 where $d{v_S}= \prod_{j=0}^b dv_{i_j}$ is the $(b+1)-$dimensional Lebesgue integral on $[0,1]^{b+1}$ and $S^{-}$ is as in \eqref{def-S-}.

This completes the proof of part (ii).
 \end{proof}

\begin{proof}[\textbf{Proof of Lemma \ref{lem:c-sc}}]
 
 \noindent (i) First suppose $\boldsymbol {\omega} \in \mathcal{P}(2k) \setminus E_b(2k)$. Then from \eqref{SA-S} and Lemma 5.2 in \citep{bose2021some}, and using the fact that $p/n \rightarrow y>0$ as $n\rightarrow \infty$, it is easy to see that $$\displaystyle \lim_{n \rightarrow \infty}\frac{1}{p^{r+1}n^{b-r}} | \Pi_{S_{C^{(s)}}}(\boldsymbol {\omega})|=0\ \ \text{if}\ \boldsymbol {\omega}.$$  
Now suppose $\boldsymbol{\omega}$ is an even word of length $2k$ with $b$ distinct letter and $(r+1)$ even generating vertices. Suppose $i_1,i_2,\ldots,i_b$ are the positions where new letters made their first appearances. First we fix the generating vertices $\pi(i_j), 0 \leq j\leq b$ where $\pi(i_0)=\pi(0)$. Let
\begin{align*} 
s_i= \pi(i)- \pi(i-1)  \  \  \text{for }1 \leq i \leq 2k.
\end{align*}
Clearly, from \eqref{SA-link}, $\boldsymbol {\omega}[i]=\boldsymbol {\omega}[j]$ 
if and only if $\xi_{\pi}(i)=\xi_{\pi}(j)$. That is, 
$|s_i| =  |s_j| \ (\text{mod } n)$, that is, $s_i  =  s_j \ (\text{mod } n)$ or   $s_i  =- s_j \ (\text{mod } n) .$
  Thus for every $i,\ 1 \leq i \leq 2k$,
\begin{equation}\label{SCplus_minus}
s_i = s_{i_j}\ (\text{mod } n) \mbox{ or }	s_i =  - s_{i_j}\ (\text{mod } n)
\end{equation} 
for some $j \in \{1,2,\ldots,b\}$. Thus, we have 
 \begin{align}\label{SC-relation}
 \pi(i)  &=  \pm(\pi(i_j)-\pi(i_j-1))+\pi(i-1)=\pm s_{i_j}+ \pi(i-1)\ (\text{mod } n) \ \ \text{for some }j.
  \end{align}
First we fix the generating vertices $\pi(i_j),\ j=0,1,2,\ldots,b$. Let $S=\{\pi(i_j):0 \leq  j \leq b\}  \ \mbox{ and }\ S^{\prime}=\{i : \pi(i) \notin S\}$. Having chosen a particular sign in \eqref{SC-relation}, for every $i \in S^{\prime}$, there is some $i_j, 1 \leq j \leq b$,
\begin{align}\label{SC-1}
\pi(i)&=  (s_{i_j}+ \pi(i-1))\ (\text{mod } n) \  \text{ or } \pi(i)=  (-s_{i_j}+ \pi(i-1))\ (\text{mod } n) \nonumber\\
i.e.,& \pi(i)= \displaystyle \sum_{j<i} \beta_{ij}\pi(i_j)\ (\text{mod } n)\ \ \text{ for some } \{\beta_{ij}\} \subset \mathbb{Z}
\end{align}
Thus for every $i\in S^{\prime}\setminus \{2k\} $, there exists a unique integer $m_{i,n}$ such that 
\begin{align}\label{SC-2}
 1 \leq \displaystyle \sum_{j<i} \beta_{ij}\pi(i_j)+ m_{i,n}\leq n.
\end{align}
As we have already fixed the generating vertices, from \eqref{SC-1} and \eqref{SC-2} it follows that there is a unique choice for all $\pi(2i-1),1 \leq i \leq k$ such that $2i-1\in S^{\prime}$. For all $2i\in S^{\prime}, 1 \leq i< k$, we can have $\lfloor y_n \rfloor$ choices as $1 \leq \pi(2i) \leq p$ and $y_n=p/n$. Moreover, there is an additional choice if 
\begin{align}\label{SC-3}
\displaystyle \sum_{j<2i} \beta_{2ij}\pi(i_j)+ m_{2i,n}\leq p - \lfloor y_n \rfloor n.
\end{align}
Next let
  $$
  v_{2i}=  \frac{\pi(i)}{p},\ \ v_{2i-1}=  \frac{\pi(i)}{n} \ \ \text{ for }0 \le i \leq k,\ \\
 \ \  L(a)= \max\{m \in \mathbb{Z}\}, \ F(a)= a- L(a). 
   $$
 Also let 
 \begin{equation}\label{def-S-sc}
 S^{-}= \{2i: 2i\notin S^{\prime} \ \text{ and \eqref{SC-3} holds true}\}
 \end{equation}

 Now observe that from \eqref{SC-2} and \eqref{SC-3} it follows that for every $i \in S^{-}$,
 \begin{align}
 F(y_n L_{2i,n}^T(v_{S}))\leq y_n - \lfloor y_n \rfloor
 \end{align}
where $L_{2i,n}^T$ is the linear combination defined in \eqref{linear-toe}. Now, this linear combinations vary depending on the sign chosen for each $s_i$. As we know for each block of an even word, the number of positive and negative signs in the relations among the $s_i$'s (i.e., the equation like \eqref{SC-relation}) are equal. Therefore there are $\displaystyle \prod_{i=1}^b{ {k_i-1} \choose {\frac{k_i}{2}}}= a_{\boldsymbol {\omega}}$ different linear combinations corresponding to each word $\boldsymbol{\omega}$, where $k_1,\ldots,k_b$ are the block sizes of $\boldsymbol {\omega}$.
 
From \eqref{SA-Pi(omega)} and the discussion above, it is easy to see that for a word $\boldsymbol{\omega}$ of length $2k$, 
\begin{align*}
& \frac{1}{p^{r+1}n^{b-r}}\big|\Pi_{S_{C^{s}}}(\boldsymbol{\omega})\big|=\\
& a_{\boldsymbol{\omega}}\bigg[\lfloor y_n \rfloor^{k-(r+1)} + \displaystyle \sum_{\phi \neq S_0\subset S^{-}}\lfloor y_n \rfloor^{|S^{-}- S_0|}\big|\big\{v_S: F(y_n L_{2i,n}^T(v_{S}))\leq y_n - \lfloor y_n \rfloor, \textbf{1}(2i \in S_0)\big\}\big|\bigg].
 \end{align*} 
 Therefore as $n \rightarrow \infty$,
 \begin{align}\label{SC-integral}
 \displaystyle \lim_{n \rightarrow \infty}\frac{1}{p^{r+1}n^{b-r}}\big|\Pi_{S_{C^{(s)}}}(\boldsymbol{\omega})\big|=& a_{\boldsymbol{\omega}}\bigg[\lfloor y \rfloor^{k-(r+1)} + \displaystyle \sum_{\phi \neq S_0\subset S^{-}}\lfloor y \rfloor^{|S^{-}- S_0|} \nonumber\\ 
 &\int_0^1\int_0^1 \cdots \int_0^1 \textbf{1}\big(F(y L_{2i}^T(v_{S}))\leq y - \lfloor y \rfloor) \ \forall 2i\in S_0\big) \ dv_S \bigg].
 \end{align}
 
When $y\geq 1$, the rhs of \eqref{SC-integral} is obviously positive. We next show that, $\int_0^1\int_0^1 \cdots \int_0^1 \textbf{1}\big(F(y L_{2i}^T(v_{S}))\leq y - \lfloor y \rfloor) \ \forall 2i\in S^{-}\big) \ dv_S \ >0$ when $y < 1$. 

First note that as $y<1, \lfloor y \rfloor=0$. Now note that we had previously concluded in the proof of part (i) of Lemma \ref{lem:toe-S} that for certain values of $v_S\in [0,1]^{b+1}$, $\textbf{1}(0 \leq L_{i}^T(v_S)\leq 1, \forall i \in S^{\prime})=1$. As, $\{v_S: 0 \leq L_{2i}^T(v_S)\leq 1, \forall i \in S^{-}\} \subset \{v_S: 0 \leq L_{i}^T(v_S)\leq 1, \forall i \in S^{\prime}\}$, for these chosen values of $v_S$, we have   $\textbf{1}(0 \leq L_{2i}^T(v_S)\leq 1, \forall 2i \in S^{-})=1$. Therefore, with this choice of $v_S\in [0,1]^{b+1}$,$$yL_{2i}^T(v_S)\leq y< 1 \implies F(yL_{2i}^T(v_S))\leq y.$$
 That is, the integral $\int_0^1\int_0^1 \cdots \int_0^1 \textbf{1}\big(F(y L_{2i}^T(v_{S}))\leq y - \lfloor y \rfloor) \ \forall 2i\in S^{-}\big) \ dv_S$ is positive. 

Hence the proof of part (i) is complete. 

\noindent (ii) Using the same arguments as in the proof of Lemma \ref{lem:toe}, it follows that 
$$\displaystyle \lim_{n \rightarrow \infty}\frac{1}{p^{r+1}n^{b-r}} | \Pi_{S_C}(\boldsymbol {\omega})|=0\ \ \text{if}\ \boldsymbol {\omega}\  \text{is not symmetric}.$$

Next, suppose $\boldsymbol{\omega}$ is a symmetric word with $b$ distinct letters and $(r+1)$ even generating vertices. Also suppose the letters make their first appearances at $i_1,i_2, \ldots,i_b$ positions in $\boldsymbol{\omega}$.

Using similar arguments as \eqref{T-pi-relation} in Lemma \ref{lem:toe} and \eqref{SC-3} in the proof of part (i), we have for $2i-1 \in S^{\prime}$, there is a unique choice of $\pi(2i-1)$, once the generating vertices have been chosen. For all $2i\in S^{\prime}, 1 \leq i< k$, we can have $\lfloor y_n \rfloor$ choices for $\pi(2i)$ as $1 \leq \pi(2i) \leq p$ and $y_n=p/n$. Moreover, there is an additional choice for if 
\begin{align}\label{C-1}
\displaystyle \sum_{j<2i} \beta_{2ij}\pi(i_j)+ m_{2i,n}\leq p - \lfloor y_n \rfloor n.
\end{align}
 Using the same notations and arguments as in part (i) we have
 \begin{align}
 F(y_n L_{2i,n}^T(v_{S}))\leq y_n - \lfloor y_n \rfloor
 \end{align}
where $L_{2i,n}^T$ is a particular set of linear combinations defined whose sign has been chosen as in \eqref{T-integral-gen}, (see proof of Lemma \ref{lem:toe}). Thus we have that 
\begin{align}\label{C-integral}
 \displaystyle \lim_{n \rightarrow \infty}\frac{1}{p^{r+1}n^{b-r}}\big|\Pi_{S_{C}}(\boldsymbol{\omega})\big|=& [\lfloor y \rfloor^{k-(r+1)} + \displaystyle \sum_{\phi \neq S_0\subset S^{-}}\lfloor y \rfloor^{|S^{-}- S_0|} \nonumber\\ 
 &\int_0^1\int_0^1 \cdots \int_0^1 \textbf{1}\big(F(y L_{2i}^T(v_{S}))\leq y - \lfloor y \rfloor) \ \forall 2i\in S_0\big) \ dv_S 
 \end{align}
As in the proof of part(i), we can show that when $y<1$,  $$\int_0^1\int_0^1 \cdots \int_0^1 \textbf{1}\big(F(y L_{2i}^T(v_{S}))\leq y - \lfloor y \rfloor) \ \forall 2i\in S^{-}\big) \ dv_S \ >0.$$ 
  
  This completes the prof of part (ii).
 \end{proof}

\providecommand{\bysame}{\leavevmode\hbox to3em{\hrulefill}\thinspace}
\providecommand{\MR}{\relax\ifhmode\unskip\space\fi MR }
\providecommand{\MRhref}[2]{%
  \href{http://www.ams.org/mathscinet-getitem?mr=#1}{#2}
}
\providecommand{\href}[2]{#2}

\bibliography{mybibfilefinal}

\begin{figure}[htp]
 \begin{subfigure}{.5\textwidth}
  \centering
  \includegraphics[width=.7\linewidth]{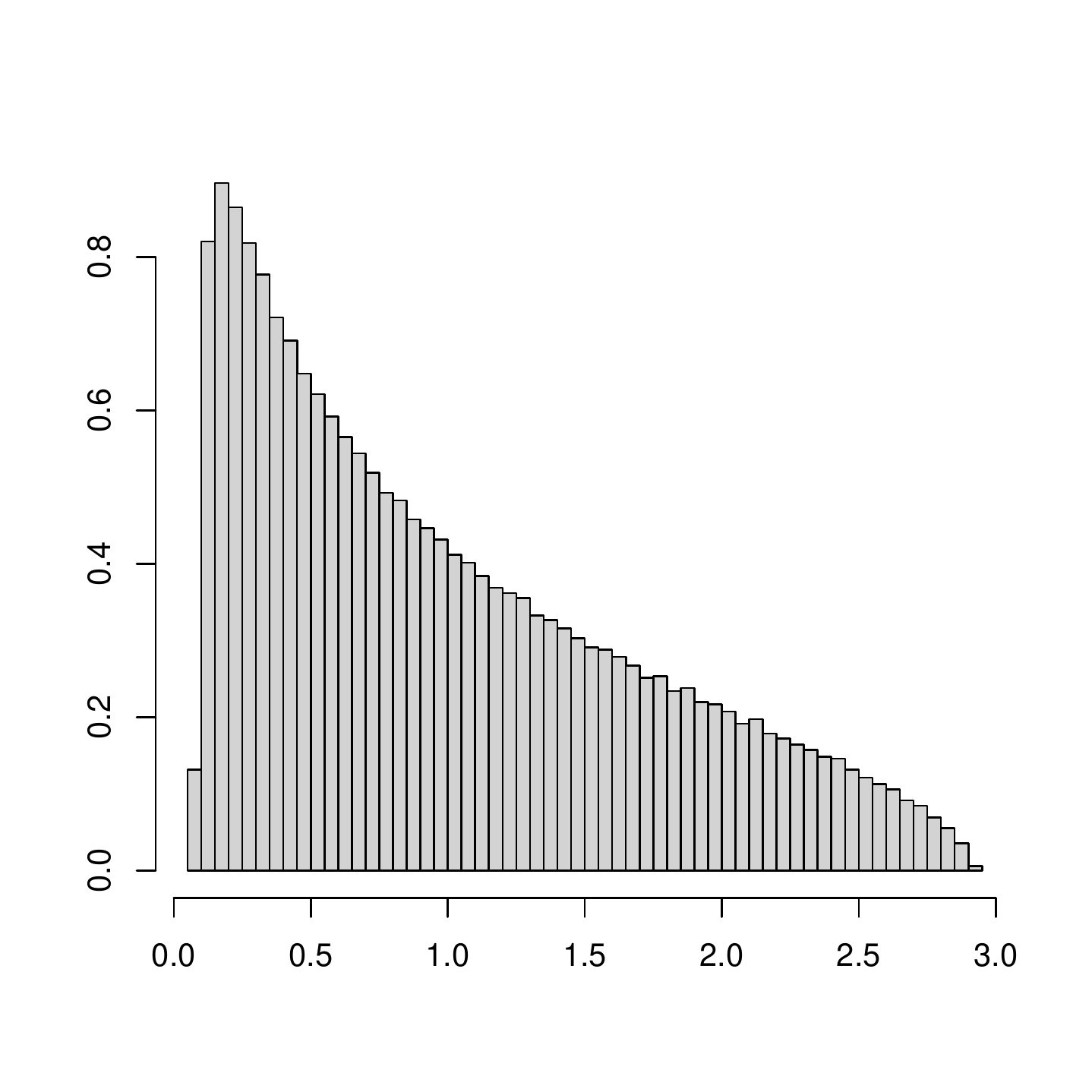}  
  \caption{Input is i.i.d $x_{ij}\sim N(0,1)/\sqrt{n}$ for every $n$.}
\end{subfigure}
\begin{subfigure}{.5\textwidth}
  \centering
  \includegraphics[width=.7\linewidth]{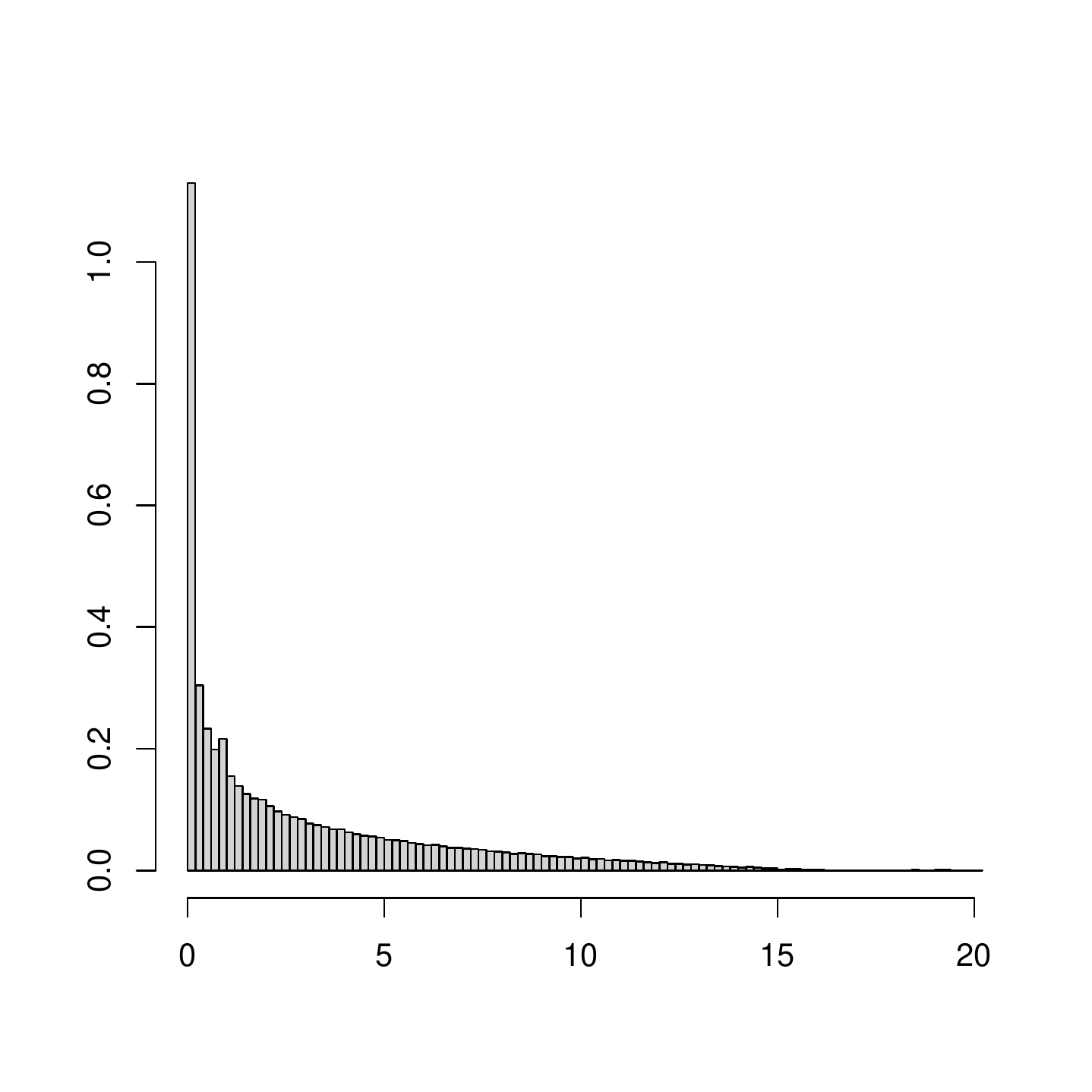}  
  \caption{Input is i.i.d $x_{ij} \sim Ber(3/n)$ for every $n$.}
\end{subfigure}
\begin{subfigure}{.5\textwidth}
 \centering
  \includegraphics[width=.7\linewidth]{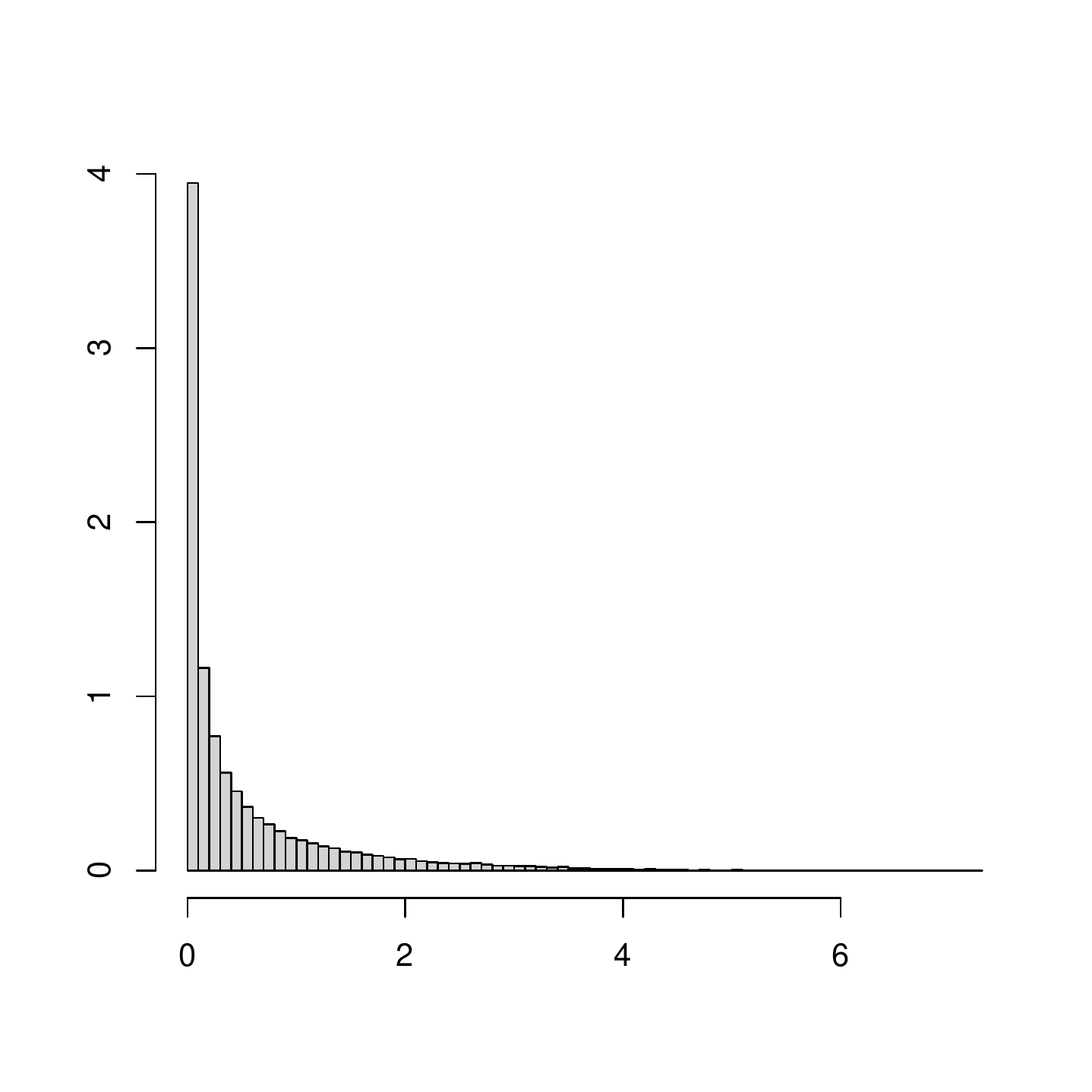}  
  \caption{Input is i.i.d $x_{ij} \sim \frac{(i+j)^2}{2n^2}Ber(3/n)$ for every $n$.}
\end{subfigure}
\begin{subfigure}{.5\textwidth}
 \centering
  \includegraphics[width=.7\linewidth]{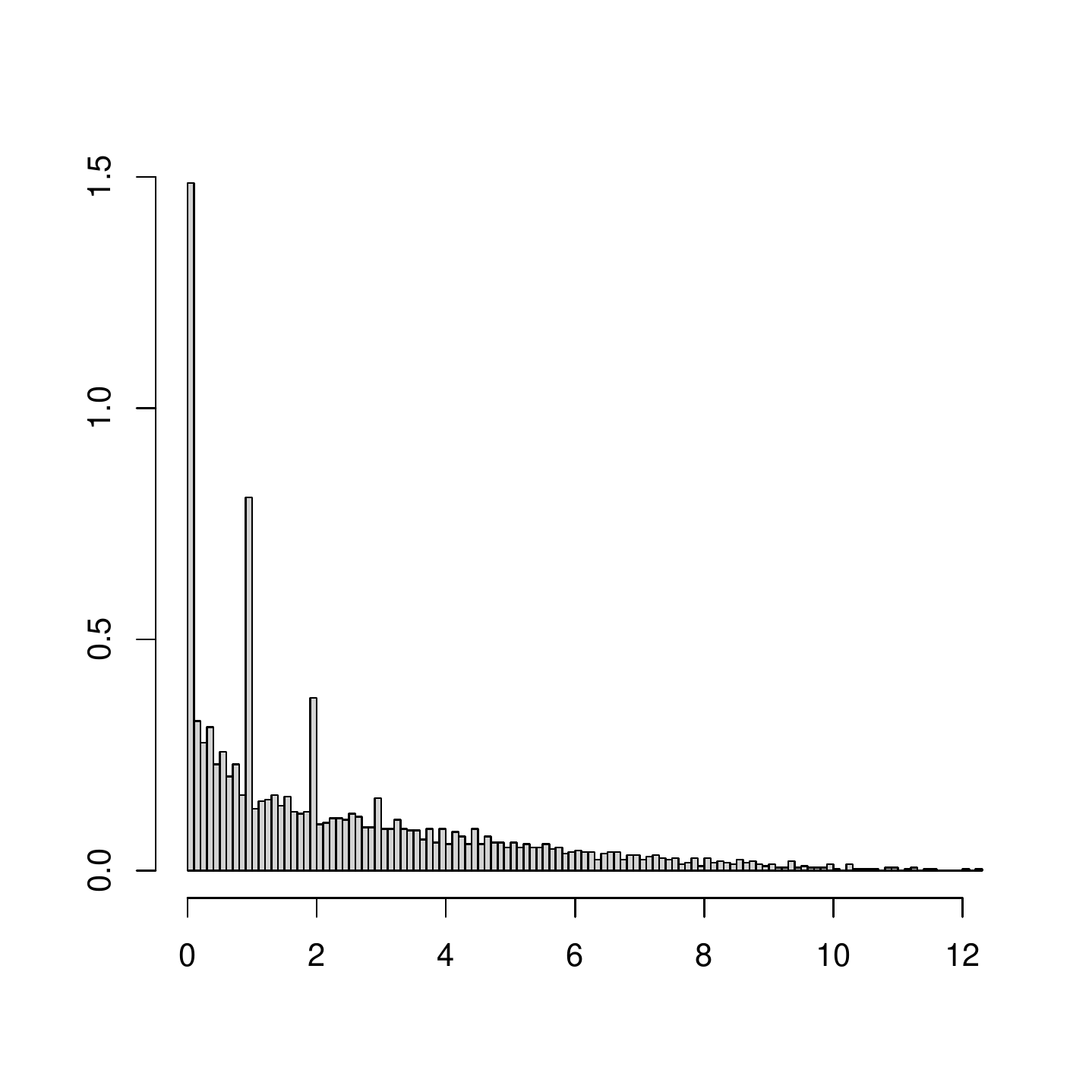}  
  \caption{Input is i.i.d $x_{ij} \sim Ber(3/n), i \leq j$ and $0$ otherwise. }
\end{subfigure}
\caption{Histogram of the eigenvalues of $S$ for $p=1000,n=2000$, 30 replications. }
\label{fig:S1}
\end{figure}

 \begin{figure}[htp]
\begin{subfigure}{.5\textwidth}
  \centering
  \includegraphics[width=.7\linewidth]{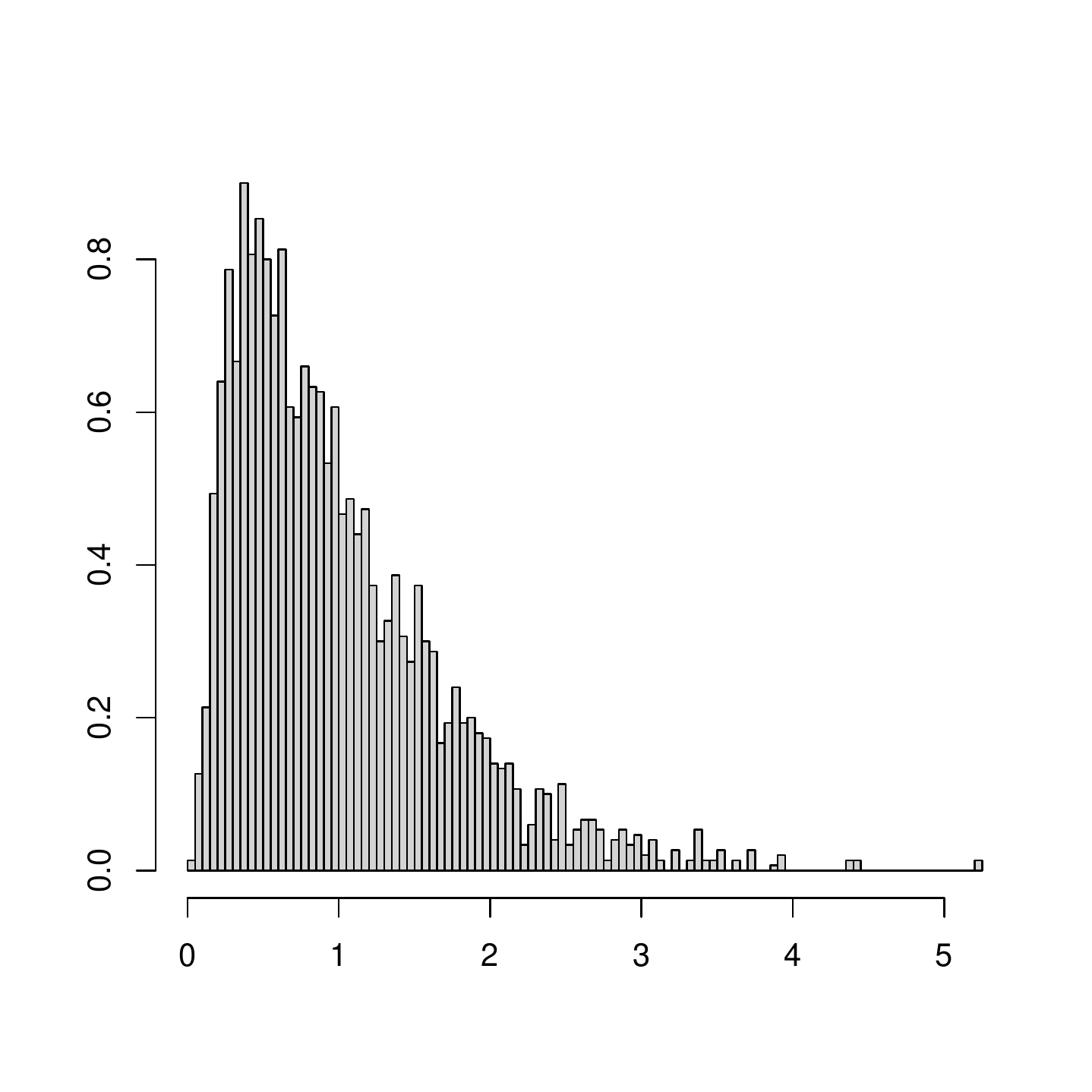}  
\end{subfigure}
\begin{subfigure}{.5\textwidth}
 \centering
  \includegraphics[width=.7\linewidth]{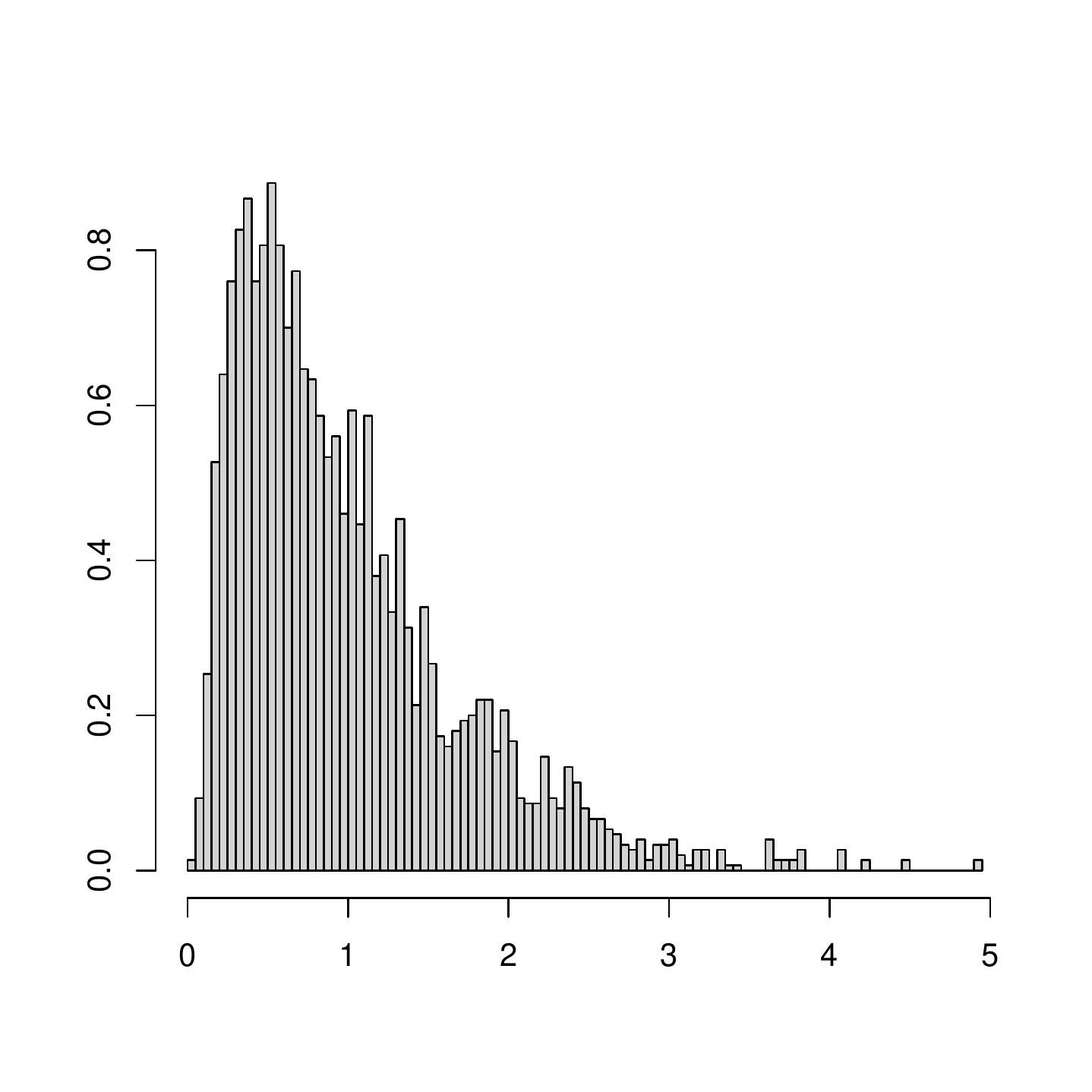}  
\end{subfigure}
\\
\begin{subfigure}{.5\textwidth}
  \centering
  \includegraphics[width=.7\linewidth]{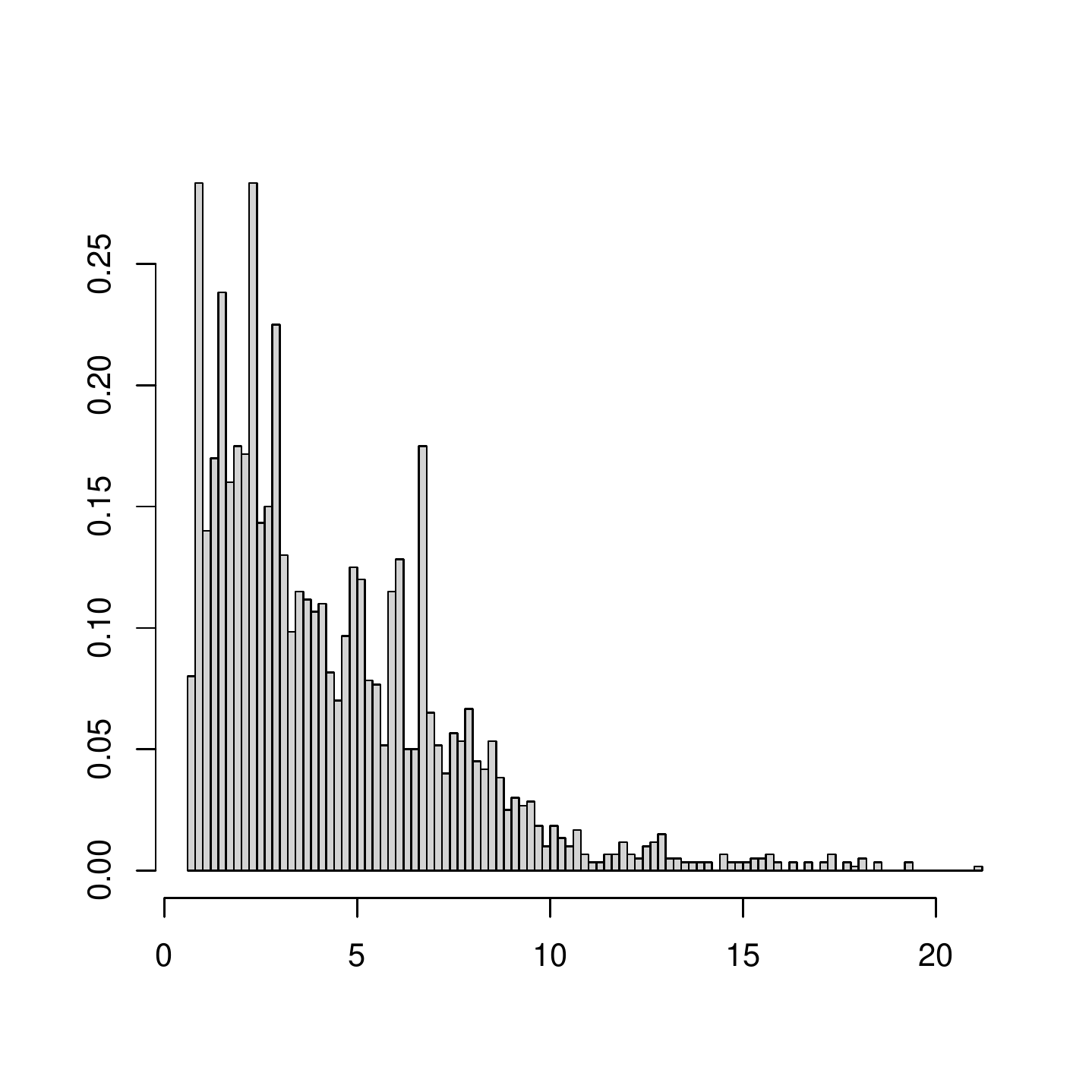}  
\end{subfigure}
\begin{subfigure}{.5\textwidth}
  \centering
  \includegraphics[width=.7\linewidth]{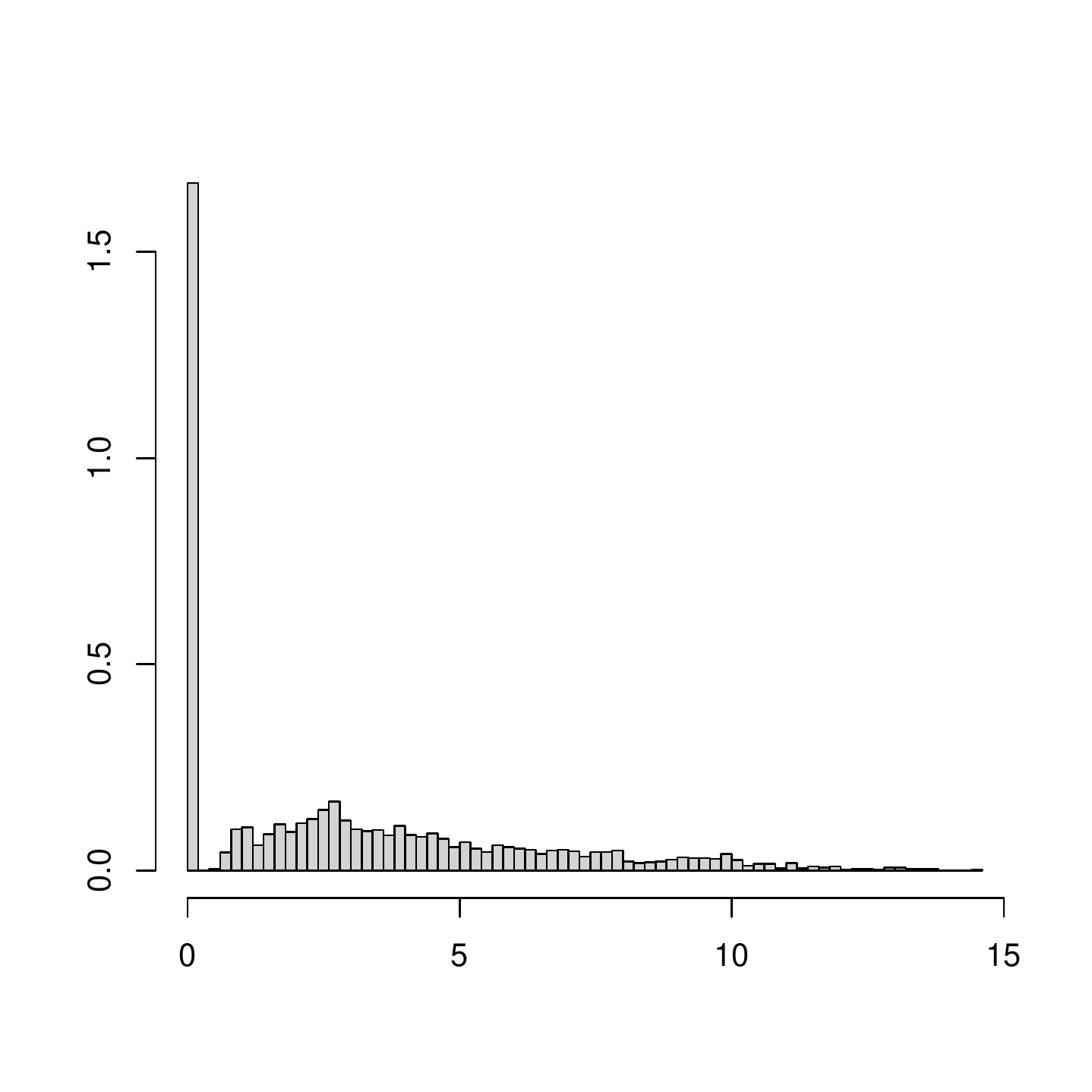}  
\end{subfigure}

\caption{Histogram of the eigenvalues of $S_{R^{(s)}}$ with entries i.i.d. $N(0,1)/\sqrt{n}$ (top row) and  i.i.d. Ber$(3/n)$ for every $n$ (bottom row), $p=1000,n=2000$, 2 replications.}
\label{fig:SA-1}
\end{figure}
 \begin{figure}[htp]
 \begin{subfigure}{.5\textwidth}
  \centering
  \includegraphics[width=.7\linewidth]{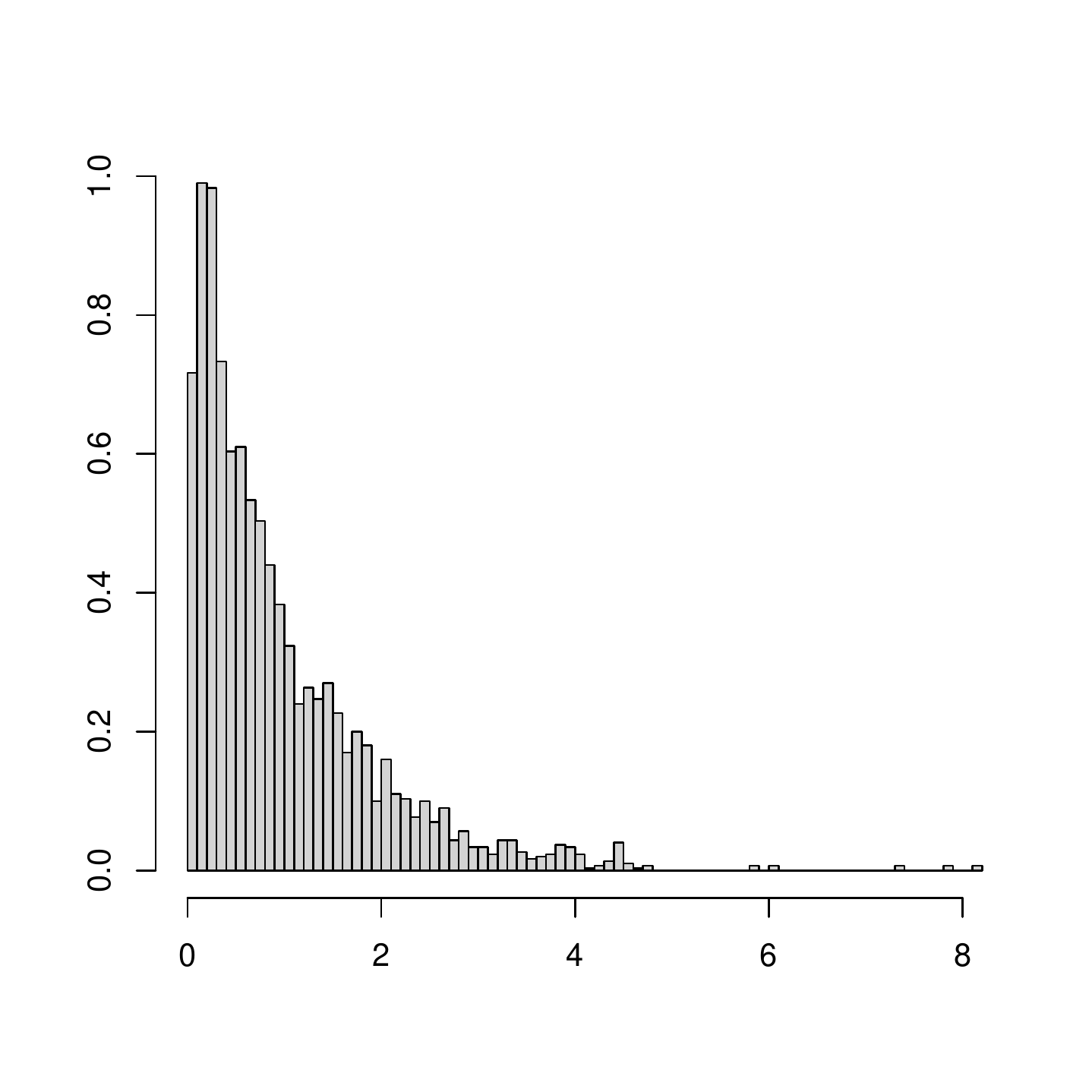}  
\end{subfigure}
\begin{subfigure}{.5\textwidth}
 \centering
  \includegraphics[width=.7\linewidth]{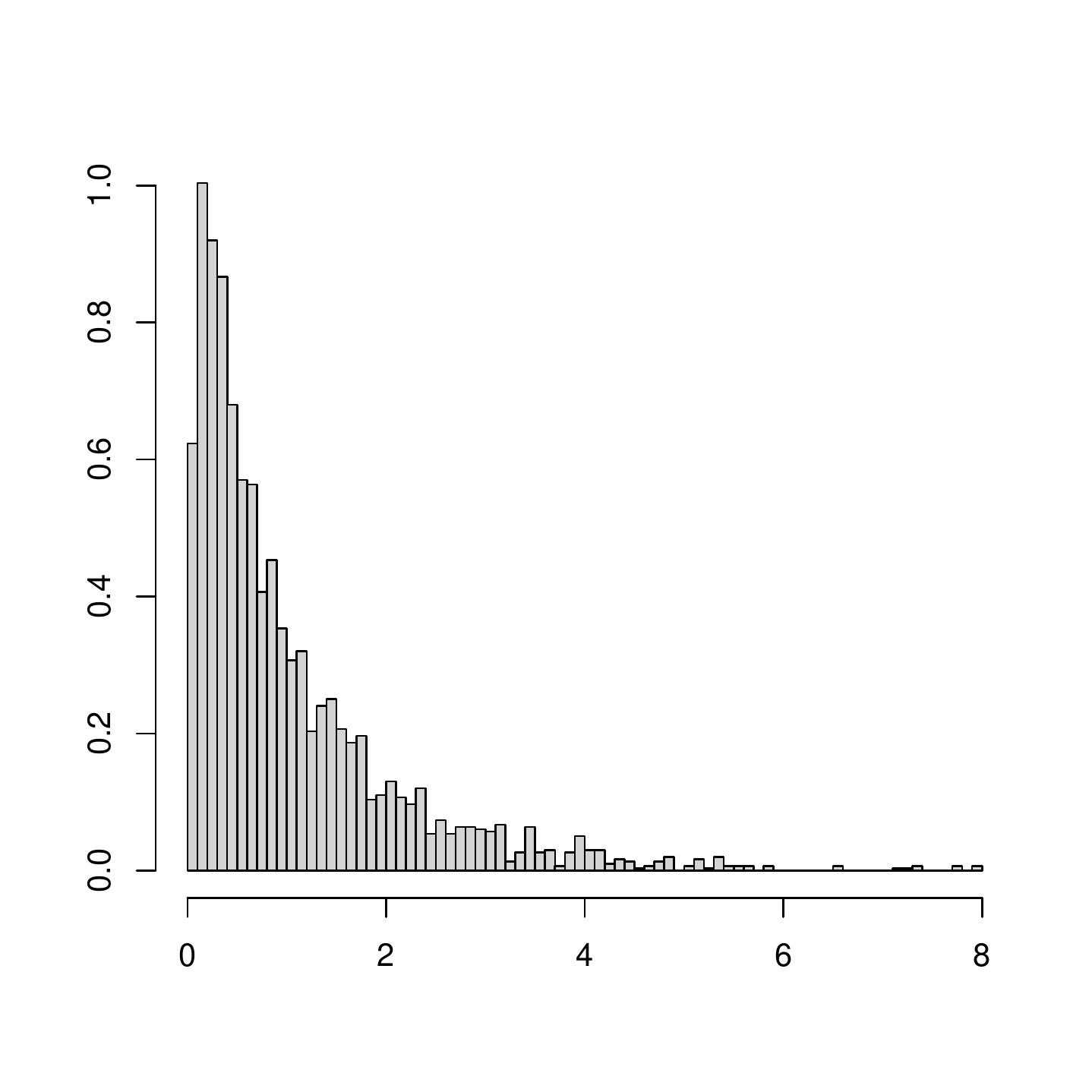}  
\end{subfigure}
\\
\begin{subfigure}{.5\textwidth}
  \centering
  \includegraphics[width=.7\linewidth]{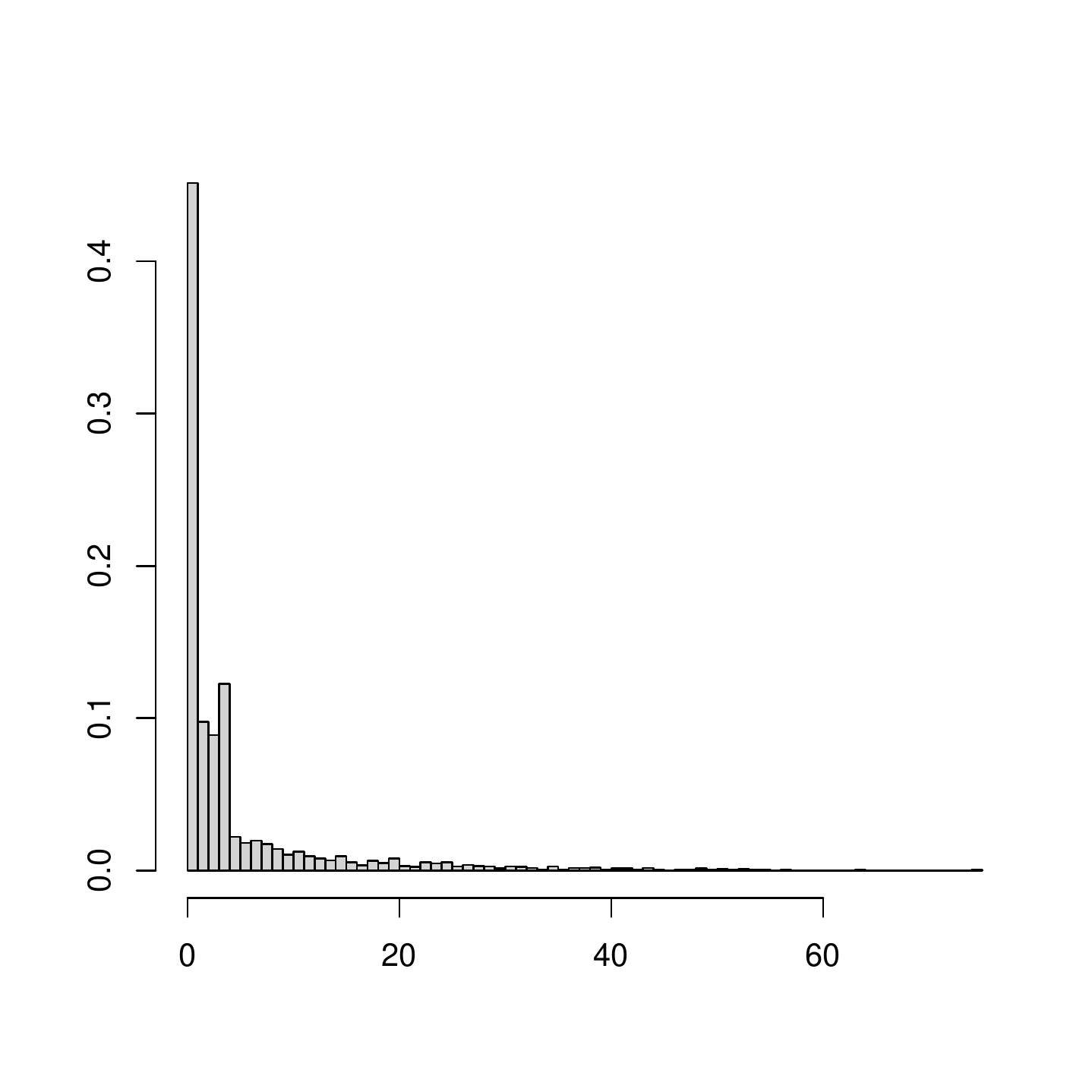}  
\end{subfigure}
\begin{subfigure}{.5\textwidth}
 \centering
  \includegraphics[width=.7\linewidth]{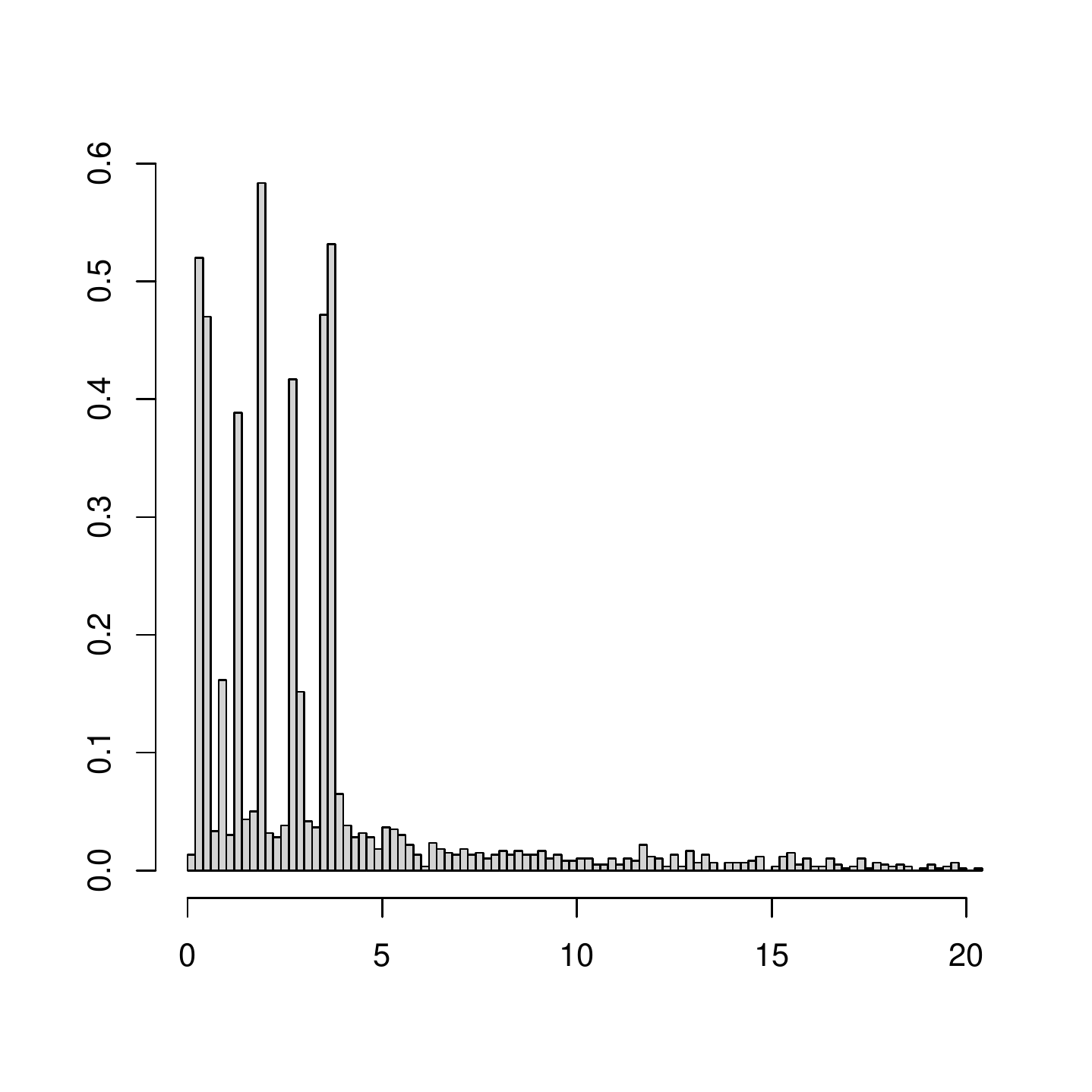}  
\end{subfigure}
\caption{Histogram of the eigenvalues of $S_{C^{(s)}}$ with entries i.i.d. $N(0,1)/\sqrt{n}$(top row) and  i.i.d. Ber$(3/n)$ for every $n$ (bottom row), $p=1000,n=2000$, 2 replications.}
\label{fig:SA-2}
\end{figure}

\begin{figure}[htp]
 \begin{subfigure}{.5\textwidth}
  \centering
  \includegraphics[width=.7\linewidth]{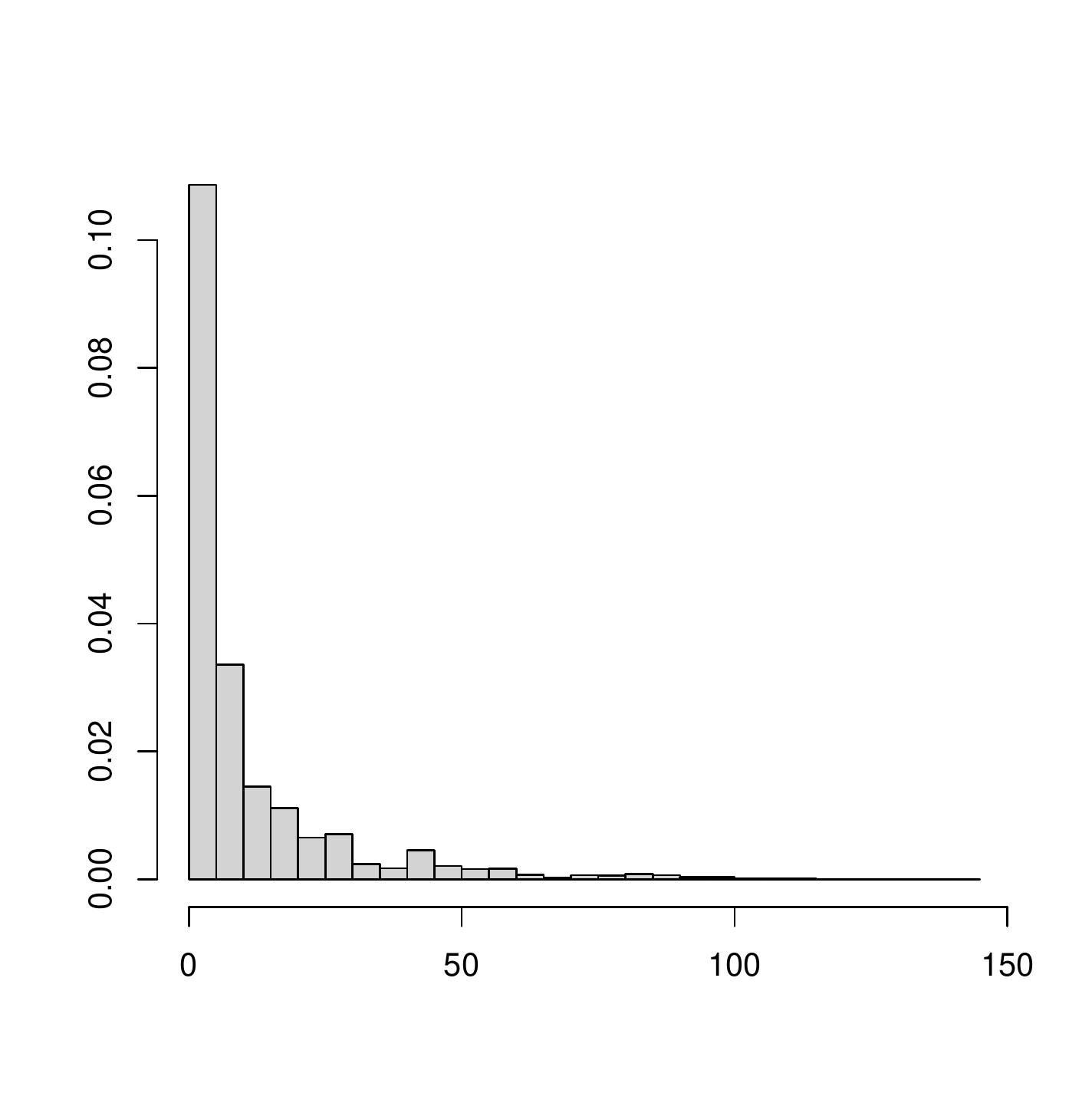}  
  \caption{$S_{T^{(s)}}$}
\end{subfigure}
\begin{subfigure}{.5\textwidth}
  \centering
  \includegraphics[width=.7\linewidth]{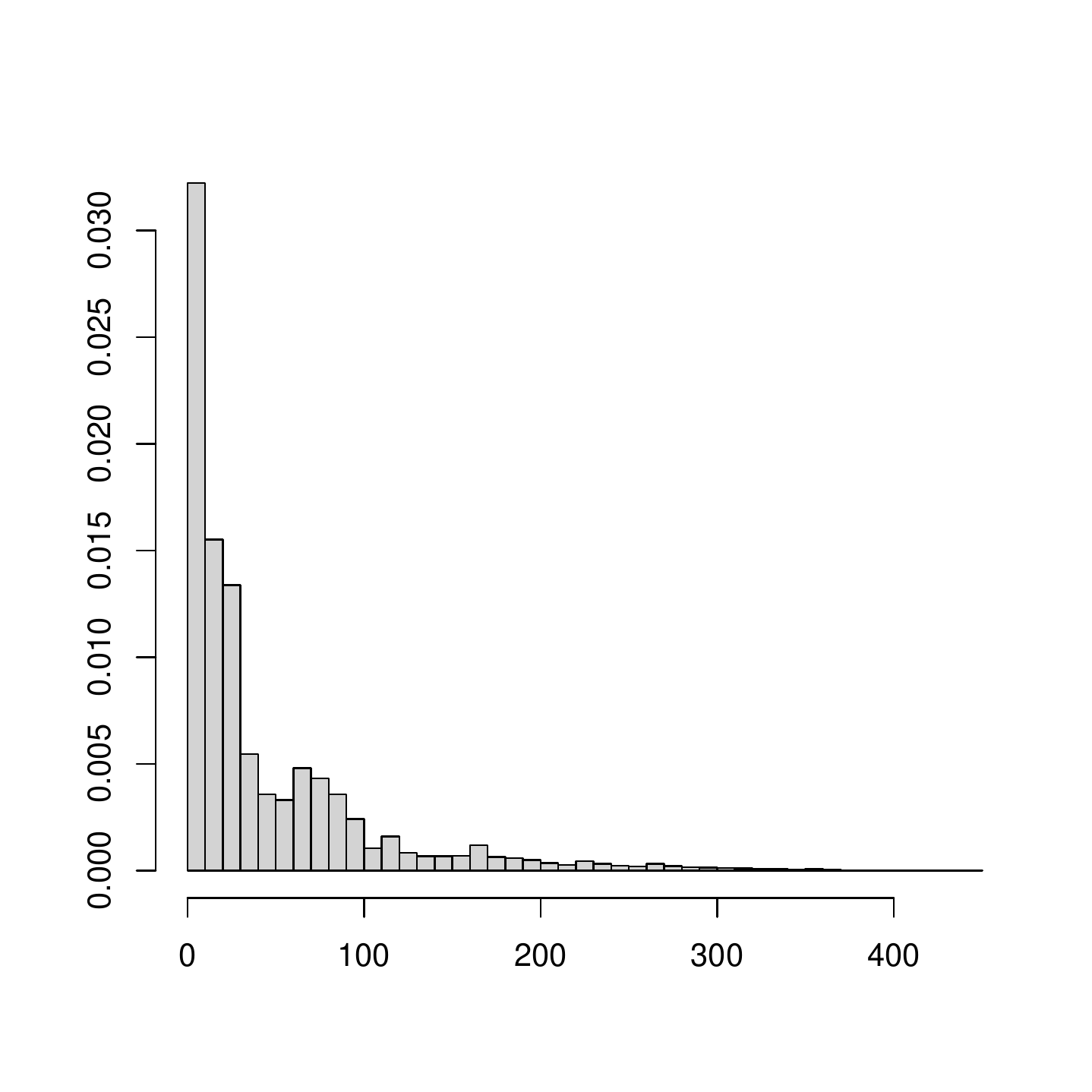}  
  \caption{$S_{H^{(s)}}$}
\end{subfigure}
\begin{subfigure}{.5\textwidth}
 \centering
  \includegraphics[width=.7\linewidth]{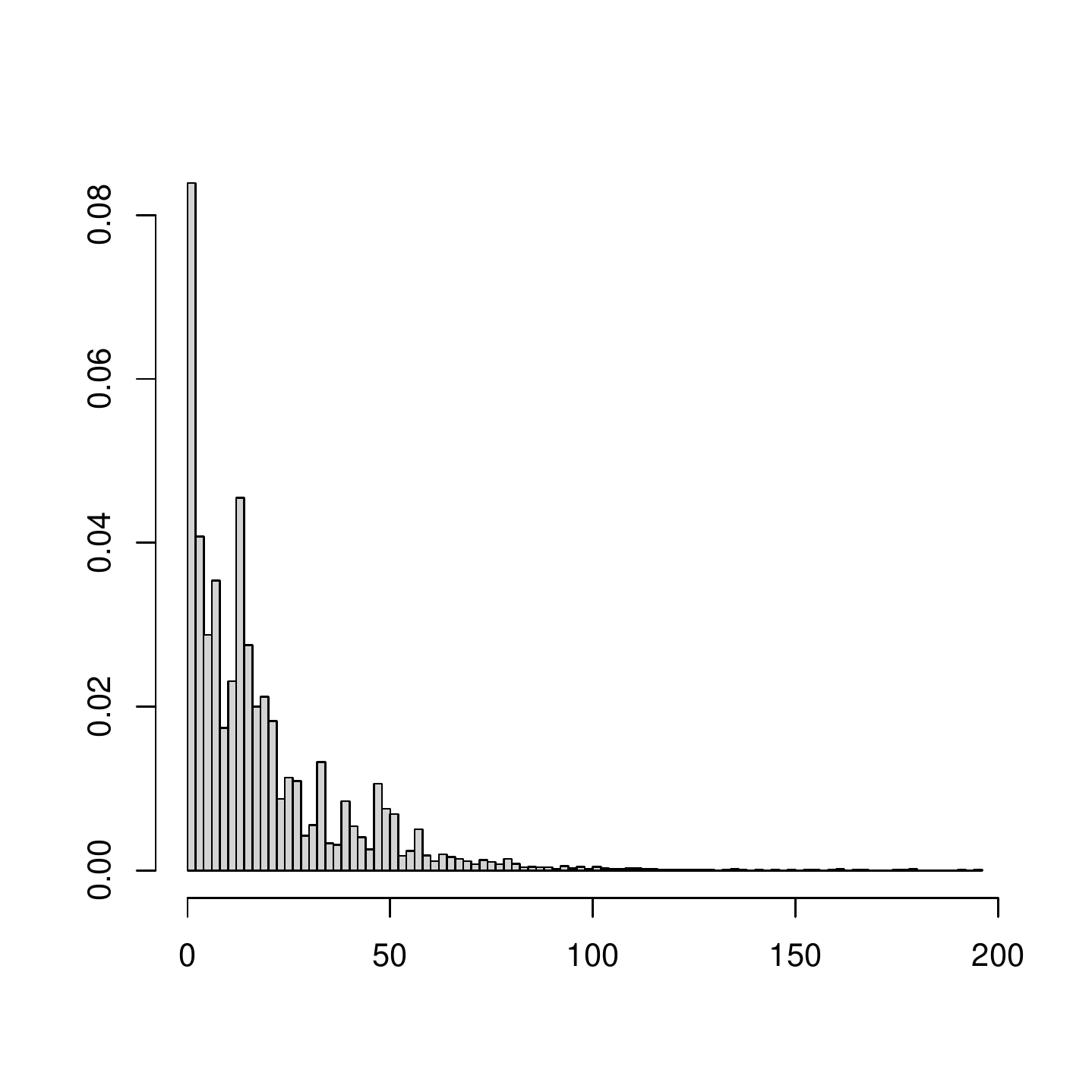}  
  \caption{$S_{R^{(s)}}$}
\end{subfigure}
\begin{subfigure}{.5\textwidth}
 \centering
  \includegraphics[width=.7\linewidth]{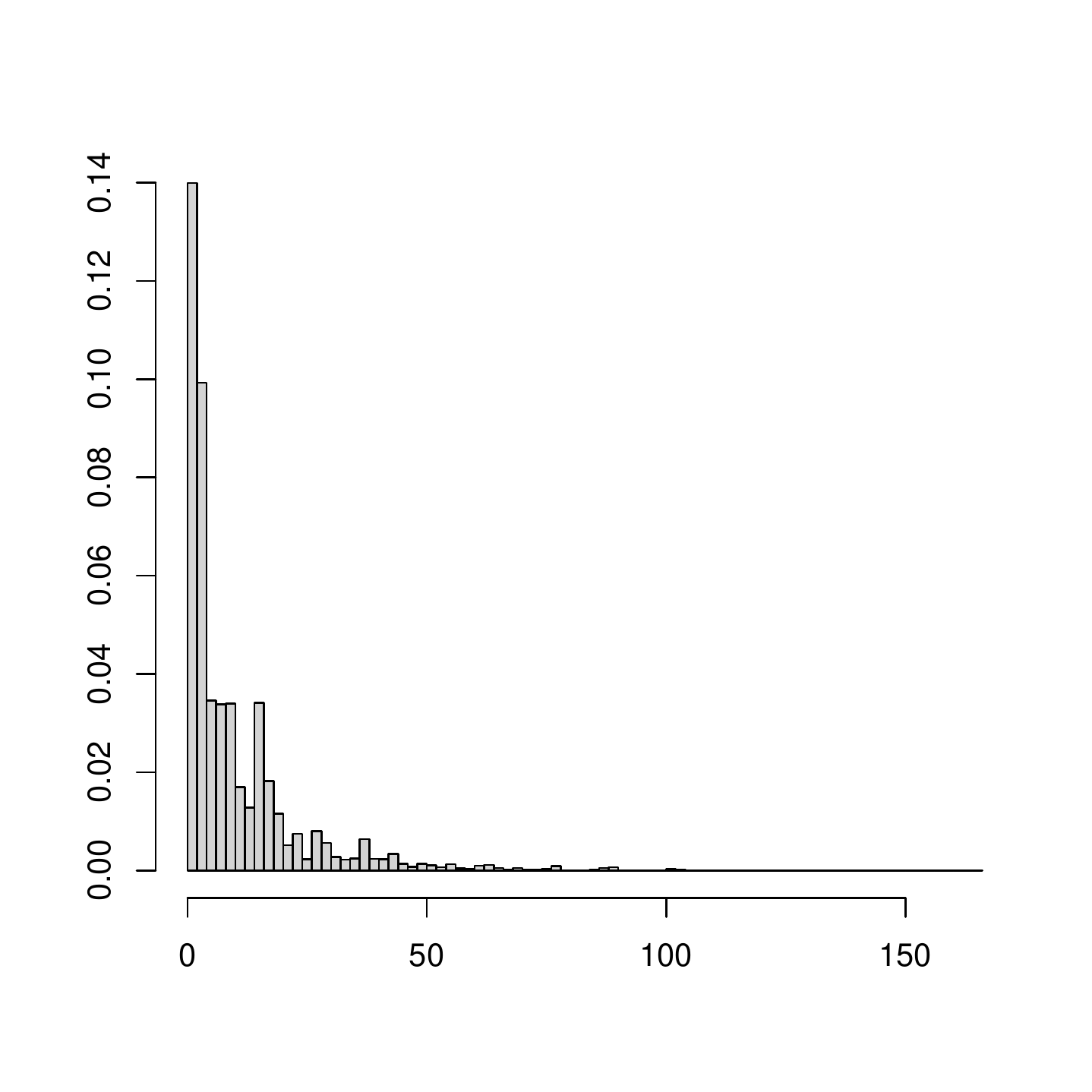}  
  \caption{$S_{C^{(s)}}$ }
\end{subfigure}
\caption{Histogram of the eigenvalues of $S_A$ for $p=1000,n=2000$, 30 replications with variance profile $\sigma(x)= x^2+4x$ and $x_{ij,n}\sim Ber(3/n)$. }
\label{fig:S3}
\end{figure}

\end{document}